\documentclass[11pt]{proc-l}
\usepackage{mathrsfs}
\usepackage{euscript}
 \usepackage{url}
\usepackage{setspace}
\usepackage{amssymb}
\usepackage[OT2,T1]{fontenc}
\usepackage{helvet}
\DeclareSymbolFont{cyrletters}{OT2}{wncyr}{m}{n}
\DeclareMathSymbol{\E}{\mathalpha}{cyrletters}{"03}
\DeclareMathSymbol{\Zhe}{\mathalpha}{cyrletters}{"11}
\DeclareMathSymbol{\De}{\mathalpha}{cyrletters}{"44}
\usepackage{fullpage,epsf,graphicx,url} 
\DeclareMathSymbol{\E}{\mathalpha}{cyrletters}{"03}
\usepackage[linkcolor=green!40!blue,colorlinks=green!40!blue]{hyperref} 
\hypersetup{colorlinks=true, citecolor=NavyBlue!60!RoyalBlue, filecolor=Maroon, linkcolor=black!65!purple, urlcolor=black!65!purple}
\usepackage{adjustbox}
\usepackage{mathrsfs}
\usepackage{amsthm}
\usepackage{adjustbox}

\usepackage{amsfonts}
\usepackage{CJKutf8}
\usepackage{pb-diagram}

\usepackage{tikz-cd}
\tikzset{%
    symbol/.style={%
        draw=none,
        every to/.append style={%
            edge node={node [sloped, allow upside down, auto=false]{$#1$}}}
    }
}
\newcommand\scalemath[2]{\scalebox{#1}{\mbox{\ensuremath{\displaystyle #2}}}}

\usepackage{enumitem}
\usepackage{etoolbox}
\usepackage[mathb,mathx]{mathabx}
\newtheorem{theorem}{Theorem}[section]
\newtheorem{lemma}[theorem]{Lemma}
\newtheorem{corollary}[theorem]{Corollary}
\theoremstyle{definition}
\newtheorem{definition}[theorem]{Definition}
\newtheorem{example}[theorem]{Example}

\theoremstyle{remark}
\newtheorem{remark}[theorem]{Remark}

\numberwithin{equation}{section}

\begin{document}
\title[Triole I]{Differential Calculus in Triole Algebras}

\author{Jacob Kryczka}
\address{Yanqi Lake Beijing Institute of Mathematical Sciences and Applications (BIMSA), West Road of Yanqi Lake 101408, Huairou District, Beijing, China}
\email{jkryczka@bimsa.cn}
\thanks{}
\maketitle

\begin{abstract}
 This work is the first in a series of papers that, among other things, extends the formalism proposed in \cite{Diole1}, wherein a new context for obtaining differential calculus in vector bundles was established. Here we provide a modest but interesting generalization of this formalism to include a class of vector bundles with additional inner structure provided by fiber metrics that we call \emph{triole algebras.} We discuss the basics of a triolic-linear algebra and study various functors of differential calculus over these algebraic objects. In doing so, we establish a conceptual framework for making sense of calculus on bundles that preserve a vector-valued fiber metric structure. 
\end{abstract}
\section*{Introduction}
Linear connections appear naturally in many areas of the mathematical sciences. Prominent examples are the canonical flat connection on a finite separable algebraic extension, and the Levi-Civita connection in the tangent bundle to a manifold. These facts and more are perhaps most easily seen from the perspective of differential calculus over graded commutative algebras \cite{Nes}. Moreover, there is a natural cohomology theory associated to such flat connections which arises from a de Rham like complex. These cohomologies are natural invariants and so one should wonder: what are these invariants and how do we go about computing them?

Suggestions towards answering this question come from theory of differential geometry where flat connection cohomology appears as the de Rham cohomology with twisted coefficients, discussed in §§\ref{sssec:Generalnotionofinnerstructure} and sequence (\ref{eqn:deRhamSequenceforConnections}) below. Moreover, this kind cohomology appears in some situations in physics, in particular in field theory $-$ a seemingly disjoint area from that which studies these algebraic extensions. 

Treated in a unified point of view, the language of connections should form a basis for both of these areas and more. Indeed, evidence in favour of such a unified language was provided within the formalism of \emph{diolic calculus} \cite{Diole1}. Our immediate hopes is how this language can be naturally generalized in a conceptual way, to include in the discussion those linear connections which preserve various types of additional structures of importance both in physics and mathematics, most notably in this work, \emph{fiber metrics} in vector bundles.

In this work we pay special attention to vector-valued forms, which themselves are ubiquitous in mathematics and play a key role in the study of higher co-dimensional geometries. For example, they occur naturally in the study of Riemannian submanifolds as the second fundamental form. In this situation and even more general ones, there are natural group actions acting on the vector-valued forms and the algebraic invariants of these forms under such actions provide invariants for the given geometries. Before one should even hope to use these invariants, their algebraic structure must be known and a conceptual origin and language for their discussion should be given. In this work we pose such a formalism.  

\subsubsection{Organization of the paper.}
Building on the general idea of a square-zero extension, rather, a change of perspective on such objects, as posed in \cite{Diole1}, in §\ref{sec:trioles} we introduce a new type of graded commutative algebras (Definition \ref{TrioleAlgebraDefinition}) capturing the notion of a vector-valued form. From this perspective, the theory of vector-valued forms is not a theory of vector spaces supplied with additional structure, but really an aspect of our algebraic theory of what we call \emph{triole algebras.}

We describe differential calculus over such algebras and find the category of trioles is constructed so as to conceptualize the differential geometry of vector bundles provided with inner structures. In particular, in this setting the resulting algebraic differential calculus necessarily \emph{respects} these structures. Consequently, we find a satisfactory language that is well adapted to discuss symmetries and more generally, the preservation of general inner structures in vector bundles. 
Specifically, the concept of an infinitesimal symmetry of an inner structure, which is often realized as an element of the kernel of an appropriate covariant derivative related to a connection, appears simply as the derivations of our triole algebra. 
Therefore, without the need to pass to more complicated or external situations, we deduce the basis of a calculus of symmetries related to inner structures by considering only the functors of differential calculus and a suitable class of algebras. This language also suggests several natural generalizations of well-known objects, for instance of derivations in vector bundles, that are not observable by traditional geometric means. 

\subsection*{Acknowledgments}
This work contains selected results that were obtained as part of the authors PhD thesis which was prepared and defended at LAREMA, Department of Mathematics, University of Angers, France in July 2021. I am indebted to my supervisors and my monitoring committee for their patience and constant support during this time.
Several new results have been obtained during the authors post-doc appointment at Yanqi Lake Beijing Institute of Mathematical Sciences and Applications, Beijing, China. It is my pleasure to thank these organizations for comfortable working environments. 

We are also obligated to mention that some of the results obtained in this work are reported in an independent paper of ours which has been accepted for publication in the AMS Contemporary Mathematics Series for the Alexander M. Vinogradov Memorial Conference, \emph{Diffieties, Cohomological Physics and Other Animals} \cite{KryAMS}.

\subsection{Conventions}
\label{ssec:conventions}
We will freely use the functors of differential calculus \cite{Vin01},\cite{KraVerb},\cite{Ver} straightforwardly adapted to the setting of arbitrary graded commutative algebras, recalled for convenience below. In this paper we apply these constructions to a particular class of graded algebras introduced in this paper that we call \emph{triole algebras}. 
As their name suggests, they are graded algebras consisting of three components, and they immediately generalize the notion of a diolic algebra proposed in \cite{Diole1}. We refer the reader to §1.3.1 of that paper for our conventions regarding notation.

\subsubsection*{Differential Calculus over Graded Commutative Algebras}
Throughout this work refer to a functor of the differential calculus simply as a \emph{diffunctor}. 
We consider modules over a certain commutative $\mathbb{K}$-algebra $A$, where, for simplicity, we assume that $\mathbb{K}$ is a field of characteristic zero. In the cases where this must be specified, we denote the operator of multiplication of elements of a module $P$ by an element $a\in A$ by $a_P$. The fact that $\mathcal{O}$ is an object of a symmetric monoidal, closed and (co)-complete category $(\mathrm{C},\otimes)$ is expressed, by abuse of notation, as $\mathcal{O}\in\mathrm{C}.$
We will deal with categories of graded objects i.e. of $\mathcal{G}$-graded objects for some abelian group $\mathcal{G}$ such that
$$\mathrm{C}^{\mathcal{G}}=\prod_{\mathcal{G}}\mathrm{C},$$
and the associated categories of $\mathcal{G}$-graded commutative monoids $\mathrm{Comm}\big(\mathrm{C}^{\mathcal{G}}\big),$ and for such a graded algebra $\mathcal{A}$, the associated category of $\mathcal{G}$-graded modules over it, $\mathrm{Mod}_{\mathrm{C}}^{\mathcal{G}}(\mathcal{A}).$
\begin{example}
When $\mathrm{C}=\mathrm{Vect}_{\mathbb{K}}$ is the category of vector spaces over a ground field $\mathbb{K}$ of characteristic zero and adopt the usual terminology of saying that a vector space $V$ is $G$-\emph{graded} if $V=\bigoplus_{g\in G} V_g$ with $V_g\in\mathrm{Vect}_{\mathbb{K}},$ with $V_0:=\mathbb{K}.$
Elements $v\in V_g$ are said to be \emph{homogeneous of degree} $g\in G.$ We consider graded commutative $\mathbb{K}$-algebras and modules over them.
\end{example}
For any $\mathcal{P}=\bigoplus_{g\in \mathcal{G}}P_g,\mathcal{Q}=\bigoplus_{g\in \mathcal{G}}Q_g\in\mathrm{Mod}_{\mathrm{C}}^{\mathcal{G}}(\mathcal{A}),$ with $P_g,Q_g\in\mathrm{Mod}_{\mathrm{C}}(A_0)$, we say that a morphism $\varphi:\mathcal{P}\rightarrow \mathcal{Q}$ is of \emph{degree} $g$ if it is a $\mathbb{K}$-linear such that $\varphi(P_h)\subseteq Q_{g+h}.$
Given such a map, we understand the compositions $\varphi\circ a_{\mathcal{P}}$ and $a_{\mathcal{Q}}\circ \varphi$ for $a\in \mathcal{A}$ also as maps $\mathcal{P}\rightarrow \mathcal{Q}:$
$$\varphi\circ a_{\mathcal{P}}(p):=\varphi(ap),\hspace{10mm} a_{\mathcal{Q}}\circ \varphi(p):=a\varphi(p),\hspace{2mm} p\in \mathcal{P}.$$
Let $\mathrm{Hom}^{\mathcal{G}}(\mathcal{P},\mathcal{Q}):=\bigoplus_g\mathrm{Hom}_{\mathcal{R}}^g(\mathcal{P},\mathcal{Q}),$ and $\mathrm{Hom}_{\mathcal{R}}^g(\mathcal{P},\mathcal{Q})=\{\varphi:\mathcal{P}\rightarrow \mathcal{Q}|\varphi(\mathcal{P}_{h})\subseteq \mathcal{Q}_{h+g},h\in G\}.$
In $\text{Hom}_{\mathcal{R}}\big(\mathcal{P},\mathcal{Q}\big),$ we have two graded $\mathcal{A}$-module structures (left and right):
$$
a^<\varphi:=a_{\mathcal{Q}}\circ\varphi,\hspace{5mm}
a^>\varphi:=(-1)^{a\cdot \varphi}\varphi\circ a_{\mathcal{P}},$$
where $a_{\mathcal{P}},a_{\mathcal{Q}}$ are the endomorphisms of multiplication by $a$ in $\mathcal{P},\mathcal{Q}.$ 

For homogeneous $a\in\mathcal{A}$ and  $\varphi\in\mathrm{Hom}_{\mathcal{R}}(\mathcal{P},\mathcal{Q}),$ 
\begin{equation}
\delta_a\varphi:=\big[a,\varphi]=a^<\varphi-a^>\varphi.
\end{equation}

\begin{definition}
$\Delta\in\mathrm{Hom}_{\mathcal{R}}(\mathcal{P},\mathcal{Q})$ is a (graded) \textbf{\emph{differential operator}} of order $\leq k$ if $\delta_{a_0,\ldots,a_k}(\Delta)=0.$
\end{definition}
The set of  order $\leq k$ graded differential operators is denoted by 
$\mathrm{Diff}_k(\mathcal{P},\mathcal{Q}).$ This set carries natural left (resp. right) $A$-module structures denoted by $\mathrm{Diff}^<$ (resp. $\mathrm{Diff}^>$). Note further that $\mathrm{Diff}_0(\mathcal{P},\mathcal{Q}):=\mathrm{Hom}_{\mathcal{A}}(\mathcal{P},\mathcal{Q})$ and we have a sequence of  graded $\mathcal{A}$-bimodules
$$\mathrm{Diff}_0(\mathcal{P},\mathcal{Q})\subset \mathrm{Diff}_1(\mathcal{P},\mathcal{Q})\subset \ldots \subset \mathrm{Diff}_{k}(\mathcal{P},\mathcal{Q})\subset \mathrm{Diff}_{k+1}(\mathcal{P},\mathcal{Q})\subset\ldots$$
with $\mathrm{Diff}(\mathcal{P},\mathcal{Q}):=\bigcup_{k\geq 0}\mathrm{Diff}_k(\mathcal{P},\mathcal{Q}).$

Differential operators give a functorial assignment for each $k\geq 0$:
$$\mathrm{Diff}_k(\mathcal{A},-)_{\mathcal{G}}:\mathrm{gMod}^{\mathcal{G}}(\mathcal{A})\rightarrow \mathrm{gMod}^{\mathcal{G}}(\mathcal{A}).$$

The $\mathcal{A}$-bimodule $\mathrm{Diff}(\mathcal{A})_{\mathcal{G}}$ is also a $\mathcal{G}$-graded Lie algebra with respect to the graded commutator $[-,-]:\mathrm{Diff}_k(\mathcal{A})_g\times \mathrm{Diff}_{\ell}(\mathcal{A})_h\rightarrow \mathrm{Diff}_{k+\ell-1}(\mathcal{A})_{g+h},$ given by
$[\Delta_g,\nabla_h]=\Delta_g\circ\nabla_h-(-1)^{gh}\nabla_h\circ\Delta_g.$

The functor of \emph{graded derivations} is defined as $D_1(-)_{\mathcal{G}}:\mathrm{Mod}^{\mathcal{G}}(\mathcal{A})\rightarrow \mathrm{Mod}^{\mathcal{G}}(\mathcal{A}), \hspace{1mm} \mathcal{P}\mapsto D(\mathcal{P}):=\{\Delta\in\mathrm{Diff}_1^<(\mathcal{P})|\Delta(1_{\mathcal{A}})=0\}.$
Similar definitions arise for bi-derivations and higher order diffunctors, for instance $D_2^{\mathcal{G}}:\mathrm{Mod}^{\mathcal{G}}(\mathcal{A})\rightarrow\mathrm{Mod}^{\mathcal{G}}(\mathcal{A}),\hspace{1mm} \mathcal{P}\mapsto D_2(\mathcal{P}):=D_1^<\big(\mathcal{A},D_1(\mathcal{P})\subset \mathrm{Diff}_1^>(\mathcal{P})\big)$.

Let $\Delta:\mathcal{P}\rightarrow \mathcal{Q}$ be a DO of order $k$ and grading $g$. Let its $h$-th component be $\Delta_h:\mathcal{P}_h\rightarrow \mathcal{Q}_h[g]=\mathcal{Q}_{h+g}.$ Obviously, $\Delta_h$ is a DO of order $k$ in the category of $\mathcal{A}_0$-modules.
If $\Delta\in\mathrm{Diff}_k^{>,<}(\mathcal{P},\mathcal{Q})_{\mathcal{G}}$, according to the definition, for any $k$-tuples of elements $a_1,\ldots,a_k\in \mathcal{A}$, we have that 
$$\delta_{a_1,\ldots,a_k}(\Delta)\in\mathrm{Hom}_{\mathcal{A}}\big(\mathcal{P},\mathcal{Q}\big).$$
If we are in the situation that $\mathcal{P}=\mathcal{Q}=\mathcal{A},$ then $\delta_{a_1,\ldots,a_k}(\Delta)$ is viewed as an element of $\mathcal{A},$ which follows from the triviality $\mathrm{Hom}_{\mathcal{A}}(\mathcal{A},\mathcal{A})=\mathcal{A}.$
In other words, the assignment $(a_1,\ldots,a_k)\mapsto \delta_{a_1,\ldots,a_k}(\Delta)$ is a $k$-linear map $\mathcal{A}\times\ldots\times \mathcal{A}.$ 
Since we have that $\delta_{a,b}=(-1)^{ab}\delta_{b,a}$ it follows that this map is symmetric. Set 
$$[a_1,\ldots,a_n]_{\Delta}:=\delta_{a_1,\ldots,a_k}(\Delta).$$
The natural embedding of $\mathcal{A}$-modules $\mathrm{Diff}_{k-1}(\mathcal{A})\subset \mathrm{Diff}_k(\mathcal{A})$ allows us to define the $\mathcal{G}$-graded quotient module
$\mathrm{Smbl}_k(\mathcal{A}):=\mathrm{Diff}_k(\mathcal{A})/\mathrm{Diff}_{k-1}(\mathcal{A}),$
called the module of \emph{symbols} of order $k$.
The coset of the operator $\Delta\in\mathrm{Diff}_{k}(\mathcal{A})$ modulo the space $\mathrm{Diff}_{k-1}(\mathcal{A})$ is denoted by $\mathrm{smbl}_k(\Delta).$ 
The \emph{algebra of symbols} for the algebra $\mathcal{A}$ is defined by putting 
$$\mathrm{Smbl}_*(\mathcal{A}):=\bigoplus_{n=0}^{\infty}\mathrm{Smbl}_n(\mathcal{A}).$$
The multiplication operation in this algebra is induced by the composition of differential operators. More precisely, for elements $\mathrm{smbl}_{\ell}(\Delta)\in\mathrm{Smbl}_{\ell}(\mathcal{A})$ and $\mathrm{smbl}_k(\nabla)\in\mathrm{Smbl}_k(\mathcal{A}),$ we define
$$\mathrm{smbl}_{\ell}(\Delta)\star \mathrm{smbl}_k(\nabla):=\mathrm{smbl}_{k+\ell}(\Delta\circ \nabla)\in\mathrm{Smbl}_{\ell+k}(\mathcal{A}).$$
This operation is well defined, since it does not depend on the choice of representatives. The modules $\mathrm{Smbl}_n(\mathcal{A})$ are graded by elements of $\mathcal{G}$ (inherited grading from $\mathcal{A}$), and so the algebra $\mathrm{Smbl}_*(\mathcal{A})$ is $\mathbb{Z}\oplus\mathcal{G}$ graded. We will define the \emph{degree} of a graded symbol of a $g$-graded differential operator $\Delta$ to be $|\mathrm{smbl}_{k}(\Delta)|:=\deg(\Delta)=g\in\mathcal{G}.$
\begin{lemma}
The algebra $\mathrm{Smbl}_*(\mathcal{A})_{\mathcal{G}}$ is a $\mathcal{G}$-graded commutative algebra.
\end{lemma}
The Kozul sign rule simply reads as 
$\mathrm{smbl}_{\ell}(\Delta_h)\star \mathrm{smbl}_k(\nabla_g)=(-1)^{gh}\mathrm{smbl}_k(\nabla_g)\star \mathrm{smbl}_{\ell}(\Delta_h).$

\begin{remark}
Symbols of differential operators themselves define diffunctors 
$\mathrm{Smbl}_k(\mathcal{P},-):\mathrm{Mod}^{\mathcal{G}}(\mathcal{A})\rightarrow \mathrm{Mod}^{\mathcal{G}}(\mathcal{A}),$ and they are represented by modules of so-called \emph{cosymbols}. 
\end{remark}
The functors of $k$-multi-derivations, $D_k(-):\mathrm{Mod}^{\mathcal{G}}(\mathcal{A})\rightarrow \mathrm{Mod}^{\mathcal{G}}(\mathcal{A}),$ for $k\in \mathbb{N}$ is defined as in \cite[{\normalfont\S 2.0.1}]{Diole1}. An element $\nabla\in D_k(\mathcal{P})$ may be viewed as a graded skew-symmetric $\mathcal{P}$-valued multi-derivation of $\mathcal{A}$ with $k$-entries.

\begin{remark}
When we wish to pass to the geometric setting, which is to say, the context of differential geometry of vector bundles, we freely use that $P=\Gamma(X,\mathcal{P}_X)$ is the module of global sections of a locally free, finite rank sheaf of $A=\Gamma(X,\mathcal{O}_X)=\mathcal{O}_X(X)$-modules, where $X$ is an affine, smooth variety. We will freely pass between the language of $A$-modules and sheaves of $\mathcal{O}_X$-modules, without further mention.
\end{remark}

\section{General prerequisites}
In this section we recall some general facts concerning bilinear forms in §§\ref{ssec:Bilinformprelims} and the concept of symmetries of inner structures in vector bundles in §§\ref{sssec:Generalnotionofinnerstructure}.
\subsubsection{Preliminaries on bilinear forms}
\label{ssec:Bilinformprelims}
In what follows let $\mathbb{K}$ be a field of characteristic zero and let $A$ be a commutative unital $\mathbb{K}$-algebra. We work in the category of $A$-modules.
Let $P,Q$ be two such $A$-modules and  $g:P\times P\rightarrow Q$ be a morphism of $\mathbb{K}$-vector spaces which is additionally $A$-bilinear. We may then equivalently view $g:P\otimes_A P\rightarrow Q$ as an $A$-module homomorphism. Denote the collection of such bilinear forms by $\mathrm{Bil}(P,Q).$
\begin{remark}
When $X$ is a separated Noetherian $\mathbb{K}$-scheme and $\mathcal{E}=P,\mathcal{L}=Q$ correspond to locally free sheaves of $A=\mathcal{O}_X$-modules, respectively, then $g$ is equivalently a choice of a global section of the $\mathcal{O}_X$-module $\mathcal{H}\mathrm{om}\big(T^2\mathcal{E},\mathcal{L}),$ with $T(-)$ the tensor sheaf.
\end{remark}

The bilinear form $g$ is called \emph{symmetric} if it is invariant under the swapping $P\otimes P\rightarrow P\otimes P$ of the two factors. It is \emph{alternating} if $g|_{\Delta(P)}\equiv 0;$ that is, is identically zero when restricted to diagonal tensor $\Delta:P\rightarrow P\otimes P.$

\begin{remark}
In the sheaf-theoretic global setting, $g$ is symmetric (resp. alternating) if corresponding to a global section of $\mathcal{H}\mathrm{om}\big(\mathrm{Sym}^2\mathcal{E},\mathcal{L}\big)$ (resp. $\mathcal{H}\mathrm{om}\big(\bigwedge^2\mathcal{E},\mathcal{L}\big)$).
\end{remark}

A $Q$-\emph{valued quadratic form} over $A$, is a triple $(P,q,Q)$ where $q\in\text{Hom}_A(P,Q)$, which satisfies $q(ap)=a^2\cdot q(p),$ for $a\in A,p\in P,$ and there exists a corresponding bilinear form
$g_q:P\otimes P\rightarrow Q, g_q(p_1,p_2):=g(p_1+p_2)-q(p_1)-q(p_2),$
for $p_1,p_2\in P.$
The first condition above can be seen as requiring the $A$-module morphism $q$ to satisfy commutativity of 
\[
\adjustbox{scale=.88}{
\begin{tikzcd}
A\otimes P\arrow[d, "(-)^2\otimes q"] \arrow[r,] & P \arrow[d, "q"] 
\\
A\otimes Q\arrow[r,] & Q
\end{tikzcd}}
\]

\begin{lemma}
\label{bilquad}
Every bilinear form determines a quadratic form.
\end{lemma}
To show this, let $g\in\text{Bil}(P,Q)$ and write $q_b\in\text{Quad}(P,Q)$ to be the defined by 
$q_b:=g\circ \Delta:P\rightarrow P\otimes P\rightarrow Q.$  
\begin{remark}
It should be noted that not every quadratic form is determined by a bilinear form, so the converse to Lemma \ref{bilquad} fails.
\end{remark}

Given a bilinear form $g\in\text{Bil}(P,Q),$ we have the corresponding \emph{adjoint morphism}
\begin{equation}
    \label{eqn:biladj}
    \mathfrak{g}:=g^{\sharp}:P\rightarrow \mathrm{Hom}_A(P,Q),\hspace{5mm} \mathfrak{g}(p_1)p_2:=g(p_1,p_2).
\end{equation} 
Using the associated adjoint morphism to a vector-valued bilinear form $g$, we say that  
$g$ is \emph{non-singular/regular} if $\mathfrak{g}$ is an isomorphism of $A$-modules. Similarly, a quadratic form $q\in\mathrm{Quad}(P,Q)$ is non-singular when the associated bilinear form $g_q$ is.
Given a regular bilinear form $g$, a triple $(P,b,Q)$ is sometimes called a \emph{bilinear space} on $M.$
Without specifying morphisms for now, let us write the category of bilinear spaces on $M$ as $\mathrm{Bil}_M.$ 

It will be useful to define the $Q$-\emph{transpose} as an exact contravariant functor,
$(-)^{\circ} :=\mathrm{Hom}_A(-,Q):\mathrm{Mod}(A)\rightarrow \mathrm{Mod}(A).$
This comes with a canonical morphism of functors
$\mathrm{Can}^Q:\mathrm{Id}\rightarrow \big((-)^{\circ}\big)^{\circ}.$
When restricting to those $A$-modules which are geometric and locally free, this is an isomorphism.
The $Q$-transpose of a morphism $\psi:R\rightarrow \text{Hom}_A(R,Q),$ is defined to be the composition
$$\psi^{t}:R\rightarrow \text{Hom}\big(\text{Hom}(R,Q),Q\big)\rightarrow \text{Hom}(R,Q),$$
as $\psi^{t}:=(\psi)^{\circ}\circ \text{can}^Q.$
\begin{lemma}
$g\in\mathrm{Bil}(P,Q)$ is symmetric if and only if $\mathfrak{g}=\mathfrak{g}^t.$ Similarly, it is skew-symmetric if and only if $\mathfrak{g}=-\mathfrak{g}^t.$
\end{lemma}

\subsection{Inner structures and a calculus of symmetries}
We now present the definition of an inner structure and of a symmetry that we adopt in this paper.
\subsubsection{A general notion of `inner structure'}
\label{sssec:Generalnotionofinnerstructure}

Let $V$ denote a finite dimensional $\mathbb{K}$-vector space. We give meaning to the notion of an \emph{inner structure} on $V$ in terms of multi-linear algebra. Namely, consider the tensor algebra
$$T(V):=\bigoplus_{p,q}V_q^p,\hspace{5mm} V_q^p:=V^{\otimes p}\otimes (V^{\vee})^{\otimes q}.$$
Denoting by $\mathrm{GL}(V)$ the automorphism group of $V,$ note that any $g\in \mathrm{GL}(V)$ extends to a homomorphism of the algebra $T(V)$, for instance $g(\mathcal{P}_1\otimes \mathcal{P}_2)=g(\mathcal{P}_1)\otimes g(\mathcal{P}_2)$ for some $\mathcal{P}_1,\mathcal{P}_2\in T(V).$ We will write $g^{\otimes p}:=g\big|_{V^{\otimes p}\subset V_q^p}$ and similarly $g^{\vee \otimes q}$ for its restriction to $(V^{\vee})^{\otimes q}\subset V_q^p.$ 

The group $\mathrm{GL}(V)$ acts on $T(V)$ and the orbits are called \emph{tensor types}. Tensors belonging to the same orbit are said understood to be equivalent. Standard identifications such as $\mathrm{End}(V)=\mathrm{Hom}_{\mathbb{K}}(V,V)$ and $V\otimes V^{\vee}=\mathrm{Hom}(V,V)$ will be used without mention. For instance, a (Lie, associative, etc.) algebra structure on $V$ i.e., an element of $\mathrm{Hom}(V\otimes V,V)$ is to be identified with an element of $V_2^1=V\otimes V^{\vee}\otimes V^{\vee}$. This way the isomorphism class of such an algebra is identified with the tensor type of the corresponding to it element of $V_2^1.$

The same considerations are valid for an arbitrary $A$-module $P$, yielding $T(P)$ and so on, however the standard identifications as made above are generally not valid. In cases of interest, for instance in physics, when $P$ is the module of sections of a vector bundle, these identifications do hold.

\begin{definition}
An \textbf{\emph{inner structure}} in a vector bundle $\pi$ is an element of $T\big(\Gamma(\pi)\big).$
\end{definition}

Suppose now that $\nabla$ is a linear connection in a projective $A$-module $P$. In the algebraic setting, these are described by $A$-module homomorphisms $\nabla:D(A)\rightarrow \mathrm{Der}(P)$ from the $A$-module of derivations of $A$, to that of Der-operators in $P$, defined to be $\mathbb{K}$-linear endomorphisms $\nabla_X$ of $P$ satisfying the so-called Der-Leibniz rule $\nabla_X(ap)=X(a)p+a\nabla_X(p),$ for $a\in A,p\in P,X\in D(A).$ 
\begin{remark}
When $\nabla$ is a morphism of Lie algebras, we obtain the notion of a flat connection and this agrees with geometric definitions when taking $A:=C^{\infty}(M)$ and $P:=\Gamma(\pi)$.  Der-operators may therefore be understood geometrically as linear vector fields on the total space of $\pi$, and so elements $\nabla_X$ correspond to covariant derivatives. 
\end{remark}

Given such a connection $\nabla$ in $P$, natural cohomology theories may be associated. These arise from certain de Rham like complexes obtained by studying $\Omega^k(P)\equiv \Omega^k(A)\otimes_A P,$ the $P$-valued $k$-forms on $A.$ When $P$ is projective these coincide with alternating $k$-forms valued in $P.$
Following ones usual intuition for the de Rham algebra $(\Omega(M),d)$ on a smooth manifold $M$, we arrive at the existence of a first order differential operator $d_{\nabla}$ specified as follows. Its zeroth component is given by $d_{\nabla}^0:P\rightarrow \Omega^1(P),p\mapsto d_{\nabla}^0(p)$ which is defined on $X\in D(A)$ as $d_{\nabla}^0(p)(X):=\nabla_X(p).$
This is $A$-linear and we have the extension to $k$-forms as a map
$d_{\nabla}^k:\Omega^k(P)\rightarrow \Omega^{k+1}(P)$ defined in analogy with the ordinary de Rham differential as
\begin{eqnarray}
\label{eqn:CovariantdeRhamDifferential}
\big(d_{\nabla}^k\omega\big)(X_1,\ldots,X_{k+1})&=&\sum_i(-1)^{i+1}\nabla_{X_i}\big(\omega(X_1,\ldots,\hat{X}_i,\ldots,X_{k+1})\big) \nonumber
\\
&+&\sum_{i< j}(-1)^{i+j}\omega\big([X_i,X_j],X_1,\ldots\hat{X}_i,\ldots,\hat{X}_j,\ldots,X_{k+1}\big),\end{eqnarray}
for $X_1,\ldots,X_{k+1}\in D(A).$
It is natural to wonder when the sequence
\begin{equation}
\label{eqn:deRhamSequenceforConnections}
0\rightarrow P\xrightarrow{d_{\nabla}^0}\Omega^1(P)\xrightarrow{d_{\nabla}^1}\Omega^2(P)\rightarrow\ldots\rightarrow \Omega^n(P)\rightarrow 0,
\end{equation}
 defined by the differential (\ref{eqn:CovariantdeRhamDifferential}), is in fact a complex i.e. when the equation $d_{\nabla}\circ d_{\nabla}=0$ holds.

An answer is reached by understanding the obvious wedge product in $\Omega^*(P)$ gives this ths structure of a graded $\Omega(A)$-module. 
One may extend this algebra structure to $\wedge:\Omega^p\big(\mathrm{End}_A(P)\big)\times \Omega^k(P)\rightarrow\Omega^{p+k}(P)$ between $\mathrm{End}_A(P)$-valued forms and $P$-valued forms, explicitly given as 
$(\mathcal{R}\wedge\omega)(X_1,\ldots,X_{p+k}):=\sum_{I\subset I_{k,p},|I|=p}(-1)^I\mathcal{R}(X_I)\big(\omega(X_{\overline{I}})\big).$
There arises a degree $p$ homomorphism of $\Omega(A)$-modules, $\Omega^p\big(\mathrm{End}_A(P)\big)\rightarrow \Omega^{*+p}(P),\mathcal{R}\mapsto \mathcal{R}\wedge \omega,$ and one can observe for all $\omega \in\Omega^{k-2}(P)$, we get
\begin{equation}
    \label{eqn:squarezerocovariantdeRham}
   \big( d_{\nabla}\circ d_{\nabla}\big)(\omega)=R^{\nabla}\wedge \omega,
\end{equation}
where $R^{\nabla}$ is the curvature of $\nabla.$ In particular, $\big(d_{\nabla}\circ d_{\nabla}\big)$ is degree $2$ homomorphism and for $p\in P$ one has $\big(d_{\nabla}\circ d_{\nabla}\big)(p)=R^{\nabla}\wedge p.$ Consequently, one may conclude from (\ref{eqn:squarezerocovariantdeRham}) that
\begin{lemma}
\label{CovariantdeRhamisaComplex}
Sequence (\ref{eqn:deRhamSequenceforConnections}) is a complex if and only if $\nabla$ is a flat connection.
\end{lemma}
When Lemma \ref{CovariantdeRhamisaComplex} holds, the graded cohomology module of (\ref{eqn:deRhamSequenceforConnections}) is 
\begin{equation}
    \label{eqn:Flat connection cohomology spaces}
    H_{\nabla}(P):=\bigoplus_k H_{\nabla}^k(P),\hspace{5mm} H_{\nabla}^k(P):=\ker(d_{\nabla}^k)/\mathrm{im}(d_{\nabla}^{k-1}).
    \end{equation}

\begin{example}
\label{TrivialConnectiongivesdeRhamcohomology}
Suppose that $A=P=C^{\infty}(M)$ with $\nabla:=D$ the trivial linear connection defined by $D_X:=X,$ for every $X\in D(M)$. In this case $H_D(A)$ is nothing but the de Rham cohomlogy of $M.$
\end{example}

It is easy to understand what are zeroth cohomologies of a given flat connection. 

\begin{lemma}
\label{ZerothCohomologiesofConnection}
The space $H_{\nabla}^0(P)=\ker(d_{\nabla}^0)$ coincides with the space of $\nabla$-constant sections of $P$. 
\end{lemma}

Generally, the cohomologies (\ref{eqn:Flat connection cohomology spaces}) supply natural invariants of the geometry and deserve to be studied and computed, however, that is outside the scope of this paper; instead, we are interested in discussing to what extent the theory of connections should form a basis for the language on which we construct our mathematics. 

Any linear connection in the $A$-module $P$ induces a unique connection in the tensor product $P\otimes_A P$, as well as in the $A$-linear dual module $P^{\vee}$, one can easily see that there is an induced linear connection $\nabla^{T(P)}$ in $T(P).$
In particular, there are induced connections $\nabla_q^p\in P_q^p$ and they admit restrictions to various natural sub-modules of $P_q^p.$

\begin{example}
Consider $\nabla_q^p$ and the $\Omega(A)$-submodule $\Omega\big(\mathrm{Sym}^p(P)\big)$ of $\Omega\big(P_0^p\big).$
Then $d_{\nabla_q^p}$ leaves this submodule invariant and consequently 
$d_{\nabla^{\mathrm{Sym}}}:=d_{\nabla_0^k}\big|_{\Omega\big(\mathrm{Sym}^p(P)\big)}$ is a linear connection in $\mathrm{Sym}^p(P).$
\end{example}
\begin{example}
Similarly $\nabla_q^0$ restricts to provide a linear connection $\nabla^{\mathrm{Poly}}$ on $\mathrm{Sym}^q\big(P^{\vee}\big),$ the polynomials of degree $q$ on $P.$
\end{example}
The following definition plays a fundamental role throughout this paper.

\begin{definition}
\label{preserveaninnerstructure}
Let $\nabla$ be a linear connection in $P.$ It is said to \textbf{\emph{preserve an inner structure}} $\Xi\in T(P)$ if $d_{\nabla^{T(P)}}(\Xi)=0.$
An inner structure in an $A$-module $P$ that admits a preserving-it
linear connection is called a \textbf{\emph{gauge structure}}.
\end{definition}
Note that the condition given above for $\nabla$ preserving $\Xi$ is equivalent to $\nabla_X^{T(P)}(\Xi)=0,$ for every $X\in D(A).$

Furthermore, note that general inner structures are not usually gauge structures and we do not investigate here the existence of gauge structures in $A$-modules. For instance a generic endomorphism $\Xi\in P_1^1$ does not admit a preserving-it connection even locally.

Finally, let $Q$ be another $A$-module. We may follow the discussion presented above and define the notion of an \emph{inner structure in} $P$ with \emph{values in} $Q$, as some element of $\Xi\in T(P)\otimes_A Q.$
\begin{example}
Consider a $Q$-valued inner structure in $P$ of the form $\Xi\in P_2^0\otimes_A Q.$ Such inner structures are precisely the vector-valued bilinear forms on $P$ as in §§\ref{ssec:Bilinformprelims}.
\end{example}

\subsubsection{Symmetries}
\label{sssec:Symmetries}
Suppose that $P$ is supplied with a bilinear form $b:P\times P\rightarrow A.$
The $b$-\emph{orthogonal group} of $P$ is 
\begin{equation}
\label{orthogonalgroup}
\mathrm{O}(P,b):=\{\varphi\in\mathrm{Aut}(P)|b\big(\varphi(p_1),\varphi(p_2)\big)=b(p_1,p_2), p_1,p_2\in P\}.
\end{equation}
Moreover, the sub-module of \emph{infinitesimal symmetries} is
$$\mathrm{o}(P,b):=\{\varphi\in\mathrm{End}(P)|b\big(\varphi(p_1),p_2\big)+b\big(p_1,\varphi(p_2)\big)=0\}\subset \mathrm{End}(P).$$
This is not only a sub-module, but a Lie sub-algebra as well. Similarly, for any $\psi\in\mathrm{End}(P)$, the module of \emph{symmetries of} $\psi$ is
\begin{equation}
    \label{symmetries}
    \mathrm{GL}(P,\psi):=\{\Phi\in\mathrm{Aut}(P)|\big[\Phi,\psi\big]=0\}.
\end{equation}

One can obtain similar symmetry algebras and infinitesimal symmetry algebras for various types of internal structures.
For example, if $J$ is an inner complex structure $\mathrm{GL}(E,J)$ is the symmetry algebra of $J$. 

Generally, if $\psi$ is an endomorphism of an $A$-module $P$, all endomorphisms of $P$ commuting with it constitute a Lie sub-algebra $\mathrm{gl}(P,\psi)$ of $\mathrm{End}(P).$ In particular, $\mathrm{gl}(E,J)$ is the Lie algebra of infinitesimal symmetries of an inner complex structure $J$.

\section{The category of trioles}
\label{sec:trioles}
We now introduce the main objects of study in this paper, first in a general manner and then specifying to our case of interest. Let $(\mathrm{C},\otimes)$ be a category satisfying the assumptions of §\ref{ssec:conventions}.

\begin{definition}
\label{TrioleAlgebraDefinition}
A \textbf{\emph{triole algebra internal to}} $\mathrm{C}$, is an object $\EuScript{T}$ of $\mathrm{Comm}\big(\mathrm{C}^{\mathbb{Z}}\big)$ with homogeneous components $\EuScript{T}_0:=A,\EuScript{T}_1:=P,\EuScript{T}_2:=Q,$ with $\EuScript{T}_i:=\emptyset,$ for all other $i.$ We have that
$P,Q\in\mathrm{Mod}_{\mathrm{C}}(A),$ and $P$ is endowed with a $Q$-valued $A$-bilinear form\footnote{To be more precise, a $\mathrm{C}$-morphism defined by the usual hom-tensor adjunction as $\mathfrak{g}\in\text{Hom}_{\mathrm{C}}(P,\text{Hom}_{\mathrm{C}}(P,Q)\big))$, where we implcitly apply the adjunction over the tensor $P\cong P\otimes_A A.$} $g:P\times P\rightarrow Q$ such that the algebra structure is defined by the commutative algebra structure in $A, A\cdot A\subset A$ as well as the $A$-module structures
$A\cdot P\subseteq P,A\cdot Q\subset Q$ in $P,Q$, respectively and also by the condition that $P\cdot P\subseteq Q$ is dictated by $p_1\cdot p_2:=g(p_1,p_2)\in Q.$ Finally, we impose that $Q\cdot Q:=0.$
\end{definition}

\begin{remark}
Note that the condition that $Q\cdot Q=0$ demonstrates that these algebras closely resemble the notion of a $2$-trivial extension. 
\end{remark}
A \emph{triole} $\mathbb{K}$-\emph{algebra} is a triole algebra internal to the category $\mathrm{Vect}_{\mathbb{K}}$ and we simply call these \emph{triole algebras}.
Denote the subcategory of $\mathbb{Z}$-graded commutative $\mathbb{K}$-algebras, $\mathrm{Comm}_{\mathbb{K}}^{\mathbb{Z}}$ which
consists of triole algebras as 
$\mathrm{Triole}_{\mathbb{K}}.$ Denote these objects as $\EuScript{T}=A\oplus (P,g)\oplus Q.$
There exists a subcategory of $\mathrm{Triole}_{\mathbb{K}}$ whose objects $\EuScript{T}$ are triole algebras with graded fixed component $\EuScript{T}_0:=A$, that we denote by $\mathrm{Triole}_{\mathbb{K}}(A).$ We call such objects \emph{triole algebras over} $A.$ There is one further subcategory that we denote by $\mathrm{Triole}_{\mathbb{K}}(A,Q)$ and which consists of triole algebras with fixed components $\EuScript{T}_0:=A$ and $\EuScript{T}_2:=Q.$ We call these $Q$-\emph{valued triole algebras over} $A.$
In each of the above categories morphisms are defined as those in in $\mathrm{Comm}(\mathrm{C}^{\mathbb{Z}})$, and we will explain all details for $\mathrm{Triole}_{\mathbb{K}}(A).$

\begin{definition}
\label{morphismoftrioles}
A \textbf{\emph{morphism of triole algebras}} (over $A$) $\Psi:\EuScript{T}=A\oplus(P,b)\oplus Q\rightarrow \EuScript{T}'=A\oplus (P',b')\oplus Q'$ is a degree zero algebra morphism that amounts to the datum of a triple $\Psi=(\psi_0=\mathrm{id}_A,\psi_1,\psi_2)$ where $\psi_0$ is the identity map and where $\psi_1:P\rightarrow P'$ and $\psi_2:Q\rightarrow Q'$ are $A$-module homomorphism such that the diagram 
\begin{equation}
\label{eqn:trihomdiag}
\adjustbox{scale=.88}{
\begin{tikzcd}
P\times P\arrow[d,"\psi_1\times \psi_1"] \arrow[r, "b"] & Q \arrow[d,"\psi_2"]
\\
P'\times P' \arrow[r, "b'"] & Q'
\end{tikzcd}}
\end{equation}
of $A$-modules commutes,
i.e. $b'\circ (\psi_1\times \psi_1)=\psi_2\circ b.$
\end{definition}

One may consider important subcategories of triole algebras over $A$ with symmetric, alternating, or quadratic forms, that we denoted by $\mathrm{Triole}_{\mathbb{K}}^{\mathrm{sym}}(A),\mathrm{Triole}_{\mathbb{K}}^{\mathrm{alt}}(A)$ and $\mathrm{Triole}_{\mathbb{K}}^{\mathrm{quad}}(A),$ respectively.

\begin{remark}
\label{Forgetful functor Triole to Diole remark}
Let $\EuScript{T}=A\oplus (P,g)\oplus Q$ be a triole algebra over $A.$ 
There exists an obvious forgetful functor, 
$\mathrm{Triole}_{\mathbb{K}}(A)\rightarrow \mathrm{Diole}_{\mathbb{K}}(A),$ to the category of diolic algebras over $A,$ that was introduced in\cite{Diole1}.
This functor sends $\EuScript{T}$ to the diolic algebra $\EuScript{A}:=A\oplus P.$ Specifically, we forget the second component $Q$ and set $g$ to be the trivial multiplication $g(P,P):=\emptyset.$
\end{remark}

\begin{example}
Consider the symmetric algebra of $P$ over $A$, denoted by $\mathrm{Sym}_A^*(P).$ Define 
$\EuScript{T}_{\mathrm{Sym}}:=A\oplus (P,\odot)\oplus \mathrm{Sym}_A^2(P),$
where $\odot:P\otimes P\rightarrow P\odot P,$ is the symmetrizer.
Clearly this is just $\mathrm{Sym}_A^*(P)$ truncated at degree $2$. This is a triolic algebra, called the \textit{free symmetric triolic algebra on} $P.$ Similar considerations give the skew-symmetric free triole algebra.
\end{example}
Throughout this paper we will restrict our attention to a geometrically meaningful class of triole algebras.
\begin{definition}
\label{regulartriolealgebra}
A triole algebra $\EuScript{T}$ is said to be \textbf{\emph{regular}} when $A$ is smooth, $P,Q$ are geometric $A$-modules and $g$ is a regular bilinear form.
\end{definition}

Important examples of regular triole algebras arise when $Q$ is the module of smooth sections of a line bundle.
\begin{example}
Let $L$ be a line bundle over a smooth manifold $X$ with $Q=\Gamma(L).$ Let $a_1,\ldots,a_n$ be global sections of $\mathbb{G}_m.$ Consider
$\EuScript{T}_Q^{\mathbb{G}_m}:=A\oplus \big(Q^{\oplus n},\ell(a_,\ldots,a_n)^Q\big)\oplus Q^{\otimes 2}/T^{> 2}Q,$
with bilinear form
$\ell(a_1,\ldots,a_n)^Q:Q^{\oplus n}\otimes_A Q^{\oplus _n}\rightarrow Q^{\otimes 2}/T^{>2}Q,$
defined by 
$(q_1,\ldots,q_n)\otimes (q_1',\ldots,q_n')\longmapsto \sum_{i=1}^na_iq_i\otimes q_i'.$
This defines triole algebra over $A,$ precisely since the natural multiplication in $Q^{\otimes 2}$ inherited from the entire tensor algebra $T^{\bullet}Q,$ squares to zero in the quotient $Q^{\otimes 2}/T^{>2}Q.$
\end{example}

The construction of a category of triole algebras is actually a functorial assignment in the following sense.
\begin{lemma}
Let $\phi:A\rightarrow B$ be a morphism of commutative algebras and let $\EuScript{T}\in\mathrm{Triole}_{\mathbb{K}}(A),$ be a triole over $A.$ Then there is a canonically associated triole over $B$, denoted $\phi_*(\EuScript{T})
\in\mathrm{Triole}_{\mathbb{K}}(B).$
\end{lemma}
\begin{proof}
Via base change, we get that by definition the triole over $B$, as
$\phi_*(\EuScript{T}):=B\oplus (P_B,\phi_*b)\oplus Q_B,$
where the induced bilinear form $\phi_*b$ is the morphism of $B$-modules 
$\phi_*(b):P_B\times P_B\rightarrow Q_B,$
defined by
$\phi_*(b)\big(b_1\otimes p_1, b_2\otimes p_2\big):=b_1b_2\otimes b(p_1,p_2),$ for all $b_1,b_2\in B,p_1,p_2\in P.$
\end{proof}
This assignment of categories extends to a functor between the corresponding categories of triole algebras with fixed degree zero components. Namely, by
denoting this assignment by
$\mathcal{T}\text{ri}_{\mathbb{K}}:\mathrm{Comm}(\mathrm{C}^{\mathbb{Z}})\rightarrow \mathrm{Cat},$ where $\mathrm{Cat}$ is the ($(2,1)$-)category of ordinary categories, we have the following.
\begin{lemma}
For any $\phi:A\rightarrow B$, we have a corresponding functor
$\mathcal{T}\mathrm{ri}_{\mathbb{K}}(\phi):\mathrm{Triole}_{\mathbb{K}}(A)\rightarrow \mathrm{Triole}_{\mathbb{K}}(B).$
\end{lemma}
\begin{proof}
We will only show that given $\phi$ and a morphism $\Psi:\mathscr{T}\rightarrow \mathscr{T}'$ of trioles over $A$, we have an induced morphism 
$\phi_*(\Psi):\phi_*(\mathscr{T})\rightarrow \phi_*(\mathscr{T}'),$ of trioles over $B$. Recall that for $\Psi$ as above, we have the relation $b'\circ (\psi_1\times \psi_1)=\psi_2\circ b.$ By the functoriality of base change, we have that $\phi$ induces a $B$-module homomorphism 
$\phi_*(\psi_1):P_B\rightarrow P_B',$
defined as $b\otimes p\longmapsto b\otimes \psi_1(p),$ for $b\in B$ and $p\in P$, which is to say $\phi_*(\psi_1)=id_B\otimes \psi_1.$ This is the unique morphism determined by the commutativity with the universal homomorphisms for scalar extensions $\nu:P\rightarrow P_B$ and $\nu':P'\rightarrow P_B
'.$ That is, such that $\nu'\circ  \psi_1=\phi_*(\psi)\circ \nu.$ Similarly we have $\phi_*(\psi_2)=id_B\otimes \psi_2.$ With these components, we must show that $\phi_*(\Psi)$ is really a triolic map. That is, we need to establish that $\phi_*(b')\circ \phi_*(\psi_1\times \psi_1)=\phi_*(\psi_2)\circ \phi_*(b).$ Indeed this holds since
\begin{eqnarray*}
\phi_*(\psi_2)\circ \phi_*(b)(b_1b_2\otimes p_1\otimes p_2)&=& \phi_*(\psi_2)\circ \mu_B(b_1,b_2)\otimes b(p_1,p_2)
\\
&=&(id_B\otimes \psi_2)\circ (\mu_B\otimes b) \big(b_1b_2\otimes (p_1,p_2)\big)
\\
&=& (\mu_B\otimes\psi_2\circ b)\big(b_1b_2\otimes (p_1,p_2)\big)
\\
&=&\mu_B(b_1,b_2)\otimes \psi_2\big(b(p_1,p_2)\big)
\\
&=&\mu_B(b_1,b_2)
\otimes\big(b'\otimes (\psi_1\times \psi_1(p_1,p_2)\big)
\\
&=&\mu_B\otimes b'\circ (id_B\otimes \psi_1\times \psi_1)\big(b_1b_2\otimes (p_1,p_2)\big)
\\
&=&\phi_*(b')\circ \phi_*(\psi_1\times \psi_1)\big(b_1,b_2\otimes (p_1,p_2)\big),
\end{eqnarray*}
for all $b_i\in B,p_i\in P,$ as required.
\end{proof}

\subsubsection{Automorphisms of triole algebras}
\label{sssec:Automorphismsoftrioles}
In physics it is often of great interest to study when certain structures which represent physical quantities are invariant under automorphisms, i.e. gauge transformations. To this end, let $\mathrm{Aut}(P),$ and $\mathrm{Aut}(Q)$ denote the automorphism groups of $P,Q,$ respectively. There is a natural action of $\mathrm{Aut}(P)\times \mathrm{Aut}(Q)$ on $\mathrm{Bil}(P,Q).$ 
Consequently, we have have well-defined notions of automorphisms on all relevant structures which comprise a triolic algebra.
The action of
$\lambda=(\rho_P,\rho_Q)\in \text{Aut}(P)\times \text{Aut}(Q)$ on some vector-valued bilinear form $g:P\times P\rightarrow Q$
is defined as
$$\lambda g(p_1,p_2):=\rho_Q g\big(\rho_P^{-1}(p_1),\rho_P^{-1}(p_2)\big),$$
for $p_1,p_2\in P.$
What this means is that the bilinear form $\lambda g$ defined via tha pair of automorphisms $\rho_P,\rho_Q$ of $P,Q$, respectively is such that 
\[
\adjustbox{scale=.88}{
\begin{tikzcd}
P\times P\arrow[d, "\rho_P\times \rho_P"] \arrow[r, "g"] & Q \arrow[d, "\rho_Q"]
\\
P\times P\arrow[r, "\lambda g"] & Q
\end{tikzcd}}
\]
commutes.
In fact, we see this fits into the framework fo triolic morphism since the defining relation of such morphisms $b'\circ(\psi_1\times \psi_1)=\psi_2\circ g,$ can be case in the above form by noting that since $\psi_1\in \text{Aut}(P),$ we have the existence of $\psi_1^{-1}.$ Consequently we can write
$$b'=\psi_2\circ b (\psi_1^{-1}\times \psi_1^{-1}),$$
where we find that the transformed metric is 
$\lambda b:=b'.$
Identifying $\psi_2$ with $\rho_Q,$ as above we find conformity with our notions.
Such automorphisms should be studied in relation to gauge invariant quantities for physical systems modeled on triolic algebras.

\subsubsection{Isometries and similarities}
\label{sssec:Isometriesandsimilaritiesoftrioles}

A special class of morphisms of triolic algebras arises naturally, and consists of those which are isomorphisms of underlying vector bundles.

\begin{definition}
\label{isometryoftriole}
Let $\EuScript{T},\EuScript{T}'$ be two objects of $\mathrm{Triole}_{\mathbb{K}}(A,Q).$
A morphism of triole algebras $\varphi:\EuScript{T}=A\oplus (P,b)\oplus Q\rightarrow \EuScript{T}=A\oplus (P',b')\oplus Q,$ where $\varphi_0:\mathrm{id}_A$ and $\varphi_2=\mathrm{id}_Q,$ with an $A$-module isomorphism $\varphi_1:P\rightarrow P',$ such that the diagram
\[
\adjustbox{scale=.88}{
\begin{tikzcd}
P\otimes P \arrow[d, "\varphi_1\otimes \varphi_1"] \arrow[r, "b"] & Q \arrow[d, "\mathrm{id}_Q"]
\\
P'\otimes P'\arrow[r,"b'"] & Q
\end{tikzcd}}
\]
commutes,
is called an \textbf{\emph{isometry}}.
\end{definition}

Explicitly, $\varphi$ is an isometry if $b'\big(\varphi_1(p_1),\varphi_1(p_2)\big)=b(p_1,p_2).$ They are so-named as they induce an isometry of the underlying bilinear forms.

In terms of the associated adjoint morphisms, an isometry can be defined by requiring commutativity of the diagram
\[
\adjustbox{scale=.88}{
\begin{tikzcd}
P \arrow[d, "\varphi_1"] \arrow[r, "\mathfrak{b}"] & \text{Hom}_A(P,Q)
\\
P'\arrow[r,"\mathfrak{b}'"] & \text{Hom}_A(P',Q)\arrow[u, "(\varphi_1)^{\circ}"]
\end{tikzcd}}
\]

The composition of two isometries of trioles is easily seen to be an isometry. Moreover, each isometry has an inverse and so their totality forms a group, denoted by 
$\mathrm{Isom}_A(\EuScript{T},\EuScript{T}').$

Denote the category of regular $Q$-valued triole algebras over $A$ with morphisms given by isometries by $\mathrm{Triole}_{\mathbb{K}}(A;Q)_{\mathrm{isom}}.$ All subcategories consisting of symmetric, quadratic or alternating forms will be appropriately embellished (for instance, $\mathrm{Triole}^{\mathrm{sym}}(A;Q)_{\mathrm{isom}}$).
One defines an isometry of a quadratic form to be an isomorphism of $A$-modules $\varphi:P\cong P'$ such that $q'\circ \varphi=q.$ 

We will introduce one more interesting class of morphisms between triole algebras.
\begin{definition}
\label{similarityoftrioles}
A morphism of triole algebras $\Psi=(id_A,\psi_1,\psi_2)$ over $A$, where $\psi_1:P\rightarrow P'$ and $\psi_2:Q\rightarrow Q',$ are $A$-module isomorphisms is said to be a \textbf{\emph{similarity transformation}}.
\end{definition}
Given a similarity transformation $\Psi,$ it is in particular a morphism of triole algebras in the sense that 
the usual diagram (\ref{eqn:trihomdiag}) commutes. By passing to adjoints, we see that being a similarity transform of trioles over $A$ amounts to the commutativity of 
\[
\adjustbox{scale=.88}{
\begin{tikzcd}
P\arrow[d, "\psi_1"] \arrow[r, "\mathfrak{b}"] & \text{Hom}_A(P,Q)
\\
P'\arrow[r,"\mathfrak{b}'"] & \text{Hom}_A(P',Q')\arrow[u, "\psi_2^{-1} (\psi_1)^{\circ}"]
\end{tikzcd}}
\]
The morphism on the right-most vertical is simply $\psi_2^{-1}\circ (\psi_1)^{\circ}:\text{Hom}_A(P',Q')\rightarrow \text{Hom}_A(P,Q),$ defined by $f\longmapsto \psi_2^{-1}\circ f\circ \psi_1.$

The collection of similarity transformations between two triole algebras $\EuScript{T},\EuScript{T}'$ has the structure of a group, denoted by $\mathrm{Sim}_A(\EuScript{T},\EuScript{T}').$

\subsubsection{Coordinates}
\label{sssec:coordinatesforforms}
Let $g:P\otimes P\rightarrow Q$ be a vector valued fiber metric. The $A$-module of such forms $\mathrm{Bil}(P,Q)$ is to be understood as taking place in the canonical identification
\begin{equation}
\label{eqn:Bilinear form identification}
    \mathrm{Bil}(P,Q)\cong P^{\vee}\otimes P^{\vee}\otimes Q.
\end{equation}
Let $(e_{\alpha}),(\epsilon_A)$ be bases for $P,Q$ and write $(e^{\alpha}),(\epsilon^A)$ the bases of the corresponding dual modules. 
Every $\Phi \in\text{Aut}(P)$ determines $\mathbb{G}=(g_{\alpha}^{\beta})\in \text{Gl}(m_P;A)$ via the action $\mathbb{G}(e_{\alpha})=g_{\alpha}^{\beta}e_{\beta}.$
Similarly, $\mathbb{H}\in \text{Aut}(Q)$ determines a matrix $\mathbb{H}=(h_{A}^B)\in\text{Gl}(m_Q;A),$ via $\mathbb{H}(\epsilon_A)=h_A^B\epsilon_B.$
Every $g\in\text{Bil}(P,Q)$ determines an $m_Q$-tuple of scalar bilinear forms $(g^1,g^2,\ldots,g^{m_Q}),$ with $g^A=\epsilon^A\circ g,$ for $A=1,\ldots,m_Q.$
 
\begin{example}
Suppose our bundles are of rank $2.$ Then $G\in \mathrm{Bil}(P,Q)$ is represented under the identification (\ref{eqn:Bilinear form identification}) as either
$$G=\bigg(\scalemath{.80}{\begin{bmatrix}
g_{11}^1 & g_{12}^1
\\
g_{21}^1 & g_{22}^1
\end{bmatrix},\begin{bmatrix}
g_{11}^2 & g_{12}^2
\\
g_{21}^2 & g_{22}^2
\end{bmatrix}} \bigg)\hspace{3mm}
\text{
or as a tensor }\hspace{2mm}
g_{ij}^k e^i\otimes e^j\otimes f_k,$$ which in components, reads
$g_{11}^1 e^1\otimes e^1\otimes f_1 +g_{12}^1e^1\otimes e^2\otimes f_1+g_{21}^1 e^2\otimes e^1\otimes f_1+g_{22}^1e^2\otimes e^2\otimes f_1+g_{11}^2e^1\otimes e^1\otimes f_2+g_{12}^2e^1\otimes e^2\otimes f_2 + g_{21}^2e^2\otimes e^1\otimes f_2 +g_{22}^1e^2\otimes e^2\otimes f_2.$
\end{example}

In turn, each $g^A$ determines $G^A\in\text{Gl}(m_P\times m_P;A)$ where $G^A=(G_{\alpha\beta}^A),$ with $G_{\alpha\beta}^A=b^A(e_{\alpha},e_{\beta}).$
In such coordinates, the action of $\mathrm{Aut}(P)\times \mathrm{Aut}(Q)$ on $\mathrm{Bil}(P,Q)$ is defined by setting $\tilde{g}=(\mathbb{G},\mathbb{H})g$ and consequently, we have that
$$\tilde{G}_{\alpha\beta}^A=h_B^A\big(g^{-1}\big)_{\alpha}^{\delta}\big(g^{-1}\big)_{\beta}^{\gamma}G_{\delta\gamma}^B.$$

\subsection{Modules over triole algebras}
\label{ssec:Modulesovertrioles}
We now introduce the module-objects over an algebra of trioles, and this is done in the usual categorical manner. 
Namely, a \emph{left} $\EuScript{T}$-\emph{module} $\EuScript{R}$ is a graded $\mathbb{K}$-vector space with an action $\nu\colon \EuScript{T}\otimes\EuScript{R}\rightarrow \EuScript{R}$ such that the diagrams
\[
\adjustbox{scale=.88}{
\begin{tikzcd}
\EuScript{T}\otimes\EuScript{T}\otimes\EuScript{R}\arrow[d,"\mu\otimes \EuScript{R}"] \arrow[r, "\EuScript{T}\otimes \nu"] & \EuScript{T}\otimes\EuScript{R}\arrow[d,"\nu"]
\\
\EuScript{T}\otimes\EuScript{R}\arrow[r,"\nu"]&\EuScript{R}
\end{tikzcd}\hspace{50pt}
\begin{tikzcd}
1\otimes\EuScript{R}\arrow[d,"i\otimes\EuScript{R}"] \arrow[r,"\cong"] & \EuScript{R}
\\
\EuScript{T}\otimes\EuScript{R}\arrow[ur,"\nu"]
\end{tikzcd}}
\]
commute, where $i$ is the unit map for the algebra $\EuScript{T}$ and $\mu$ is the algebra multiplication.

Explicitly, the action $\nu$ reads homogeneously as
$\EuScript{T}_0\cdot \EuScript{R}_g\subseteq \EuScript{R}_g,$ when restricted to degree zero components, and
$P\cdot \EuScript{R}_g\subseteq \EuScript{R}_{g+1},Q\cdot \EuScript{R}_g\subseteq \EuScript{R}_{g+2}$
for all $g\in \mathbb{Z}.$ In particular each component of $\EuScript{R}$ is an $A$-module.
We obtain an obvious category of $\EuScript{T}$-modules, to be denoted $\mathrm{Mod}(\EuScript{T})$. 

Rather than consider arbitrary modules over a triolic algebra, we pass immediately to a particularly relevant subcategory of $\mathrm{Mod}(\EuScript{T})$.

\begin{definition}
\label{TruncatedTriolicModuleDefinition}
A \textbf{\emph{truncated triolic module}} is a $\mathbb{Z}$-graded $\EuScript{T}$-module $\EuScript{R}$ given by the datum $(\EuScript{R},\nu_A^{(i)},\nu_P^{(0)},\nu_P^{(1)},\nu_Q^{(0)}),$
where: $\nu_A^{(i)}:A\otimes \EuScript{R}_i\rightarrow \EuScript{R}_i$ are the $A$-module structures for $i=0,1,2$, $\nu_P^{(0)}:P\otimes \EuScript{R}_0\rightarrow \EuScript{R}_1,\nu_P^{(1)}:P\otimes \EuScript{R}_1\rightarrow \EuScript{R}_2,$ and finally $\nu_Q^{(0)}:Q\otimes \EuScript{R}_0\rightarrow \EuScript{R}_2.$ These are
subject to the compatibility condition
\begin{equation}
\label{eqn:truntrimod}
\nu_{\EuScript{T}_i}^{(j+k)}\circ\big(id_{\EuScript{T}_i}\otimes \nu_{\EuScript{T}_j}^{(k)}\big)=\nu_{\EuScript{T}_{i+j}}^{(k)}\circ \mu_{i,j}\circ \alpha_{ijk},
\end{equation}
where 
$\alpha_{ijk}:\EuScript{T}_i\otimes (\EuScript{T}_j\otimes \EuScript{R}_k)\rightarrow (\EuScript{T}_i\otimes\EuScript{T}_j)\otimes\EuScript{R}_k,$ is the associator isomorphism, $\mu_{i,j}:\EuScript{T}_i\otimes\EuScript{T}_j\rightarrow \EuScript{T}_{i+j}$ is the multiplication.\footnote{In particular $\mu_{P,P}(-,-)=g(-,-):P\otimes P\rightarrow Q,$}
\end{definition}
Let us unpack the expression above.
If $\nu_A^{(i)}$ are the $A$-module structures for each homogeneous component $\EuScript{R}_i\subset\EuScript{R},$ we have in particular, $\nu_A^{0}:A\otimes \EuScript{R}_0\rightarrow \EuScript{R}_0,\nu_A^{1}:A\otimes \EuScript{R}_1\rightarrow \EuScript{R}_1,$ and $\nu_A^2:A\otimes \EuScript{R}_2\rightarrow \EuScript{R}_2.$
The above compatibility condition between the maps tells us that
$$
\scalemath{.88}{
\begin{cases}
\nu_P^{(1)}\big(p_1,\nu_{P}^{(0)}(p_2,r_0)\big)=\nu_{Q}^{(0)}\big(g(p_1,p_2),r_0\big)\in R_2,
\\
\nu_{P}^{(0)}\big(ap,r_0\big)=a\nu_P^{(0)}\big(p,r_0)\in R_1,
\\
\nu_P^{(1)}\big(ap,r_1\big)=a\nu_P^{(1)}(p,r_1)\in R_2,
\\
\nu_Q^{(0)}\big(aq,r_0\big)=a\nu_Q^{(0)}(q,r_0)\in R_2.
\\
\end{cases}}
$$
Consequently, a truncated triolic module $(\EuScript{R},\nu)$ is the datum of  a triple $\EuScript{R}=\EuScript{R}_0\oplus \EuScript{R}_1\oplus \EuScript{R}_2,$ with $\EuScript{R}_i\in \mathrm{Mod}(A),$ whose $\EuScript{T}$-module structure maps $\nu$ are all $A$-bilinear and satisfy the compatibility equation
\begin{equation}
    \label{eqn:simptrunctrimod}
\nu_P^{(1)}\big(p_1,\nu_{P}^{(0)}(p_2,r_0)\big)=\nu_{Q}^{(0)}\big(g(p_1,p_2),r_0\big).
\end{equation}

\begin{definition}
A \textbf{\emph{strict morphism}} $\Psi:(\EuScript{R},\nu,\eta)\rightarrow (\EuScript{R}',\nu',\eta')$ of truncated $\EuScript{T}$-modules is a collection $\Psi=(\psi_0,\psi_1,\psi_2)$ of $A$-module homomorphisms which respect the triolic structure maps $\nu,\eta$.
\end{definition}
Respecting the structure maps means that
$\psi_1\big(\nu_{R_0}(p,r_0\big)=\nu_{R_0'}\big(p,\psi_0(r_0)\big),\psi_2\big(\nu_{R_1}(p,r_1\big)=\nu_{R_1'}\big(p,\psi_1(r_1)\big)$ and $\psi_2\big(\eta_{R_0}(q,r_0)\big)=\eta_{R_0'}\big(q,\psi_0(r_0)\big).$
Truncated triolic modules together with strict triolic homomorphisms constitute a category $\tau\mathrm{Mod}(\EuScript{T}).$ 

\begin{remark}
\label{EaseofNotation}
To alleviate the burden of cumbersome notation, in what follows we will denote $\lambda_i:=\nu_P^{(i)}$ for $i=0,1,$ and $\nu:=\nu_Q^{(0)}.$
The compatibility relation is thus $\lambda_1\big(p_1\otimes \lambda_0(p_2\otimes r_0)\big)=\nu\big(g(p_1,p_2)\otimes r_0)\big).$
\end{remark}

\subsubsection{Building triole algebras}
\label{sssec:BuildTrioles}
To study various constructions related to natural operations on bilinear forms, it is necessary to construct artificial triolic algebras in the sense of \cite{Diole1}.
Let $\EuScript{T}=A\oplus (P,b)\oplus Q$ and $\EuScript{T}'=A\oplus (P',b')\oplus Q$ be two triole algebras over $A.$

\begin{definition}
The \textbf{\emph{triolic orthogonal sum}} is $\EuScript{T}\perp \EuScript{T}':=A\oplus \big(P\otimes P', b+b'\big)\oplus Q,$
where 
$(b+b')\big((p_1,p_1'),(p_2,p_2')\big):=b(p_1,p_2)+b'(p_1',p_2').$
\end{definition}
This defines a functor $\perp:\mathrm{Triole}(A)\times \mathrm{Triole}(A)\rightarrow \mathrm{Triole}(A).$
The interaction of this functor with various species of triolic algebras is provided in the following.
\begin{lemma}
We have that $\perp\colon \mathrm{Triole}^{\sharp}(A)\times \mathrm{Triole}^{\sharp}(A)\rightarrow \mathrm{Triole}^{\sharp}(A)$, where $\sharp=\mathrm{sym},\mathrm{quad},\mathrm{alt}.$
\end{lemma}
\begin{proof}
Follows immediately by noting the orthogonal sum of two symmetric, quadratic or alternating forms is again symmetric, quandratic or alternating, respectively.
\end{proof}

\begin{definition}
The \textbf{\emph{triolic product}}, denote $\otimes^{\mathrm{T}}$ of $\EuScript{T}=A\oplus (P,b)\oplus Q$ and $\EuScript{T}'=A\oplus (P',b')\oplus Q',$ is the triole algebra given by $\EuScript{T}\otimes^{\mathrm{T}} \EuScript{T}':=A\oplus \big(P\otimes P', b\otimes b'\big)\oplus Q\otimes Q',$
with 
$(b\otimes b')\big(p_1\otimes p_1',p_2\otimes p_2'\big):=b(p_1,p_2)\otimes b'(p_1',p_2'),$
for $p_i\in P,p_i'\in P',i=1,2.$
\end{definition}
The following is readily checked.

\begin{lemma}
The triolic product defines a functor $\otimes^{\mathrm{T}}:\mathrm{Triole}(A)\times \mathrm{Triole}(A)\rightarrow \mathrm{Triole}(A)$ which has the following properties,
\begin{itemize}
    \item $\mathrm{Triole}^{\mathrm{sym}}(A)\times \mathrm{Triole}^{\mathrm{sym}}(A)\rightarrow \mathrm{Triole}^{\mathrm{sym}}(A),$
    
    \item $\mathrm{Triole}^{\mathrm{alt}}(A)\times \mathrm{Triole}^{\mathrm{alt}}(A)\rightarrow \mathrm{Triole}^{\mathrm{alt}}(A)$
    
    \item $\mathrm{Triole}^{\mathrm{quad}}(A)\times \mathrm{Triole}^{\mathrm{quad}}(A)\rightarrow \mathrm{Triole}^{\mathrm{quad}}(A)$
    
    \item $\mathrm{Triole}^{\mathrm{sym}}(A)\times \mathrm{Triole}^{\mathrm{alt}}(A)\rightarrow \mathrm{Triole}^{\mathrm{alt}}(A)$
    
    \item $\mathrm{Triole}^{\mathrm{sym}}(A)\times \mathrm{Triole}^{\mathrm{quad}}(A)\rightarrow \mathrm{Triole}^{\mathrm{quad}}(A).$
\end{itemize}
\end{lemma}

Here with give a final example of a triolic construction.

\begin{example}
\label{Determinant triole}
Let $\EuScript{T}=A\oplus (P,g)\oplus Q$ be a triole with $P$ of rank $n$. Consider the determinant functor $\det$ applied to the adjoint morphism $\mathfrak{g}$ and the composition of rank $1$-modules, 
$$\det(P)\xrightarrow{\det(\mathfrak{g})}\det\big(\mathrm{Hom}(P,Q)\big)
\xrightarrow{\gamma}\mathrm{Hom}\big(\det(P),Q^{\otimes n}\big),$$
where the morphism
$\gamma:\det\big(\mathrm{Hom}(P,Q)\big)\rightarrow \mathrm{Hom}\big(\det(P),Q^{\otimes n}\big)$ is given by $(f_1\wedge\ldots\wedge f_n)\mapsto \gamma(f_1\wedge\ldots\wedge f_n)(p_1\wedge\ldots\wedge p_n)$ defined to be $\det\big(f_i(p_j)_{ij}\big).$
The \emph{determinant triole} associated to $\EuScript{T}$ is  $\det(\EuScript{T}):=A\oplus \big(\det(P),\det(g)\big)\oplus Q^{\otimes n}$ whose adjoint is given by $\gamma \circ \det(\mathfrak{g}).$ 
The determinant form is 
$\det(P)\otimes \det(P)\rightarrow Q^{\otimes n}$ given by $p_1\wedge\ldots\wedge p_n\otimes p_1'\wedge\ldots\wedge p_n'\mapsto \det\big(g(p_i,p_j')_{ij}\big).$
\end{example}

\begin{lemma}
The assignment of a triole algebra to the corresponding determinant triole as in Example \ref{Determinant triole} is functorial with respect to similarities. That is, we have a well-defined functorial assignment $\mathrm{det}:\mathrm{Triole}(A)^{\mathrm{Sim}}\rightarrow \mathrm{Triole}(A)^{\mathrm{Sim}}.$
\end{lemma}
\begin{proof}
Let $\Psi=(\mathrm{id}_A,\psi_1,\psi_2):\EuScript{T}=A\oplus (P,g)\oplus Q\rightarrow \tilde{\EuScript{T}}=A\oplus (\tilde{P},\tilde{g})\oplus \tilde{Q}$ be a similarity of trioles over $A.$
Then since $\psi_1,\psi_2$ are $A$-module isomorphisms it is readily verified that the map
$$\mathrm{det}(\Psi):=\big(\mathrm{id}_A,\mathrm{det}(\psi_1),\psi_2^{\otimes n}):\mathrm{det}(\EuScript{T})\rightarrow \mathrm{det}(\tilde{\EuScript{T}})$$
is a similarity of trioles over $A$ as well. This follows from the fact that the relation $\mathrm{det}(\tilde{g})\big(\det(\psi_1)\otimes \det(\psi_1)\big)=\psi_2^{\otimes n}\mathrm{det}(g)$ holds.
\end{proof}

\subsubsection{Compliments, Lagrangians and sub-algebras}
Fix $\EuScript{T}=A\oplus (P,g)\oplus Q\in\mathrm{Triole}(A,Q).$
Recall that a 
\emph{graded sub-algebra}, $\EuScript{B}\subset\EuScript{A}$ of an arbitrary $\mathcal{G}$-graded commutative $\mathbb{K}$-algebra is a graded vector-subspace which is closed under the graded commutative monoidal structure. Namely, we have $\mu_{\EuScript{A}}|_{\EuScript{B}}:\EuScript{B}_i\times \EuScript{B}_j\rightarrow \EuScript{B}_{i+j},$ for $i,j\in G.$ 
We study a particularly simple class of sub-triolic algebras $\EuScript{T}'\subset\EuScript{T}$ in the category $\mathrm{Triole}(A,Q).$ Such sub-algebras are in one-to-one correspondence with sub-$A$-modules $S\subset P$ with induced form $g|_S:S\otimes S\rightarrow Q.$
Supposing that $S\subset P$ is such a sub-module, which moreover corresponds to a vector sub-bundle $\mathcal{K}\subset\mathcal{E},$ then we have an exact sequence of vector bundles on $X$ given as $0\rightarrow \mathcal{K}\xrightarrow{i}\mathcal{E}\xrightarrow{p}\mathcal{E}/\mathcal{K}\rightarrow 0.$
\begin{definition}
The $Q$-\textbf{\emph{valued orthogonal compliment}} of $S$, is the sub-module of $P$, defined by 
$S^{\perp}:=\ker\big(i^{\circ Q}\circ g^{\sharp}\big).\subset P.$
\end{definition}
This is the kernel of the composition
$P\xrightarrow{g^{\sharp}}\mathrm{Hom}(P,Q)\xrightarrow{i^{\circ}}\mathrm{Hom}(S,Q)\rightarrow 0.$

\begin{remark}
In the geometric setting, $\mathcal{K}^{\perp}$ is the vector bundle defined via the sequence of locally free sheaves, 
$0\rightarrow \mathcal{K}^{\perp}\xrightarrow{i^{\perp}}\mathcal{E}\rightarrow \mathcal{H}\mathrm{om}\big(\mathcal{K},\mathcal{L}\big)\rightarrow 0.$
\end{remark}
In particular, letting $P/S$ denote the module of sections corresponding to the $\mathcal{E}/\mathcal{K}$, we have a canonical $A$-module isomorphism 
\begin{equation}
    S^{\perp}\cong\mathrm{Hom}\big(P/S,Q\big),
\end{equation}
which fits into the following digram,
\[
\adjustbox{scale=.88}{
\begin{tikzcd}
0\arrow[r,] & S^{\perp}\arrow[d, "\mathfrak{g}|_{S^{\perp}}"] \arrow[r, "i^{\perp}"] & P \arrow[d, "\mathfrak{g}"] \arrow[r,"i^{\circ}\circ \mathfrak{g}"] & \text{Hom}(S,Q)\arrow[d, "id"] \arrow[r,] & 0
\\
0\arrow[r,] &  \text{Hom}(P/S,Q)\arrow[r, "p^{\circ}"] & \text{Hom}(P,Q) \arrow[r, "i^{\circ}"] & \text{Hom}(S,Q) \arrow[r,] & 0
\end{tikzcd}}
\]
\begin{definition}
Let $\EuScript{T}\in\mathrm{Triole}(A,Q).$ Suppose $S\subset P.$
Then $S$ is a $Q$-\textbf{\emph{valued sub-Lagrangian}} if $S\subset S^{\perp}.$ The corresponding sub-triole, $\EuScript{T}_S\subset \EuScript{T},$ is called a \textbf{\emph{sub-Lagrangian triole}} in $\mathrm{Triole}(A;Q)$
\end{definition}
We are moreover interested in the situation when $g|_{S}\equiv 0.$  In this case, a a triolic sub-algebra $\EuScript{T}_S=A\oplus (S,g)\oplus Q$ of $\EuScript{T}$, is called a \emph{Lagrangian sub-triole} in $\mathrm{Triole}(A;Q),$ when $S$ is such that $S=S^{\perp}.$ We call the distinguished homogeneous piece $S$ a $Q$-\emph{valued Lagrangian}.
It is important to understand how Lagrangian triole algebras interact with various natural operations on triole algebras. 
\begin{lemma}
Let $\EuScript{T}=A\oplus (P,g)\oplus Q$ and $\EuScript{T}'=A\oplus (P',g')\oplus Q'$.
Let $S\subset P$ be a sub-Lagrangian and let $K\subset P$ be a Lagrangian.
Then $S\otimes P'\subset P\otimes P'$ is a sub-Lagrangian and $K\subset P'\subset P\otimes P'$ is a Lagrangian of the triolic product
$\EuScript{T}\otimes^{\mathrm{T}}\EuScript{T}'.$
\end{lemma}

\subsubsection{Triolic Lie Algebras}
\label{sssec:TriLieAlg}
We will now extend our result describing the notion of a diolic Lie algebra (§§3.0.1 of \cite{Diole1}) to the triolic formalism.
To this end, we understand that a triolic $\mathbb{K}$-Lie algebra is a graded vector space $\mathfrak{g}=\mathfrak{g}_0\oplus \mathfrak{g}_1\oplus \mathfrak{g}_2,$ with a graded Lie bracket $[-,-]:\mathfrak{g}_i\otimes_{\mathbb{K}}\mathfrak{g}_j\rightarrow \mathfrak{g}_{i+j},$ with the further condition that $\mathfrak{g}_2\cdot \mathfrak{g}_2\equiv 0.$ 
\begin{lemma}
The datum of a triolic Lie algebra is equivalent to the following items: a Lie $\mathbb{K}$-algebra $(\mathfrak{g}_0,[-,-])$, together with representations $\rho_i:\mathfrak{g}_0\rightarrow \mathrm{End}_{\mathbb{K}}(\mathfrak{g}_i)$ for $i=1,2$ with a skew-symmetric form $\big<-,-\big>:\mathfrak{g}_1\otimes_{\mathbb{K}}\mathfrak{g}_1\rightarrow \mathfrak{g}_2,$ such that the relation
$$\rho_2(a)\big<\xi_1,\xi_2\big>=\big<\rho_1(a)\xi_1,\xi_2\big>+\big<\xi_1,\rho_1(a)\xi_2\big>$$
holds in $\mathfrak{g}_2,$ for all $a\in\mathfrak{g}_0,\xi_1\in\mathfrak{g}_1,\xi_2\in\mathfrak{g}_2.$
\end{lemma}

With the preliminary aspects of triolic linear algebra understood, we turn our attention to describing the basic functors of differential calculus over triolic algebras.

\section{Differential Calculus in Triole Algebras}
\label{sec: Differential Calculus in Triole Algebras}

We work purely algebraically now and describe the functors of graded derivations, graded differential operators and their symbols in the setting of triolic algebras and their truncated modules. 

\subsection{Derivations in Triole algebras}
\label{ssec: Derivations in Triole algebras}
Let $D_1(\EuScript{T})_i$ denote degree $i$ order $1$ derivations of the triolic algebra $\EuScript{T},$ with $D_1(\EuScript{T})_{\mathcal{G}}=\bigoplus_{i\in \mathcal{G}}D_1(\EuScript{T})_i,$ the corresponding graded module $\EuScript{T}$-module.
Such graded derivations lie in admissible degrees $\{-2,-1,0,1,2\}$, which we now describe in turn.
\begin{lemma}
\label{triolederivdegzero}
A degree $0$ triolic derivation is a triple $X_0=\big(X_0^A,X_0^P,X_0^Q\big)$ of operators with $X_0^A\in D(A)$ an ordinary derivation, and with $X_0^P\in \mathrm{Der}(P),X_0^Q\in \mathrm{Der}(Q)$ two Der-operators with shared scalar symbol $\sigma(X_0^P)=\sigma(X_0^Q)=X_0^A,$ such that $X_0^Q\big(g(p_1,p_2)\big)=g\big(X_0^P(p_1),p_2\big)+g\big(p_1,X_0^P(p_2)\big).$
\end{lemma}
Let us denote the totality of pairs of Der-operators $X_0^P,X_0^Q$ as in Lemma \ref{triolederivdegzero}, as $\mathcal{D}\mathrm{er}(g;P,Q).$ 
\subsubsection*{Atiyah-like sequences.} 
Der-operators are algebraic versions of \emph{derivations in vector bundles}, i.e. $\mathbb{R}$-linear map
$\nabla:\Gamma(E)\rightarrow \Gamma(E)$ satisfying $\Delta(fe)=\sigma_{\Delta}(f)e+f\Delta(e)$ for a unique vector field $\sigma_{\Delta}\in D(M).$
Derivations with zero symbol are simply  endomorphisms and so there is an exact sequence of Lie algebras,
$$0\rightarrow \mathrm{End}(E)\rightarrow \mathrm{Der}(E)\rightarrow TM\rightarrow 0.$$
Locally, let $x^i$ be coordinates on $M$ and $\epsilon^{\alpha}$ on $\Gamma(E)$ give a basis for $\mathrm{Der}_x(E)$ as $\big<\nabla_i,\epsilon_{\alpha}^{\beta}\big>$ defined on $e=f_{\alpha}\epsilon^{\alpha}$ by the relations
$\nabla_i(e):=\frac{\partial f_{\alpha}}{\partial x_i}(x)\epsilon^{\alpha}$ and $
\epsilon_{\alpha}^{\beta}(e):=f_{\alpha}(x)\epsilon^{\beta}.$
Sequences of the above form are called \emph{Atiyah sequence} of the bundle \cite{Ati}. These types of sequences appear in the description of diolic derivations in degree zero \cite{Diole1} and can be extended to the triolic situation as well.

To see this, note that we have an obvious symbol map given by projection onto the shared scalar symbol, $\sigma_0^1:D_1(\EuScript{T})_0\rightarrow D(A),$ sending such a degree $0$ derivation to its symbol $X_0^A=\sigma_0^1(X_0).$ This is well-defined for the reason that the components of the graded derivation $X_0^P,X_0^Q$ share the same symbol and we therefore refer to the derivation $X_0^A$ simply as the \emph{symbol} of the operator $X_0.$

\begin{lemma}
The $A$-module $\mathrm{ker}(\sigma_0^1)$ consists of pairs $(X,Y)$ where $X:P\rightarrow P$ is an endomorphism of $P$, and $Y:Q\rightarrow Q$ is an endomorphism of $Q,$ satisfying $Y\big(g(p_1,p_2)\big)=g\big(X(p_1),p_2\big)+g\big(p_1,X(p_2)\big).$
\end{lemma}
Denote the collection of such endomorphisms by
$$\mathcal{E}\mathrm{nd}\big(g;P,Q\big):=\{\text{pairs}, X_0^P\in \text{End}_A(P),X_0^Q\in \text{End}_A(Q)| X_0^Q\circ g=g\circ (X_0^P\otimes id_P+id_P\otimes X_0^P)\}.$$

\begin{theorem}[Triolic Atiyah Sequence]
Suppose that $P,Q$ are locally free projective modules of finite rank. Then we have a short exact sequence
\begin{equation}
\label{eqn:TrioleAtiSeq01}
    0\rightarrow \mathcal{E}\mathrm{nd}\big(g;P,Q\big)\rightarrow D_1(\EuScript{T})_0\xrightarrow{\sigma} D(A)\rightarrow 0.
    \end{equation}
\end{theorem}
The triolic Atiyah sequence in degree zero may be related to the diolic Atiyah sequence using the forgetful functor in remark \ref{Forgetful functor Triole to Diole remark}. In particular we have a commutative diagram
\[
\adjustbox{scale=.88}{
\begin{tikzcd}
0 \arrow[r,]& \mathcal{E}\mathrm{nd}\big(g;P,Q\big)\arrow[d,] \arrow[r,] & D_1(\EuScript{T})_0\arrow[d,] \arrow[r,"\sigma"] & D(A)\arrow[d, "id"]\arrow[r,] & 0
\\
0\arrow[r,] & \mathrm{End}(P) \arrow[r,] & D_1(\EuScript{A})_0\arrow[r, "\sigma"] & D(A)\arrow[r,] & 0.
\end{tikzcd}}
\]
In coordinates, if $G$ is the matrix of endomorphisms of $P$ while $H$ is the matrix of endomorphisms of $Q$, we find that in the geometric case, that $D_1(\EuScript{A})_0\cong \text{Der}(P)=D(A)\oplus \text{End}(P)\ni X=\mathbb{X}+G.$ The center vertical map forgets the endomorphism of $Q,$ as
$\mathbb{X}+(G,H)\mapsto \mathbb{X}+G.$

The description of derivations of an algebra of trioles in degree zero have an interesting interpretation and in order to give a simple characterization of the space of splittings of the sequence (\ref{eqn:TrioleAtiSeq01}), we need to exploit this interpretation.
We then will be able to discuss the local description of the space of splittings of the sequence (\ref{eqn:TrioleAtiSeq01}) (see below, §§\ref{EquationofFlatTriolicConnections}).
\subsubsection{Connections preserving inner structure}
\label{sssec:Bilinpreserveconnections}
Let $\pi:E\rightarrow M$ and $\eta:F\rightarrow M$ be two vector bundles over $M$ with modules of smooth sections given by $P$ and $Q$, respectively.
Let $\nabla$ be a linear connection in the bundle $\pi.$ 

The first example we intend to discuss concerns complex structures in (real) vector bundles. From an intuitive geometrical point of view such a structure assigns a complex structure to each fiber of the considered real vector bundle in a smooth manner. To this end it suffices to introduce a multiplication by the imaginary unit operator on each fiber. Having in mind the relationship between vector bundles and modules of their smooth sections this idea is formalized as follows. A \emph{complex structure} on an $A$-module $P$ is an endomorphism $J:P\rightarrow P$ such that $J^2=-\mathrm{id}_P.$ We call a complex structure on $P$ an \emph{inner complex structure} in $\pi$, when $P$ is the module of smooth sections of a vector bundle.

To avoid confusions, we explicitly mention that a complex structure on the tangent bundle of a smooth manifold $M$ is usually called an \emph{almost-complex structure} on $M$, because it does not necessarily come from some complex manifold structure on $M$.

A connection is \emph{structure-preserving} if any parallel transport of one point in the fiber to another is an identification of the corresponding structures they are supplied with. More precisely, if such a structure is an endomorphism $\varphi$ of the bundle $\pi$, then the restriction of this to a fiber say $\varphi_{m_1}:=\varphi|_{E_{m_1}},$ for $m_1\in M,$ is identified with $\varphi_{m_2}$ each time the fibers are identified via the parallel transport isomorphism.

This geometrically clear idea is not practical because it is not possible to check what happens along each parallel transport. Despite this, it means that the endomorphism $\varphi$ is ‘constant’ with respect to the considered connection $\nabla.$ Therefore, 
 $$\nabla_X^{\mathrm{End}}(\varphi)=0,\hspace{2mm} \forall X\in D(M),$$
 where $\nabla^{\mathrm{End}}$ is the induced linear connection in the endomorphism bundle.
.
 \begin{definition}
 \label{Connection preserve definition}
 Let $A$ be an algebra and $P$ an $A$-module supplied with a connection $\nabla$ and suppose that $\varphi\in\mathrm{End}(P).$ The connection is said to \textbf{\emph{preserve}} $\varphi$ if $\nabla_{X}^{\mathrm{End}}(\varphi)=0,$ for every $X\in D(M).$
 \end{definition}
 
 \begin{remark}
Definition \ref{Connection preserve definition} applies to a wide class of inner structures of interest in geometry. For instance, it applies to a linear connection and a complex structure $J$ on a vector bundle over $M$. 
\end{remark}

Lemma \ref{triolederivdegzero} suggests that we turn our attention to study linear connections in vector bundles which preserve the inner structure defined by fibre metrics.
To this end, let the $A$-module of bilinear-forms on $P$ be  $\text{Bil}(P;A)\cong (P\otimes P)^{\vee},$ wheere $P^{\vee}$ is the $A$-linear dual of $P.$
Given a linear connection $\nabla$ on $P$, may construct a unique linear connection $\nabla^{\text{Bil}}:=\big(\nabla\otimes \nabla\big)^{\vee}$ in $\mathrm{Bil}(P;A).$
This is simply the linear connection obtained by the induced connection in the tensor product and in the dual bundle.

Identify $\beta:P\otimes P\rightarrow A$ with a bilinear form $b$ via $b(p_1,p_2)=\beta(p_1\otimes p_2).$ 
Then one may easily compute, using the Der-Leibniz rule for the induced connection $\big(\nabla\otimes \nabla\big)^{\vee},$ that we have
\begin{equation}
    \label{bilinearformconnection}
\nabla_X^{\text{Bil}}(b)(p_1,p_2)=X\big(b(p_1,p_2)\big)-b\big(\nabla_X(p_1),p_2\big)-b\big(p_1,\nabla_X(p_2)\big).
\end{equation}

\begin{definition}
\label{bilinearformpreservingconnection}
A linear connection in $P$ is said to \textbf{\emph{preserve the bilinear form}} $b$ on $P$ if $\nabla_X^{\mathrm{Bil}}(p)=0,$ for each $X\in D(M).$
\end{definition}

Consequently, if $\nabla$ preserves the bilinear form, then the equation 
\begin{equation}
    \label{preservebilinform}
    X\big(b(p_1,p_2)\big)=b\big(\nabla_X(p_1),p_2\big)+b\big(p_1,\nabla_X(p_2)\big),
    \end{equation}
 holds. This notion is generalized to suit the language of triolic algebras as follows.

\begin{definition}
\label{Pairofconnectionspreservingg}
Let $g\in\mathrm{Bil}(P,Q)$. Let $\nabla$ be a linear connection in $P$ and let $\Delta$ be a linear connection in $Q.$
The pair $(\nabla,\Delta)$ are said to \textbf{\emph{preserve}} $g$, if 
$\Delta_X\big(g(p_1,p_2)\big)-g\big(\nabla_X(p_1),p_2\big)-g\big(p_1,\nabla_X(p_2)\big)=0,$
for every $X\in D(A).$
\end{definition}

We call a pair of connections $(\nabla,\Delta)$ $g$-preserving when the above holds.

\begin{lemma}
Any connection in $P$ which preserves a bilinear form $b:P\times P\rightarrow A$ as in Definition \ref{bilinearformpreservingconnection}, $\tilde{g}:P\times P\rightarrow Q$-preserving pair $(\nabla,D_X)$ in the trivial bundle $Q,$ where $D_X:=X$ is the standard trivial connection for each $X\in D(A).$
\end{lemma}

Consequently, we find a new interpretation of degree zero triolic derivations.

\begin{corollary}
A degree zero derivation $X_0=(X_0^A,X_0^P,X_0^Q)$ of an algebra of trioles $A\oplus(P,g)\oplus Q$ is equivalently the datum of a $g$-preserving pair of connections $(\nabla,\Delta)$.
\end{corollary}

To gain some more geometric intuition, we exploit the equivalent description of derivations in a vector bundle $E$ with infinitesimal automorphisms. That is, consider $\tilde{f}:F\rightarrow E$ a regular (ie. gives a vector space isomorphism when restricted to fibers) vector bundle map covering $f:N\rightarrow M$.
$\tilde{f}$ induces a $C^{\infty}(M)$-module homomorphism $\tilde{f}^*\colon P:=\Gamma(E)\rightarrow Q:=\Gamma(F).$
Note than any section $s\in\Gamma(E)$ may be pulled back to $\tilde{f}^*(s)\in \Gamma(F)$, defined by $(\tilde{f}^*s)(n):=\big(\tilde{f}|_{F_n}^{-1}\circ s\circ f\big)(n),$ for every $n\in N.$
Then $\tilde{f}$ induces a morphism of vector bundles, $\tilde{f}_*:\mathrm{Der}(F)\rightarrow \mathrm{Der}(E)$ given by $\tilde{f}_*\big(\Delta(s)\big):=\Delta(\tilde{f}^*s).$
Now, an automorphism of $E$ is in particular a regular vector bundle map covering a unique diffeomorphism $f:M\rightarrow M$ of the base. Understanding an \emph{infinitesimal automorphism} of $E$ as a vector field on $E$ whose flow is given by local automorphisms, we see that any $\Delta\in\mathrm{Der}(E)$ corresponds to an infintesimal automrphisms $Y_E\in D(E)$ with flow $\{\tilde{f}_t\}$ by setting $\Delta(e)=\frac{d}{dt}\big|_{t=0}\tilde{f}_t^*(e)$ for $e\in \Gamma(E).$
Suppose the bundle $E$ is equipped with a fibre metric $g$. The map $\tilde{f}^*$ defines infinitesimal symmetry of $\Gamma(\pi)$ in the sense of §§\ref{sssec:Symmetries}, which is to say, an element of $o(P,g)$  if $g\big(\tilde{f}^*(p_1),\tilde{f}^*(p_2)\big)=g(p_1,p_2).$ 
If this is a parameterized family of vector bundle maps, $\{\tilde{f}_t\}_{t\in I}$, which is moreover, a time-dependent family of automorphisms, we have that
$g\big(\frac{d}{dt}\tilde{f}_{t}^*(p_1),p_2\big)+g\big(p_1,\frac{d}{dt}\tilde{f}_t^*(p_2)\big)=0.$ 

\begin{lemma}
Let $X_0$ be a degree zero triolic derivation. An element of the $A$-submodule $\text{im}(g)\cap \ker(X_0^Q)$ of $Q$ is an infinitesimal symmetry of $Q.$
\end{lemma}

Consider now the situation when the triole algebra $\EuScript{T}$ is such that $Q$ is the module of sections of a trivial vector bundle of rank $1.$ In this way, $Q=A.$ In particular, the defining metric for $\EuScript{T}$ may be viewed as an element of $\mathrm{Bil}(P)$.

\begin{lemma}
Let $\nabla$ be a $b$-preserving linear connection in $P.$ Consider some other connection $\Delta$ in $P.$ Let $\omega:=\Delta-\nabla\in\Omega^1\big(\mathrm{End}(P)\big).$ Then $\Delta$ also preserved $b$ if and only if $\omega(X)\in \mathrm{o}(P,b)$ for each $X\in D(A).$
\end{lemma}

\begin{lemma}
The curvature of a $b$-preserving linear connection is an infinitesimal symmetry. That is $R^{\nabla}(X,Y)\in\mathrm{o}(P,b)$ for every $X,Y\in D(A).$
\end{lemma}
\subsubsection{Coordinates}
Let $X_0$ be a derivation of $\EuScript{T}$ of degree zero and consider sequence (\ref{eqn:TrioleAtiSeq01}).
The $A$-module $\mathcal{E}\mathrm{nd}\big(g;P,Q\big)$ has the following interpretation.
Identifying $\mathrm{Bil}(P,Q)$ with $P^{\vee}\otimes P^{\vee}\otimes Q,$ which is the module of sections of the rank $m_P\times m_P\times m_Q$-bundle $(E\otimes E)^{\vee}\otimes F$, which is the same as $\mathrm{Bil}(E)\otimes F,$ we may observe that $\mathcal{E}\mathrm{nd}(g;E,F)$ is the $A$-sub-module of $\mathrm{End}\big(\mathrm{Bil}(E,F)\big)$ consisting of $g$-preserving endomorphisms. In particular, this sub-module is projective and so $\mathcal{E}\mathrm{nd}\big(g;E,F\big)$ is the module of smooth sections of a vector bundle, of rank $(m_P^2\times m_Q)^2.$ Consequently, if $X_0\in D(\EuScript{T})_0$ the local splitting of the triolic Atiyah sequence implies that we may write
$$X_0:=\mathbb{X}+\mathbb{G},$$ where these are rank $m_{\mathcal{E}\mathrm{nd}}^2$ matrices.

Using these descriptions of degree zero derivations, one may easily generate new derivations from existing ones in various new triole algebras, for instance in the triolic product algebra.
\begin{lemma}
Let $X_0,X_0'$ be degree zero derivations in the triolic algebras $A\oplus (P,b)\oplus Q$ and $A\oplus (P',b')\oplus Q'$, respectively. Then there exists a unique degree zero derivation in the triolic product $\EuScript{T}\otimes^{\mathrm{T}}\EuScript{T}'.$
\end{lemma}
This follows from the fact that given two linear connections $\nabla,\nabla'$ in $A$-modules $P,P'$ there exists a unique linear connection in $P\otimes_A P'$ and so the local description of the induced derivation in $\EuScript{T}\otimes^{\mathrm{T}}\EuScript{T}'$ is provided as follows.
Let $X_0^P(e_{\alpha})=\Gamma_{\alpha}^{\beta}e_{\beta},Y_0^{P'}=\tilde{\Gamma}_{\alpha'}^{\beta'} e_{\beta'}$ and $X_0^Q(\epsilon_A)=\Upsilon_{A}^B\epsilon_B,Y_0^{Q'}(\epsilon_{A'})=\tilde{\Upsilon}_{A'}^{B'}\epsilon_{B'}.$ The corresponding derivation $Z_0=(Z_0^A,Z_0^{P\otimes P'},Z_0^{Q
\otimes Q'})\in D(\EuScript{T}\otimes^{\mathrm{T}}\EuScript{T}')_0,$ is specified by
\begin{itemize}
    \item $Z_0^{P\otimes P'}(e_{\alpha}\otimes e_{\alpha'})=\big[\Gamma^{\otimes}\big]_{\alpha\alpha'}^{\beta\beta'}e_{\beta}\otimes e_{\beta'},$ where
    $[\Gamma^{\otimes}]_{\alpha\alpha'}^{\beta\beta'}=\Gamma_{\alpha}^{\beta}\otimes \delta_{\alpha'}^{\beta'}+\delta_{\alpha}^{\beta}\otimes \tilde{\Gamma}_{\alpha'}^{\beta'},$
    
    \item $Z_0^{Q\otimes Q'}(\epsilon_{A}\otimes \epsilon_{A'})=\big[\Upsilon^{\otimes}\big]_{AA'}^{BB'}\epsilon_B\otimes \epsilon_{B'},$ where 
    $[\Upsilon^{\otimes}]_{AA'}^{BB'}=\Upsilon_{A}^{B}\otimes \delta_{A'}^{B''}+\delta_{A}^{B}\otimes \tilde{\Upsilon}_{A'}^{B'},$
    
    \item These components are related via $Z_0^{Q\otimes Q'}\big((p_1\otimes p_1')\cdot (p_2\otimes p_2')\big)=X_0^Q\big(g(p_1,p_2)\big)\otimes \tilde{g}\big(p_1',p_2'\big)+g\big(p_1,p_2\big)\otimes Y_0^Q\big(\tilde{g}\big(p_1',p_2')\big).$
\end{itemize}

\subsubsection{Degree $1$ derivations.}
\label{sssec:Degree1Derivations}
We now continue our description of the graded derivations of $\EuScript{T}.$ To achieve a simple description, we present a novel definition which generalizes the usual notion of Der-operator to that of vector-valued Der-operators, whose symbol is twisted by some bilinear form. We would like to emphasize, that such a definition is not obviously obtainable by attempting to straightforwardly generalize the concept of a Der-operator to that of a Der-operator between two different bundles; however, this appropriate generalization appears naturally as a class of well-defined ordinary derivations of the algebra $\EuScript{T}.$  
\begin{definition}
\label{GeneralizedDerOperatorDefinition}
Let $E\rightarrow M$ and $F\rightarrow M$ be two vector bundles of ranks $m_P,m_Q$ with modules of sections $P,Q,$ respectively
Suppose that $\gamma:\Gamma(E)\otimes \Gamma(E)\rightarrow \Gamma(F)$ is a non-degenerate vector valued fiber metric.
An $\mathbb{R}$-linear operator $\Box:P\rightarrow Q$ is said to be a \textbf{\emph{Der-operator for}} $\gamma:P\otimes P\rightarrow Q$, if
\begin{itemize}
    \item It is an additive operator,
    
    \item It satisfies the Der-Lebiniz-like rule 
    \begin{equation}
        \label{VectorvaluedGammaDerLeibniz}
    \Box(ap)=\gamma\big(X(a)\otimes p\big)+a\Box(p),
    \end{equation}
for $X$, some $P$-valued derivation of $A,$ called the \textbf{\emph{vector-valued symbol}}, with $p\in P,a\in A.$
\end{itemize}
\end{definition}
The relation (\ref{VectorvaluedGammaDerLeibniz}) is called the $Q$-\emph{valued} $\gamma$-\emph{Der-Leibniz rule}, or simply the $\gamma$-\emph{Der-Leibniz} rule when $Q$ is understood.
We may recast this equation in a form which resembles the ordinary Der-Leibniz rule, by setting $X(a)\cdot_{\gamma}p:=\gamma\big(X(a)\otimes p\big)\in Q.$
Note that since $\gamma$ is non-degenerate, the quantity $X(a)\cdot_{\gamma}p$ is identically zero if and only if $X(a)\equiv 0$ for each $a\in A.$ As this is a $P$-valued derivation, such a quantity vanishes if and only if the coefficients determining $X$ are trivial. 

\begin{lemma}
\label{Degree1TriolicDerivations}
A degree $1$ triolic derivation is a pair $X_1=\big(X_1^A,X_1^P\big)$ of operators with $X_1^A \in D(A,P)$ and with $X_1^P:P\rightarrow Q,$ an element of $\mathrm{Diff}_1(P,Q)$ which satisfies $X_1^P(1)=0,$ and  $X_1^{P}(ap)=g\big(X_1^A(a),p\big)+aX_1^P(p),$ for all $a\in A,p\in P.$  
\end{lemma}
We will denote the collection of operators $Z:R\rightarrow S$ between two $A$-modules subject to the above Leibniz rule whose symbol is twisted by $g$ as $\mathrm{Der}^g(R,S)\subset \mathrm{Diff}_1(R,S).$ It is clear that $\mathrm{Der}^g(R,S)$ is a left $A$-module in the obvious way, due to the $A$-bilinearity of $g.$

Suppose now that $\EuScript{T}$ is regular. This suggests the existence of a generalized symbol map, since for $X_1\in D_1(\EuScript{T})_1$ we may consider the projection map 
$\sigma_1^1:D_1(\EuScript{T})_1\rightarrow D(P)$ given by $\sigma_1^1(X_1):=X_1^A.$ 
It follows from the regularity of $\EuScript{T}$ that we may identify the image with $D\big(A,\mathrm{Hom}_A(P,Q)\big)$, via the adjoint morphism $\mathfrak{g}.$ This determines a map, denoted the same
$$\sigma_1^1:D_1(\EuScript{T})_1\rightarrow D\big(A,\mathrm{Hom}(P,Q)\big).$$

\begin{lemma}
Let $\EuScript{T}$ be a regular triole algebra. Then there is an isomorphism of $A$-modules $\ker(\sigma_1^1)\cong \mathrm{Hom}(P,Q).$
\end{lemma}

\subsubsection{Coordinates}
Let us determine the local nature of elements of the $A$-module $\mathrm{Der}^g(P,Q).$ Let $\epsilon_B$ be a basis of $Q$ and $e_{\alpha}$ that of $P.$

Acting on an element $p=p^{\alpha}e_{\alpha}$ of $P$, with coefficient functions $p^{\alpha}\in A,$ and using the additivity and $Q$-valued $g$-Der Leibniz rule above, we see 
$\Box(p^{\alpha}e_{\alpha})=X(p^{\alpha})\cdot_{g}e_{\alpha}+a\Box(e_{\alpha}).$
Let the action of $\Box$ in the basis be specified as $\Box(e_{\alpha})=h_{\alpha}^B\epsilon_B,$ for functions $h_{\alpha}^{B}\in A.$
Then since $X\in D(P)$, we write the action $X(p^{\alpha})=\overline{X}^{\beta}(p^{\alpha})e_{\beta},$ where $\overline{X}^{\beta}\in D(A)$.
Consequently,
the Leibniz rule reads
$\sum_{\beta=1}^{m_P}\overline{X}^{\beta}(p^{\alpha})e_{\beta}\cdot_{g}e_{\alpha}+p^{\alpha}h_{\alpha}^B\epsilon_B.$
If we write $g_{\alpha\beta}^B$ for the functions determining $g$, we find that the above expression reads
$$\sum_{\alpha\beta}\sum_B g_{\beta\alpha}^B\overline{X}^{\beta}(p^{\alpha})\epsilon_B+p^{\alpha}h_{\alpha}^B\epsilon_B.$$

\begin{lemma}
Suppose that $\EuScript{T}$ is a regular triole algebra. Then there exists an exact sequence of $A$-modules
\begin{equation} 
 \label{eqn:TrioleAtiSeq11}
0\rightarrow \mathrm{Hom}_A(P,Q)\rightarrow D(\EuScript{T})_1\rightarrow D\big(A,\mathrm{Hom}(P,Q)\big).
\end{equation}
\end{lemma}
\begin{proof}
We wish to show, in the case the $g$ is non-degenerate, that we have an exact sequence
$0\rightarrow \text{Hom}_A(P,Q)\rightarrow D(\EuScript{T})_1\rightarrow D(A,\text{Hom}_A(P,Q))\rightarrow 0.$
It is enough to show that the rightmost map is surjective. 
To this end, let $\theta\in D(A,\text{Hom}_A(P,Q))_x$ with $x\in X.$
We must show that there exists a first-order differential operator $\Box_x:P\rightarrow Q$ near $x$ such that $\sigma_{1,x}(\Box_x)=\theta.$
Let $(U,x_1,..,x_{\ell},x_{\ell+1},..,x_{\ell+n})$ be a holomorphic chart around $x$ and let $s=(s_1,..,s_{\ell})$ and $t=(t_1..,t_{m_Q})$ be holomorphic frames of $P,Q.$
Using the representability $\text{Hom}_A\big(\Omega^1(A),\text{Hom}_A(P,Q)\big)=D(A,\text{Hom}_A(P,Q)\big)$ we may view $\theta$ equivalently as a homomorphism say $h_X:\Omega^1(A)\rightarrow \text{Hom}_A(P,Q).$ Since $(dz_{\alpha}| 1\leq \alpha \leq \ell\}$ is an $A$-modue basis of the Kahler forms, there exists necessarilly unique functions $b_{ij}^{\alpha}$ with $1\leq i\leq m_Q, 1\leq j\leq P, 1\leq \alpha\leq\ell$ such that $h_X(dz_{\alpha})(s_j)=\sum_i^{m_Q}b_{ij}^{\alpha}t_i.$
Then simply define $\Box:P\rightarrow Q$ by $\Box(\sum_j^{m_P}f_js_j)=\sum_{ij\alpha}b_{ij}^{\alpha}\frac{\partial f_j}{\partial z_{\alpha}}t_i.$
The bracket $[\Box,f]$ is $A$-linear. Letting $V\subset U$ with $A_V$ the restriction of functions to $V$ (ie. $\mathscr{O}_X(V))$ and for $\omega\in \Omega_V^1,$ given as $\omega=\sum_{\alpha=1}^{\ell}\omega_{\alpha}dz_{\alpha}$, by construciton of the symbol map and the unversality of the pair $(\Omega^1,d)$ we have
$\sigma_1(\Box)_V(\omega)=\sum_{\alpha}\omega_{\alpha}\sigma_1(\Box)_V(dz_{\alpha})=\sum_{\alpha}\omega_{\alpha}[\Box,z_{\alpha}].$
By definition of $\Box$ we get $\Box(s_j)=0$ and conseuently the final commutator reads $\sum_i b_{ij}^{\alpha}t_i.$ The latter coincides with $h_X(dz_{\alpha})(s_j).$
We thus find $[\Box, z_{\alpha}]=h_X(dz_{\alpha})$ and so 
$\sigma_1(\Box)_V(\omega)=\sum_{\alpha}\omega_{\alpha}h_X(dz_{\alpha}).$
Thus $\sigma_1(\Box)_U=h_X$ which is to say that $(\sigma_1)_x(\Box_x)=0.$

\end{proof}

We will refer to the sequence (\ref{eqn:TrioleAtiSeq11}) as the \emph{triolic Atiyah sequence of degree} $1$. Any degree $1$ derivation of $\EuScript{T}$ may be written as
$\Box_1:=||(\mathbb{X}_g)_{\alpha}^A||_{\alpha=1,..,m_P}^{A=1,..,m_Q}+||h_{\alpha}^A||_{\alpha=1,..,m_P}^{A=1,..,m_Q},$ where the action of the first term on $(p^1,\ldots,p^{m_P})^T\in P$ is defined by 
$$(\mathbb{X}_g)_{\alpha}^A(p^{\alpha}):=\sum_{\beta}^{m_P}g_{\beta\alpha}^A \overline{X}^{\beta i}\frac{\partial}{\partial x_i}p^{\alpha}.$$
We generically write this decomposition for $X_1$ in $D(\EuScript{T})_1$ as $\mathbb{X}_{g}+\mathbb{H}$.

\begin{example}
Suppose that $P,Q$ are rank $2$ bundles. The derivation $X\in D(P)$ is given by $X(f)=\big(\overline{X}^1(f),\overline{X}^2(f)\big)^T$ for $\overline{X}^i\in D(M).$ Then the action $\Box_1(p^1, p^2)^T$ reads in explicit matrix notation as
$$
\scalemath{.88}{\begin{bmatrix}
g_{11}^1\overline{X}^1(p^1)+g_{21}^1\overline{X}^2(p^1)+g_{12}^1\overline{X}^1(p^2)+g_{22}^1\overline{X}^2(p^2)
\\
g_{11}^1\overline{X}^1(p^1)+g_{21}^1\overline{X}^2(p^1)+g_{12}^2\overline{X}^1(p^2)+g_{22}^2\overline{X}^2(p^2)
\end{bmatrix} + \begin{bmatrix} 
h_1^1p^1+h_2^1p^2
\\
h_1^2p^1+h_2^2p^2
\end{bmatrix}}.$$
\end{example}
Even more explicitly, let $e_{\alpha}$ be a basis for $P$ and $\epsilon_B$ be that for $Q.$ Then there exists a basis of $\mathrm{Hom}_A(P,Q)$ given by $\{\psi_B^{\alpha}\}_{\alpha,B}$ defined by $\psi_B^{\alpha}(e_{\beta}):=\delta_{\beta}^{\alpha}\epsilon_B.$
Then $\mathbb{X}_g$ is defined as follows. Take $X(f)=X^{\alpha}(f)e_{\alpha}$, where $X^{\alpha}\in D(A).$ Then $\mathfrak{g}\big(X(f)\big)=X^{\alpha}(f)\mathfrak{g}(e_{\alpha})=X^{\alpha}(f)(\mathfrak{g}_{\alpha})_B^{\alpha'}\psi_B^{\alpha'}\in\mathrm{Hom}_A(P,Q).$
Similarly $\mathbb{G}$ is an $m_Q\times m_P$-matrix of functions.

\subsubsection{A relation with diolic derivations.}
The language of triolic algebras is perhaps more satisfactory than that of diolic algebras, as it conceptualizes various constructions occurring in the diolic formalism. Namely, those degree zero derivations of a diole algebra with values in a truncated diolic module have a natural description in terms of derivations in $\EuScript{T}.$
\begin{lemma}
Let $\EuScript{A}=A\oplus P$ be an algebra of dioles and let $\EuScript{R}=\big(R_0\oplus R_1,\varphi\big)$ be a truncated diolic module. Suppose that this module is given by $R_0=P$ and $R_1=Q.$
Then any degree zero diolic derivation with values in $\EuScript{R}$, is in particular, a degree $1$ derivation of $\EuScript{T}.$
\end{lemma}
\begin{proof}
Recall from \cite{Diole1} that any $\EuScript{R}$-valued derivation of $\EuScript{A}$ is given by a pair of operators $X_0^A\in D(A,R_0)$ and $X_0^P\in \mathrm{Diff}_1(P,R_1)$ which satisfy the generalized $\varphi$-Der-Leibniz rule $X_0(ap)=\varphi(X_0^A(a)\otimes p)=aX_0^P(p),$ where $\varphi:P\otimes_A R_0\rightarrow R_1,$ is the $A$-module homomorphism arising from the definition of a truncated diolic module.
Then, in particular for $R_0=P$ and $R_1=Q,$ we see $\gamma:=\varphi:P\otimes P\rightarrow Q$ and we find that a degree $1$ triolic derivation is a pair $X_1=(X_1^A,X_1^P)$ and $X_1^A\in D(A,P)$ with $X_1^P\in \mathrm{Diff}_1(P,Q)$ such that $X_1(ap)=\gamma(X_1^A(a)\otimes p)+aX_1^P(p).$
\end{proof}

\subsubsection{Other admissable derivations}
We now present the description of the other derivations which exists in the formalsim of differential calculus over an algebra of trioles.

\begin{lemma}
Degree $2$ triolic derivations coincide with $Q$-valued derivations of $A.$
\end{lemma}

\begin{lemma}
Degree $-1$ triolic derivations are given by pairs $X_{-1}=(X_{-1}^P,X_{-1}^Q)$ of $A$-linear operators with $X_{-1}^P\in P^*=\mathrm{Hom}_{A}(P,A)$ with $X_{-1}^Q\in \mathrm{Hom}_{A}(Q,P)$ satisfying the relation $X_{-1}^Q\big(g(p_1,p_2)\big)=X_{-1}^P(p_1)p_2-p_1X_{-1}^P(p_2).$
\end{lemma}
Let us comment on the above equation. 
Let $g_{\alpha\beta}^a$ be coefficients determining $g,$ and let $X_{-1}^P=(\Phi^a), X_{-1}^Q=(\Psi_a^{\alpha}),$ and let $u_a,v_{\alpha}$ be basis of $P,Q$ respectively.
We find that 
$$\Psi_a^{\alpha}g_{\alpha\beta}^a=(n-1)\Phi_a.$$

\begin{corollary}
When $P$ is the module of sections of a line bundle and $X_{-1}$ is a degree $-1$ triolic derivation, we have $\text{im}(g)\subset \ker(X_{-1}^Q).$ 
\end{corollary}

\begin{lemma}
There are no degree $-2$ triolic derivations.
\end{lemma}
\begin{proof}
This is easily established by contradiction. Let $X_{-2}$ be such a derivation, so it amounts to a map $X_{-2}^Q:Q\rightarrow A$ satisfying the graded Leibniz rule.
If we compute the quantity
$X_{-2}(q_1q_2),$ which is identically zero, we see 
$$0\equiv X_{-2}^Q(q_1)q_2+(-1)^{-4}q_1X_{-2}^Q(q_2)=X_{-2}^Q(q_1)q_2+q_1X_{-2}^Q(q_2).$$
By choosing a basis for $Q$ and letting $X_{\alpha}^{\gamma}$ denote coefficients for $X_{-2}^Q,$
we find by this equation yields
$(n+1)X_{\alpha}^{\gamma}=0,$
for $n$ the rank of $Q.$
We have then that all such derivations have coefficients identically zero, thus all derivations of degree $-2$ are trivial, since having negative values for $n$ is not possible. 
\end{proof}
\subsubsection{Module and other algebraic structures}
We will now demonstrate that $D_1(\EuScript{T})_{\mathcal{G}}$ is a left $\EuScript{T}$-module and explicitly describe the module structures. These results culminate in Lemma \ref{TrioleDerivationsareaTriolicModule} below.

Let us first specify the module action $\lambda_1:P\otimes D(\EuScript{T})_1\rightarrow D(\EuScript{T})_2$ and then provide both $\lambda_0:P\otimes D(\EuScript{T})_0\rightarrow D(\EuScript{T})_1$ and $\nu:Q\otimes D(\EuScript{T})_0\rightarrow D(\EuScript{T})_2.$ To this end, the map
    $\lambda_1:P\otimes D(\EuScript{T})_1\rightarrow D(\EuScript{T})_2,$
    is defined by $\lambda_1(p,X_1):=(pX_1),$ where $(pX_1)$ is the degree $2$ derivation which is defined by its action on $a\in A$ by
    $$\big(pX_1)(a):=\mathfrak{g}(p)\bigg(X_1^A(a)\bigg).$$
    
   To see that this is indeed an element of $D(\EuScript{T})_2\cong D(Q)$, we show it satisfies the Leibniz rule. Since degree $2$ derivations have only an $A$-component, write $(pX_1)\equiv (pX_1)_2^A$ to denote the restriction to $A\subset\EuScript{T}.$
    Then compute
    $$(pX_1)_2^A(ab):=\mathfrak{g}(p)\circ \bigg(X_1^A(ab)\bigg)=\mathfrak{g}(p)\big(X_1^A(a)b+aX_1^A(b)\big)=\mathfrak{g}(p) \big(X_1^A(a)\big)b+a\mathfrak{g}(p)\big(X_1^A(b)\big),$$
    using the $A$-bilinearity of $g$. 
    
    \begin{remark}
    Note we could have rewritten this using the definition of $\mathfrak{g}$ as 
    $(pX_1)(a):=p\cdot X_1^A(a)=g\big(p,X_1^A(a)\big),$ since $X_1^A(a)\in P$ and $p\cdot P\equiv g(p,P)\in Q.$
    \end{remark}

Moving forward, note that the map $\nu:Q\otimes D(\EuScript{T})_0\rightarrow D(\EuScript{T})_2$ acts by multiplication by $q.$ Namely $(qX_0)_2^A(a):=q\cdot X_0^A(a).$ This is a $Q$-valued derivation of $A$ since $X_0^A$ is an ordinary derivation of $A$.

Finally let $\tilde{p}\in P.$ Take $X_0$ a degree $0$ triolic derivation. We claim that the map $\lambda_0(\tilde{p}\otimes X_0):=(\tilde{p}X_0)$ produces a degree $1$ triolic derivation.
That is $(\tilde{p}X_0)$ has components $(\tilde{p}X_0)_1^A,(\tilde{p}X_0)_1^P$ which are related by 
$(\tilde{p}X_0)_1^P(ap)=g\big((\tilde{p}X_0)_1^A,p\big)+a(\tilde{p}X_0)_1^P(p).$
We find a well-defined module structure by setting
$$
\big(\tilde{p}X_0\big)_1^A(a):=\tilde{p}\cdot X_0^A(a),\hspace{10mm}
\big(\tilde{p}X_0\big)_1^P(p):=\mathfrak{g}(\tilde{p})\big(X_0^P(p)\big).$$
This defines the correct notion of a degree $1$ derivation since we compute the Leibniz rule as 
\begin{eqnarray*}
\big(\tilde{p}X_0\big)_1^P(ap)&:=&\mathfrak{g}(\tilde{p})\big(X_0^P(ap)\big)=\mathfrak{g}(\tilde{p})\bigg(X_0^A(a)p+aX_0^P(p)\bigg)=g\bigg(\tilde{p},X_0^A(a)p+aX_0^P(p)\bigg)
\\
&=&g(\tilde{p},X_0^A(a)p)+ag\big(\tilde{p},X_0^P(p)\big),
\end{eqnarray*}
using the fact that $X_0^P\in \text{Der}(P)$ with symbol $\sigma(X_0^P)=X_0^A.$ The right hand side is precisely $$X_0^A(a)\tilde{p}\cdot p+a \mathfrak{g}(p)\big(X_0^P(p)\big)\equiv \tilde{p}X_0^A(a)\cdot p+a\big(\tilde{p}X_0\big)_1^P(p)\equiv g\big((\tilde{p}X_0)_1^A,p\big)+a(\tilde{p}X_0)_1^P(p),$$
as required.
\begin{lemma}
\label{TrioleDerivationsareaTriolicModule}
Let $D(\EuScript{T})_{+}$ be the sub-module of non-negatively graded derivations of $\EuScript{T}.$ This is a truncated triolic module.
\end{lemma}
\begin{proof}
We need to verify that the compatibility equation \ref{eqn:simptrunctrimod} holds for any $X_0\in D_1(\EuScript{T})_0,$ and any $p_1,p_2\in P.$
This will be an equality between elements of $D_1(\EuScript{T})_2$. 
Write $\lambda_0(p_2,X_0):=Y_1$ whose components, as above are $Y_1^A=p_2X_0^A$ and $Y_1^P=\mathfrak{g}(p_2)\circ X_0^P.$ Then denote $\lambda_1\big(p_1,Y_1\big):=Z_2.$ By construction this is a $Q$-valued derivation whose $A$-component is $Z_2^A=\mathfrak{g}(p_1)\circ Y_1^A.$
This is the left hand side of the compatibilty relation. The right is simply the $Q$-valued $\tilde{Z}_2$ where $\tilde{Z}_2^A:=g(p_1,p_2)\cdot X_0^A,$ by definition of $\eta.$
It remains to check that $Z_2$ and $\tilde{Z}_2$ coincide as $Q$-valued derivations of $A.$
To see this one computes $Z_2(ab),$ for any $a,b\in A.$
We find
\begin{eqnarray*}
Z_2^A(ab)&=& \mathfrak{g}(p_1)\circ Y_1^A(ab)= \mathfrak{g}(p_1)\circ \big(p_2\cdot X_0^A(ab)\big)= \mathfrak{g}(p_1)\circ \big(p_2\cdot X_0^A(a)b+p_2\cdot aX_0^A(b)\big)
\\
&=&\mathfrak{g}(p_1)\circ \big(p_2\cdot X_0^A(a)\big)b+ \mathfrak{g}(p_1)\circ \big(p_2\cdot X_0^A(b)\big)a
\\
&=&\mathfrak{g}(p_1)(p_2)X_0^A(a)b+g_{p_1}^{\sharp}(p_2)X_0^A(b)a
\\
&=&\mathfrak{g}(p_1)(p_2)\cdot X_0^A(ab)
\\
&=&\big(g(p_1,p_2)\cdot X_0^A\big)(ab)
\\
&=&\tilde{Z}_2^A(ab).
\end{eqnarray*}

\end{proof}

\subsubsection{Triolic Lie Bracket}
Degree counting tells us the only admissible graded Lie brackets, with the bracket as usual
$[X_i,Y_j]:=X_i\circ Y_j-(-1)^{ij}Y_j\circ X_i,$ are given by the pairs of operators
$$[X_0,Y_{-1}],\hspace{3mm} [X_0,Y_0],\hspace{3mm}  [X_0,Y_1],\hspace{3mm} [X_0,Y_2],\hspace{3mm} [X_1,Y_{-1}],\hspace{3mm} [X_1,Y_1],\hspace{3mm} [X_2,Y_{-1}].$$
Computing the admissible commutators is straightforward and we present only a few results in this direction, leaving the remaining commutators as an exercise for the reader. 
\begin{remark}
We employ the following notation. For say, $[X_0,Y_0]$ which is again a degree $0$-graded derivation, thus is determined by three components given by restriction to $A,P,Q$, we write 
$\big([X_0,Y_0]\big)_0^A:=[X_0,Y_0]\big|_{A\subset\mathscr{T}},$
for the resulting $A$-component. Similarly for $P,Q$ components.
\end{remark}
    The main result we give here as a result of these calculations is the following one.
\begin{lemma}
The pair $\big(D(\EuScript{T})^+,[-,-]\big)$ is a triolic Lie algebra.
\end{lemma}
\begin{proof}
Denote by $\mathrm{ad}_i:D(\EuScript{T})_0\rightarrow \mathrm{End}\big(D(\EuScript{T})_i\big)$ the map sending a degree zero derivation $X_0$ to the graded endomorphism $[X_0,-]|_{D(\EuScript{T})_i}$ for $i=1,2,$ (i.e. the adjoint operator). By Lemma \ref{Degree zero commutators}, we have that the the degree zero component of $D(\EuScript{T})^+$ is a Lie algebra. It suffices to verify that the defining relation $\mathrm{ad}_2(X_0)\big<Y_1,Z_1\big>=\big<\mathrm{ad}_1(X_0)Y_1,Z_1\big>+\big<Y_1,\mathrm{ad}_1(X_0)Z_1\big>$ holds. This follows from the Jacobi-identity for the graded Lie algebra $[-,-].$
\end{proof}
Some explicit details are given now, as promised.
\begin{lemma}
\label{Degree zero commutators}
Let $X_0=(X_0^A,X_0^P,X_0^Q),Y_0=(Y_0^A,Y_0^P,Y_0^Q)\in D_1(\EuScript{T})_0,$ be degree zero triolic derivations. Then 
$Z:=[X_0,Y_0]=\big(Z_0^A,Z_0^P,Z_0^Q)$
is the degree zero derivation determined by its components given as
\begin{enumerate}
    \item $Z_0^A=[X_0^A,Y_0^A],$ which is to say that the symbol of the commutator is the commutator of symbols,
    \item 
    $Z_0^P=[X_0^P,Y_0^P],$ the Der-operator in $P$, is the commutator of Der-operators in $P$ and satisfies the Der-Leibniz rule $Z_0(ap)=Z_0^A(a)p+aZ_0^P(p),$ 
    
    \item $Z_0^Q=[X_0^Q,Y_0^Q]$ the Der-operator in $Q$, is the commutator of Der-operators in $Q$ and satisfies $Z_0^Q\big(g(p_1,p_2)\big)=g\big(Z_0^P(p_1),p_2\big)+g\big(p_1,Z_0^P(p_2)\big).$
\end{enumerate}
In particular, the pair $\big(D_1(\EuScript{T})_0,[-,-]\big)$ is a Lie algebra.
\end{lemma}

\begin{lemma}
The degree $1$ derivation $Z_1=[X_0,Y_1]$ has components $Z_1^A=X_0^P\circ Y_1^A-Y_1^A\circ X_0^A$ and $Z_1^P=X_0^Q\circ Y_1^P-Y_1^P\circ X_0^P$ which moreover satisfy 
$Z_1^P(ap)=g\big(Z_1^A(a),p\big)+a Z_1^P(p),$
which says, as it should be, that $\sigma\big([X_0,Y_1]\big)=\mathfrak{g}\circ Z_1^A.$
\end{lemma}

\begin{lemma}
The quantity $[X_0,Y_2]$ is the $Q$-valued derivation of $A$ determined by $\big([X_0,Y_2]\big)_2^A=X_0^Q\circ Y_2^A-Y_2^A\circ X_0^A.$
\end{lemma}

In local coordinates, the $P$-valued derivation of $A$ is given by $Z_1^A:A\rightarrow P$ defined as $Z_1^A(f)=\sum_{\alpha=1}^m\overline{Z}^{\alpha}(f)e_{\alpha}$ with $\overline{Z}^{\alpha}\in D(A)$ is prescribed by 
$\overline{Z}^{\alpha}(f)=[X,\overline{Y}^{\alpha}] + \sum_{\beta}^mg_{\beta}^{\alpha}\overline{Y}^{\beta}(f),$ and similarly for $Z_1^P$. The commutator $[X_0,Y_2]$ is described by letting $X_0^Q(f\epsilon_A)=X^i\partial_i(f)\epsilon_A+fh_A^B\epsilon_B$ denote the action of the the $\mathrm{Der}(Q)$-component of $X_0\in D(\EuScript{T})_0.$
Then
$[X_0,Y_2](f)=\sum_{i=1}^n\big(X^i\overline{\Box}^{B}(f)-\overline{\Box}^BX^i(f)\big)\epsilon_B+\sum_{A,B}h_A^B\overline{\Box}^A(f)\epsilon_B,$ which is to say $[X_0,Y_2](f)=[X,\overline{\Box}^B](f)\epsilon_B+ h_A^B\cdot \overline{\Box}^A(f)\epsilon_B.$
\subsubsection{Module-valued derivations}
For completeness, we describe those graded derivations of $\EuScript{T}$ with values in a truncated triolic module. To this end, let $\EuScript{R}\in \tau\mathrm{Mod}(\EuScript{T})$.

\begin{lemma}
Degree $0$ triolic derivations with values in $\EuScript{R}$ are given by triples $X_0=(X_0^A,X_0^P,X_0^Q)$ with $X_0^A\in D(A,R_0), X_0^P\in \mathrm{Der}^{\lambda_0}(P,R_1), X_0^Q\in \mathrm{Der}^{\nu}(Q,R_2)$ which satisfy the relation
$ X_0^Q\big(g(p_1,p_2)\big)=\lambda_1\big(X_0^P(p_1),p_2\big)+\lambda_1\big(p_1,X_0^P(p_2)\big).$
\end{lemma}
Here we are employing the usual notation for generalized Der-operators, as in Definition \ref{GeneralizedDerOperatorDefinition}. Moreover, in the same way we wrote $\mathfrak{g}:=g(p,-)\in \mathrm{Hom}(P,Q)$ for the adjoint map to $g:P\times P\rightarrow Q$, we will write, following the notation of Remark \ref{EaseofNotation}, $\lambda_1^{\sharp}$ to denote the associated map $\lambda_1^{\sharp}\in\mathrm{Hom}_A(R_1,R_2),$ obtained by fixing the first entry from $p\in P$. We do the same also for the other module structure maps arising in Definition \ref{TruncatedTriolicModuleDefinition}.

\begin{remark}
When $\EuScript{R}=\EuScript{T}$, the module structure $\lambda_1$ coincides with the defining metric $g$ of $\EuScript{T}.$
\end{remark}

\begin{lemma}
Graded $\EuScript{R}$-valued derivations of $\EuScript{T}$ of degree $1$ are given by pairs of operators $X_1=\big(X_1^A,X_1^P\big)$ with $X_1^A\in D(A,R_1)$ and with $X_1^P\in\mathrm{Der}^{\lambda_1}\big(P,R_2\big)$.
That is, $X_1(ap)=\lambda_1\big(X_1^A(a),p\big)+aX_1^P(p)\in R_2$ for $a\in A,p\in P.$
Similarly, 
$\EuScript{R}$-valued derivations of $\EuScript{T}$ of degree $2$ are $R_2$-valued derivations of the algebra $A.$
\end{lemma}

\subsubsection{Triolic Der-operators}
Fix here $\mathscr{R}\in\tau\mathrm{Mod}(\mathscr{T}).$ We will now study $\mathrm{Der}(\mathscr{R})_{\mathcal{G}}.$ In homogeneous degree $g$ these are those operators $\nabla_g\in \mathrm{Der}(\mathscr{R})_g$, which are  $\mathbb{K}$-linear graded operators 
 $\nabla_g:
 \mathscr{R}\rightarrow\mathscr{R}$ whose graded symbol $\sigma_{\nabla_g}$ is a degree $g$ graded derivation of the triolic algebra.
It is easy to see that the totality of all graded Der-operators gives a graded $\mathscr{T}$-module with respect to the left multiplication by elements of $\mathscr{T}$ and a graded $\mathbb{K}$-Lie algebra with respect to the graded commutator operation. We describe such operators for a the truncated triole module $\mathscr{R}.$

 \begin{lemma}
 \label{DegZeroTriolicDerOperators}
A degree zero graded triolic der operators $\nabla_0\in \mathrm{Der}(\mathscr{R})_0$ is a triple $\nabla_0=\big(\nabla_0^{R_0},\nabla_0^{R_1},\nabla_0^{R_2}\big)$ where $\nabla_0^{R_0}\in \mathrm{Der}(R_0),\nabla_0^{R_1}\in \mathrm{Der}(R_1),\nabla_0^{R_2}\in\mathrm{Der}(R_2)$ are ordinary Der-operators with equal scalar symbol, $\sigma_{\nabla_0^{R_i}}=X_0^A,$ for each $i=0,1,2$ subject to the following relations, 
\begin{eqnarray*}
\nabla_0^{R_1}(pr_0)&=&X_0^P(p)r_0+p\nabla_0^{R_0}(r_0),
\\
\nabla_0^{R_2}(pr_1)&=&X_0^P(p)r_1+p\nabla_0^{R_1}(r_1),
\\
\nabla_0^{R_2}\big(g(p_1,p_2)r_0\big)&=&\big[g\big(X_0^P(p_1),p_2\big)+g\big(p_1,X_0^P(p_2)\big)\big]r_0+g(p_1,p_2)\nabla_0^{R_0}(r_0),
\end{eqnarray*}
where $r_i\in R_i$ for $i=0,1,2$ and $p_1,p_2\in P.$ In the above $X_0^A\in D(A), X_0^P\in \mathrm{Der}(P)$ are the components of the graded symbol $X_0\in D(\mathscr{T})_0$ of $\nabla_0.$
 \end{lemma}
Here we have chosen to suppress the structure maps. One may readily see by re-instating these that the first two relations in Lemma \ref{DegZeroTriolicDerOperators} may be written $\lambda_i\big(\lambda_{i-1}(p,r_{i-1})\big)=\lambda_{i-1}\big(X_0^P(p),r_{i-1}\big)+\lambda_i\big(p,\nabla_0^{R_{i-1}}(r_{i-1})\big)\in R_i$ for $i=1,2.$

We establish similar results for the other homogeneous components of $\text{Der}(\mathscr{R})_{\mathcal{G}}.$

\begin{lemma}
Degree one graded triolic Der-operators $\nabla_1\in \mathrm{Der}(\mathscr{R})_1$ have graded symbol $\sigma_{\nabla_1}=X_1\in D(\mathscr{T})_1$ which amounts to a pair $\nabla_1=\big(\nabla_1^{R_0},\nabla_1^{R_1}\big)$ of generalized Der-operators whose Der-Leibniz rules read 
$$\nabla_1^{R_i}(ar_i)=X_1^A(a)r_i+a\nabla_1^{R_i}(r_i)\in R_{i+1}, i =0,1$$
and where we have the following further relation
$$\nabla_1^{R_1}(pr_0)\equiv \nabla_1^{R_1}\lambda_0(p,r_0)=\nu\big(X_1^P(p),r_0\big)-\lambda_1\big(p,\nabla_1^{R_0}(r_0)\big)\in R_2,$$
for $r_i\in R_i,i=0,1,2,p\in P$ and $a\in A.$
\end{lemma}
The first relations, when written in terms of the structure maps show that these operators have the symbols
$\sigma_{\nabla_1^{R_i}}=\big(\lambda_i\otimes id_{R_i}\big)\circ \big(X_1^A(-)\otimes id_{R_i}\big),$ for $i=0,1.$

\begin{lemma}
A degree two graded Der-operator is a first order differential operator  $\nabla_2^{R_0}:R_0\rightarrow R_2,$ which satisfies $\nabla_2^{R_0}(ar_0)=X_2^A(a)r_0+a\nabla_2^{R_0}(r_0)=\nu \big(X_2^A(a),r_0\big)+a\nabla_2^{R_0}(r_0)\in R_2.$ (recall $X_2^A\in D(A,Q)$.) In other words, $\mathrm{Der}(\mathscr{R})_2\cong \mathrm{Der}^{\nu}(R_0,R_2).$
\end{lemma}
We finish with a description of those operators of degree $-1.$

\begin{lemma}
A degree $-1$ triolic der operator is a pair $\nabla_{-1}=\big(\nabla_{-1}^{R_1},\nabla_{-1}^{R_2}\big)$ of $A$-linear operators $\nabla_{-1}^{R_1}:R_1\rightarrow R_0,\nabla_{-1}^{R_2}:R_2\rightarrow R_0,$ which satisfy $\nabla_{-1}^{R_1}(pr_0)=X_{-1}^P(p)r_0\in R_0, \nabla_{-1}^{R_2}(pr_1)=X_{-1}^P(p)r_1\in R_1$ and $\nabla_{-1}^{R_2}(qr_0)=X_{-1}^Q(q)r_0\in R_1$ and finally, $\nabla_{-1}^{R_2}\big(g(p_1,p_2)r_0\big)=X_{-1}^P(p_1)\nu_{R_0}(p_2,r_0)-X_{-1}^P(p_2)\nu_{R_0}(p_1,r_0)\in R_1.$
\end{lemma}

\subsection{Triolic connections and a splitting result}
For the remainder of this subsection, we assume that $\mathscr{T}$ is the canonical triole algebra associated with a pair of vector bundles $E,F$ over a smooth manifold $M$ of dimension $n$. We let $\{\partial_i\}_{i=1,..,n}\in D(M)$ be a local basis and we let $\{e_{\alpha}\}_{\alpha=1,..,m_P}$ be a basis for $P=\Gamma(E)$ and $\{\epsilon_A\}_{A=1,..,m_Q}$ be a basis for $Q=\Gamma(F)$. Let $g_{\alpha,\beta}^A\in C^{\infty}(M)$ denote the coefficients determining $g.$
\begin{definition}
A \emph{(homogeneous) triolic splitting}, is a degree zero splitting of the triolic Atiyah sequence \ref{eqn:TrioleAtiSeq01}, which is additionally a morphism of $C^{\infty}(M)$-modules.
\end{definition}
We can compute such splittings.

\begin{theorem}
\label{TriolicConnectionsTheorem}
A degree zero splitting $\nabla$ of the triolic Atiyah sequence of order $1$ is completely determined by symbols $(\Gamma_{i\alpha}^{\beta})$ as well as the symbols $(\Upsilon_{iA}^B)$ subject to the relation 
$$X^i\big[\partial_i(g_{\alpha\beta}^B)+g_{\alpha\beta}^A\Upsilon_{iA}^B\big]-\big(g_{\gamma\beta}^B\Gamma_{i\alpha}^{\gamma}+g_{\alpha\gamma}^{B}\Gamma_{i\beta}^{\gamma}\big)=0.$$
\end{theorem}
Since $\nabla$ is a splitting and since $\nabla_X\in D(\EuScript{T})_0$ we must have that $\nabla_X|_P\in \mathrm{Der}(P)$ and $\nabla_X|_Q\in\mathrm{Der}(Q)$ with shared scalar symbol given by $X.$ It then follows that we have the relations $\nabla_X(p)=X^i\big[\partial_i(p^{\beta})+p^{\alpha}\Gamma_{i\alpha}^{\beta}\big]e_{\beta},$ for the derivation $\nabla(X)\in D_1(\EuScript{T})_0$ restricted to $P,$ and similarly $\nabla_X(q)=X^i\big[\partial_i(q^{B})+q^{A}\Upsilon_{iA}^{B}\big]\epsilon_{B}$ when restricted to $Q.$
Evaluating the quantity $\nabla_X|_Qg(p_1,p_2)$ is readily seen to yield the
compatibility relation between the components presented above.
Note that we have written the action of $\nabla_{\partial_i}$ on sections of the bundles as
$\nabla_{\partial_i}(\epsilon_A)=\sum_{B=1}^{m_Q}\Upsilon_{iA}^B\epsilon_B$ and $\nabla_{\partial_i}(e_{\alpha})=\sum_{i=1}^{rk(P)}\Gamma_{i\alpha}^{\beta}e_{\beta}.$ Therefore, a triolic connection is determined by the pair of symbols written simply $\big(\Gamma,\Upsilon\big)$ and a compatibilty relation, governed by the metric. We call these the \emph{triolic Christofell symbols}.

There is another interpretation that we wish to give, following our discussion of connections which preserve inner structures in vector bundles given above, in particular in §§\ref{sssec:Bilinpreserveconnections}.

Let $b=b_{\alpha\beta}^Af_{\alpha\beta}^A$ be a generic element of $\mathrm{Bil}(P,Q)$.
Then $\Box\in\mathrm{Der}\big(\mathrm{Bil}(P,Q)\big)$ acts as
$$\Box\big(b_{\alpha\beta}^Af_A^{\alpha\beta})=\sigma_{\Box}\big(b_{\alpha\beta}^A\big)f_{A}^{\alpha\beta}+b_{\alpha\beta}^A\Box(f_{A}^{\alpha\beta}),$$
by definition of a Der-operator. This is determined by its action in the basis $f_A^{\alpha\beta}=e^{\alpha}\otimes e^{\beta}\otimes \epsilon_A$ where $e^{\alpha}$ denotes the basis of $P^{\vee}$ (ie. the dual one to $P$). Such an action is determined by a tuple of functions $(\Sigma)$ by
$$\Box\big(f_A^{\alpha\beta}\big)=\big[\Sigma_{A}^{\alpha\beta}\big]_{\alpha'\beta'}^{A'} f_{A'}^{\alpha'\beta'}.$$

Now, suppose that $\nabla$ is a connection $P$, determined by functions $\Gamma$ and $\Delta$ is a connection in $Q$, determined by functions $\Upsilon.$ We know how to construct a unique connection in tensor product modules and in dual modules. This connection arises via the description as Der-operators in the product bundles if and only if all operators have shared scalar symbol. In particular, we have the following.

\begin{lemma}
Let $\nabla\in\mathrm{Der}(P)$ an $\Delta\in \mathrm{Der}(Q)$ be such that $\sigma_{\nabla}=\sigma_{\Delta}=X\in D(A).$ Then there exists a unique Der-operator
$\Box\in\mathrm{Der}\big(\mathrm{Bil}(P,Q)\big)$ whose coefficients $(\Sigma)$ which determine it are given by:
$$\big[\Sigma_A^{\alpha\beta}\big]_{\alpha'\beta'}^{A'}=\big[\Gamma^{\vee}\big]_{\alpha'}^{\alpha}\otimes \delta_{\beta'}^{\beta}\otimes \delta_A^{A'}+\delta_{\alpha'}^{\alpha}\otimes \big[\Gamma^{\vee}\big]_{\beta'}^{\beta}\otimes \delta_A^{A'}+\delta_{\alpha'}^{\alpha}\otimes \delta_{\beta'}^{\beta}\otimes \Upsilon_A^{A'}.$$
\end{lemma}
This follows from taking $\Box:=\big(\nabla\otimes \nabla\big)^{\vee}\otimes \Delta,$ and computing the Der-Leibniz rule. Note further that the dual Der-operators, $\nabla^{\vee}$ act on the dual basis by $\nabla^{\vee}(e^{\alpha})=[\Gamma^{\vee}]_{\beta}^{\alpha} e^{\beta}$ and are related to the symbols $\Gamma_{\beta}^{\alpha}$ for $\nabla$ as $[\Gamma^{\vee}]_{\beta}^{\alpha}=-\Gamma_{\beta}^{\alpha}.$ That is $\nabla_{\partial_i}^{\vee}(e^{\alpha})=\big[\Gamma^{\vee}\big]_{i\beta}^{\alpha}e^{\beta}$ with $\big[\Gamma^{\vee}\big]_{i\beta}^{\alpha}=-\Gamma_{i\beta}^{\alpha}.$

Consequently, $\Box\in\mathrm{Der}\big(\mathrm{Bil}(P,Q)\big)$ is a $g$-preserving linear connection if and only if $\Box_X(g)=0,$ for each $g\in \mathrm{Bil}(P,Q)$. One may see this then reduces to the expression of Theorem \ref{TriolicConnectionsTheorem}.
\subsubsection{The equation of flat triolic splittings}
\label{EquationofFlatTriolicConnections}
The space of splittings of the ordinary Atiyah sequence for the bundle $P,$ is in correspondence with linear connections in the given bundle \cite{Ati}.
Suppose that we have such a connection $\nabla$ which we understand as a Der-operator in the module $P$. It may be viewed equivalently as a \emph{linear vector field} on the total space of our bundle. Thus, we see $\nabla$ as an assignment of a vector field $\partial_i\longmapsto \nabla(\partial_i):=\partial_i+\sum_{\alpha=1}^mu_i^{\alpha}\partial_{u^{\beta}}$ where $u_{\alpha}$ are coordinates on a trivialization $\pi$ over $U\subset M,$ a choice of a local chart. We have $nm$-arbitrary functions $u_i^{\alpha}\equiv u_i^{\alpha}(x_1,\ldots,x_n,u^1,\ldots,u^m)$ on $E$, which determine $\nabla.$
To be a flat connection, means that this assignment is not only an $\mathbb{R}$-linear map, but a morphism of Lie algebras. That is, it preserves the Lie brackets, so that we have  $[\nabla(\partial_i),\nabla(\partial_j)]=\nabla\big([\partial_i,\partial_j]\big).$
Writing this out explicitly, one finds the \emph{equation of flat connections} which is a system of $mn(n-1)/2$ non-linear partial differential equations in $mn$ unknown variables, given by 
\begin{equation}
    \frac{\partial u_j^{\alpha}}{\partial x_i}+\sum_{\beta=1}^mu_i^{\beta}\frac{\partial u_j^{\alpha}}{\partial u^{\beta}}-\frac{\partial u_i^{\alpha}}{\partial x_j}-\sum_{\beta=1}^mu_j^{\beta}\frac{\partial u_i^{\alpha}}{\partial u^{\beta}}=0, 
\end{equation}
with $1 \leq i<j \leq n,\alpha=1,\ldots,m.$
By exploiting the equivalence of Der-operators on bundles and linear vector fields on the total space, one can re-write such an equation in terms of the Christoffel symbols. 
Define a curvature tensor $R^{\nabla}(X,Y)$ on $p\in P$ according to the formula
$$R^{\nabla}(X,Y)(p^{\alpha}e_{\alpha})=X^iY^jp^{\alpha}R^{\nabla}\big(\partial_i,\partial_j\big)(e_{\alpha}),$$
and we introduce the components $\mathcal{R}_{ij\alpha}^{\beta}\in C^{\infty}(M),\alpha,\beta=1,\ldots,m_P$, according to 
$R^{\nabla}\big(\partial_i,\partial_j\big)(e_{\alpha}):=\sum_{\beta=1}^{m}\mathcal{R}_{ij\alpha}^{\beta}e_{\beta}.$
Then, to be a flat connection, amounts to 
$R^{\nabla}\equiv 0,$
which means we have
$R^{\nabla}\big(X,Y\big)=X^iY^jp^{\alpha}\mathcal{R}_{ij\alpha}^{\beta}e_{\beta}=0$ and this occurs when the components of the tensor are identically zero. They can be computed to be 
$$\mathcal{R}_{ij\alpha}^{\beta}=\partial_i\big(\Gamma_{j\alpha}^{\beta}\big)-\partial_j\big(\Gamma_{i\alpha}^{\beta}\big)+\Gamma_{i\gamma}^{\beta}\Gamma_{j\alpha}^{\gamma}-\Gamma_{j\gamma}^{\beta}\Gamma_{i\alpha}^{\gamma}-c_{ij}^k\Gamma_{k\alpha}^{\beta},$$
where $[\partial,\partial_j]=c_{ij}^k\partial_k,$ denote the Lie algebra structure constants. It can be readily checked that $\mathcal{R}_{ij\alpha}^{\beta}=-\mathcal{R}_{ji\alpha}^{\beta}.$
Now, a degree zero triolic derivation amounts to the datum of two Der-operators, one in each bundle $P,Q$ with shared scalar symbol, so it should be no surprise that a splitting that is flat, (ie. a flat connection), should incorporate the equation of flat connections for both Der-operators. In this way one may establish that the relations 
$[\nabla_{\partial_i},\nabla_{\partial_j}]-\nabla_{[\partial_i,\partial_j]}=0,$ when restricted to $p\in P$, or $q\in Q,$ read, respectively as
\begin{eqnarray}
\mathcal{R}(P)_{ij\alpha}^{\beta}&=& \partial_i\big(\Gamma_{i\alpha}^{\beta}\big)-\partial_j\big(\Gamma_{i\alpha}^{\beta}\big)+\Gamma_{i\gamma}^{\beta}\Gamma_{j\alpha}^{\gamma}-\Gamma_{j\gamma}^{\beta}\Gamma_{i\alpha}^{\gamma},
\\
\mathcal{R}(Q)_{ijA}^{B}&=& \partial_i\big(\Upsilon_{iA}^{B}\big)-\partial_j\big(\Upsilon_{iA}^{B}\big)+\Upsilon_{iD}^{B}\Upsilon_{jA}^{D}-\Upsilon_{jD}^{B}\Upsilon_{iA}^{D},
\end{eqnarray}
where we wrote $\mathcal{R}(P)$ and $\mathcal{R}(Q)$ to denote the corresponding components when appropriately restricted.

Finally by computing the relation of flat connections, restricted not simply to arbitrary $q\in Q$ but to $\text{im}(g)$, and implementing the relation between the triolic Christoffel symbols of Theorem \ref{TriolicConnectionsTheorem}, we find that the relation $R^{\nabla}(X,Y)\big(g(p_1,p_2)\big)=0,$ is computed to yield the following expression,
\begin{eqnarray}
\label{flattriolecompatibility}
\partial_i\big(g_{\gamma\beta}^D\Gamma_{j\alpha}^{\gamma}\big)&+&g_{\gamma\beta}^B\Gamma_{j\alpha}^{\gamma}\Upsilon_{iB}^D +\partial_i\big(g_{\alpha\gamma}^D\Gamma_{j\beta}^{\gamma}\big) 
+g_{\alpha\gamma}^B\Gamma_{j\beta}^{\gamma}\Upsilon_{iB}^D-\partial_j\big(g_{\gamma\beta}^D\Gamma_{i\alpha}^{\gamma}\big)-g_{\gamma\beta}^B\Gamma_{i\alpha}^{\gamma}\Upsilon_{jB}^D \nonumber
\\
&-&\partial_j\big(g_{\alpha\gamma}^D\Gamma_{i\beta}^{\gamma}\big)-g_{\alpha\gamma}^B\Gamma_{i\beta}^{\gamma}\Upsilon_{jB}^D 
-c_{ij}^k\partial_kg_{\alpha\beta}^D+g_{\alpha\beta}^B\Upsilon_{kB}^D=0.
\end{eqnarray}

\begin{theorem}
\label{FlatTriolicConnection}
Let $\nabla:D(A)\rightarrow D_1(\EuScript{T})_0,$ be a splitting of the degree zero first order triolic Atiyah sequence. Then it determines a flat triolic connection if and only if we have that 
$$\mathcal{R}(P)_{ij\alpha}^{\beta}\equiv 0,\hspace{7mm} \mathcal{R}(Q)_{ijA}^B\equiv 0,$$
for all $i,j=1,\ldots,n,\alpha,\beta=1,\ldots,m_P$ and $A,B=1,\ldots,m_Q,$ such that equation (\ref{flattriolecompatibility}) holds.
\end{theorem}
A more symmetric format for presenting (\ref{flattriolecompatibility}) is by
\begin{eqnarray*}
g_{\gamma\beta}^D\big(\partial_i\Gamma_{j\alpha}^{\gamma}-\partial_j\Gamma_{i\alpha}^{\gamma}\big) &+& g_{\alpha\gamma}^D\big(\partial_i\Gamma_{j\beta}^{\gamma}-\partial_j\Gamma_{i\beta}^{\gamma}\big)+g_{\alpha\gamma}^B\big(\Gamma_{j\beta}^{\gamma}\Upsilon_{iB}^D-\Gamma_{i\beta}^{\gamma}\Upsilon_{jB}^D\big)+g_{\gamma\beta}^{B}\big(\Gamma_{j\alpha}^{\gamma}\Upsilon_{i B}^D-\Gamma_{i \alpha}^{\gamma}\Upsilon_{j B}^D\big)
\\
&+&\partial_i(g_{\gamma\beta}^D)\Gamma_{j\alpha}^{\gamma}+\partial_i(g_{\alpha\gamma}^D)\Gamma_{j\beta}^{\gamma}-\partial_j(g_{\gamma\beta}^D)\Gamma_{i\alpha}^{\gamma}-\partial_j(g_{\alpha\gamma}^D)\Gamma_{i\beta}^{\gamma}+g_{\alpha\beta}^B\Upsilon_{kB}^D=0.
\end{eqnarray*}
 It remains to interpret this on physical grounds and to see if flat triolic connections appear in nature.

\section{Bi-derivations of triole algebras}
\label{sec:Bi-derivations of triole algebras}
We will now describe the functor of bi-derivations $D_2(\EuScript{T})$ for our triole algebra $\EuScript{T}$. Namely, we describe
\begin{eqnarray}
\label{eqn:BiDer}
D_2^{\mathcal{G}}:\text{Mod}^{\mathcal{G}}(\EuScript{T})&\rightarrow&\text{Mod}^{\mathcal{G}}(
\EuScript{T}) \nonumber
\\
\EuScript{P}&\mapsto& D_2^{\mathcal{G}}(\EuScript{P}):=D\big(\EuScript{T},D(\EuScript{P})\subset \mathrm{Diff}_{1}^>(\EuScript{P})_{\mathcal{G}}\big),
\end{eqnarray}
where $\text{Diff}_1^>$ is the functor of first order differential operators supplied with its right module structure. Elements of this module are viewed as graded derivations $\Delta:\EuScript{T}\rightarrow\text{Diff}_1^>(\EuScript{T}),$ such that $\Delta(\EuScript{T})\subset D(\EuScript{T}).$ This functor co-represents the module of two-forms in the graded algebraic setting. Indeed, this functorial algebraic definition coincides with the usual one in the following way. Setting $\EuScript{P}=\EuScript{T}$ for simplicity, given a map $\Delta:\EuScript{T}\times\EuScript{T}\rightarrow\EuScript{T}$ such that
$\Delta(t,-)$ is a derivation of $\EuScript{T}$, one has the graded Leibniz rule,
$$\Delta(t_1t_2,-)=(-1)^{\Delta t_2+t_1t_2}t_2^{>}\Delta(t_1,-)+(-1))^{t_1\Delta}t_1^{>}\Delta(t_2,-),$$
for $t_1,t_2\in \EuScript{T},$ where $t_i^>(-)$ denotes the right module action.

Now for every $\Delta \in D_2^{\mathcal{G}}(\EuScript{T})$ there is a naturally associated mapping, $
\widetilde{\Delta}:\EuScript{T}\times\EuScript{T}\rightarrow\EuScript{T},\hspace{2mm}
(a,b)\longmapsto \widetilde{\Delta}(a,b):=\big[\Delta(a)\big](b),$
for all homogeneous $a,b\in\EuScript{T}.$
We therefore fix our notation by viewing a bi-derivation $\Pi$ as either
$\big[\widetilde{\Pi}_h^{k,j}:\EuScript{T}_k \times\EuScript{T}_j\rightarrow \EuScript{T}_{k+j+h}\big]$ or as $\big[\Pi_h^{k}(-):\EuScript{T}_k\rightarrow D_1(\EuScript{T})_{k+h}\big],$
which are related by the rule
$\big[\Pi_h^k(t_k)\big]_{k+h}^{j}(t_j)=\widetilde{\Pi}_h^{k,j}(t_k,t_j).$

\begin{remark}
The admissible degrees are easily seen to be $g\in \{-4,-3,-2,-1,0,1,2\}.$ 
Moreover, for arbitrary $k$-multi derivations, the admissible derivations are in degrees $\{-2k,\ldots,2\}.$
\end{remark}

In particular, for the triole algebra, in homogeneous components, a bi-derivation $\Pi_g$ of degree $g$ amounts to a triple of operators $\Pi_g^A:A\rightarrow \big(D(\EuScript{T})\subset \mathrm{Diff}_1^>(\EuScript{T})
\big)_{g},\Pi_g^P:P\rightarrow \big(D(\EuScript{T})\subset \mathrm{Diff}_1^>(\EuScript{T})
\big)_{g+1} , \Pi_g^Q:Q\rightarrow \big(D(\EuScript{T})\subset \mathrm{Diff}_1^>(\EuScript{T})
\big)_{g+2}$ and we should pay special attention to $\Pi_0^Q|_{\mathrm{Im}(g)}:g\big(P,P)\subset Q\rightarrow \big(D(\EuScript{T})\subset \mathrm{Diff}_1^>(\EuScript{T})
\big)_{g+2},$ as the latter will elucidate compatibility between the components of the derivation and the fiber metric.

 \begin{lemma}
 \label{DegreeZeroBiDerivations}
A degree $0$ triolic bi-derivation $\Pi_0$ is given by  a collection of operators $\Pi_0=\big(\widetilde{\Pi}_0^{AA},\widetilde{\Pi}_0^{AP},\widetilde{\Pi}_0^{AQ},\widetilde{\Pi}_0^{PP}\big)$ where $ 
\widetilde{\Pi}_0^{AA}\in D_2(A)$ is an ordinary bi-derivation of the algebra $A,$ and
$\widetilde{\Pi}_0^{AP}\in D\big(A,\mathrm{Der}(P)\big)
,
\widetilde{\Pi}_0^{AQ}\in D\big(A,\mathrm{Der}(Q)\big),
$ while $
\widetilde{\Pi}_0^{PP}\in \mathrm{Der}\big(P,\mathrm{Der}^g(P,Q)\big)\subset \mathrm{Diff}_1\big(P,\mathrm{Der}^g(P,Q)\big).$
These latter objects satisfy 
$\widetilde{\Pi}_0^{PP}(p_1,ap_2)=g\big(\widetilde{\Pi}_0^{AP}(a,p_1),p_2)\big)+a\widetilde{\Pi}_0^{PP}(p_1,p_2).$
\end{lemma}

Note that this description implies that the operator $\widetilde{\Pi}_0^A$ (which consists of $\widetilde{\Pi}_0^{A*},*=A,P,Q$) is an element of  $D\big(A,\mathcal{D}\mathrm{er}\big(g;P,Q\big)\big).$ Moreover, in the case that we have a fixed splitting of $\mathfrak{at}_1(\EuScript{T})_0,$ that is, $P,Q$ are endowed with flat connections, then we may interpret $\widetilde{\Pi}_0^A$ as an element of $D(A)\otimes \mathrm{Der}\big(\mathrm{Bil}(P,Q)\big).$
We exploit this interpretation to deduce coordinate expressions for triolic Poisson structures below in Theorem \ref{TriolicPoissonStructure}.

\begin{lemma}
Degree $1$ triolic bi-derivations are pairs of operators $\Pi_1=(\Pi_1^{AA},\Pi_1^{AP})$ with $\Pi_1^{AA}\in D\big(A,D(P)\big)\cong D_2(A)\otimes_A P,$ and with $\Pi_1^{AP}\in D(A)\otimes_A \mathrm{Der}^g(P,Q).$ These components are related by
$\widetilde{\Pi}_1^{AP}(a,bp)=g\big(\widetilde{\Pi}_1^{AA}(a,b),p\big) + b\widetilde{\Pi}_1^{AP}(a,p).$
\end{lemma}

One may interpret these as generalized Hamiltonian connections. That is, they are a natural generalization of the objects that we characterized as degree $0$ graded bi-derivations of the algebra of dioles, but now take values in another module $Q,$ with a suitably twisted Hamiltonian symbol by the vector-valued fiber metric $g.$

\begin{lemma}
Degree $2$ bi-derivations of $\EuScript{T}$ coincide with $Q$-valued bi-derivations of $A$.
\end{lemma}

In cases of interest, we also have bi-derivations in degrees $-1,-2$ which admit a more interesting interpretation in terms of a graded Poisson bracket and we elucidate this point of view in §§\ref{sssec:Degree-1-2PoissonStructures}.

\subsection{Triolic Poisson structures}
\label{ssec:TriolicPoissonStructures}
We turn our attention to computing Poisson structures in the triole algebra $\EuScript{T}.$ That is, we must compute $[\![\Pi_g,\Pi_g]\!]=0,$ for the admissible gradings of triolic bi-derivations.
\subsubsection{Degree $0$ triolic Poisson structures}
\label{Degree0TriolicPoissonstructures}
For arbitrary homogeneous elements $t_i,t_j,t_k\in \EuScript{T},$ the relation 
$[\![\Pi_0,\Pi_0]\!](t_i,t_j,t_k)=0,$
amounts to the graded Jacobi identity and reads
$$
\widetilde{\Pi}_0^{\EuScript{T}_i,\EuScript{T}_{j+k}}\big(t_i,\widetilde{\Pi}_0^{\EuScript{T}_j,\EuScript{T}_k}(t_j,t_k)\big)=-(-)^{(i+j)k}\widetilde{\Pi}_0^{\EuScript{T}_k,\EuScript{T}_{i+j}}\big(t_k,\widetilde{\Pi}_0^{\EuScript{T}_i,\EuScript{T}_j}(t_i,t_j)\big)+(-1)^{ij}\widetilde{\Pi}_0^{\EuScript{T}_j,\EuScript{T}_{i+k}}\big(t_j,\widetilde{\Pi}_0^{\EuScript{T}_j,\EuScript{T}_k}(t_i,t_k)\big).$$

Using Lemma  \ref{DegreeZeroBiDerivations}, we can compute the Jacobi-identity evaluted on triples $(t_i,t_j,t_k)$ in the case where we have $(a_0,a_1,a_2)\in A\times A\times A,$ and where we have $(a_0,a_1,p)\in A\times A\times P$ and $(a_0,a_1,q)\in A\times A\times Q.$ All remaining cases are similarly established and since it is not very instructive to present them in full detail, we summarize their important points:
\begin{enumerate}[label=$\bullet$]
    \item The Jacobi-identity evaluated on degree zero elements $a_0,a_1,a_2,$ gives the usual Jacobi identity for the Poisson bracket of the biderivation $\Pi_0^{AA}$ of the algebra $A,$
    
    \item On $(a_0,a_1,p_0)$ we find 
    $\widetilde{\Pi}_0^{AP}\big(a_0,\widetilde{\Pi}_0^{AP}(a_1,p_0)\big)=-\widetilde{\Pi}_0^{PA}\big(p_0,\widetilde{\Pi}_0^{AA}(a_0,a_1)\big)+\widetilde{\Pi}_0^{AP}\big(a_1,\widetilde{\Pi}_0^{AP}(a_0,p_0)\big),$
   rearranged using skew symmetry to give
    \begin{equation}
    \label{eqn:Jacobifor AAP}
    \widetilde{\Pi}_0^{AP}\big(a_0,\widetilde{\Pi}_0^{AP}(a_1,p_0)\big)-\widetilde{\Pi}_0^{AP}\big(a_1,\widetilde{\Pi}_0^{AP}(a_0,p_0)\big)-\widetilde{\Pi}_0^{AP}\big(\widetilde{\Pi}_0^{AA}(a_0,a_1),p_0\big)=0.
    \end{equation}

    \item For triples $(a_0,a_1,q_0),$ we have
    \begin{equation}
    \label{eqn:Jacobi for AAQ}
      \widetilde{\Pi}_0^{AQ}\big(a_0,\widetilde{\Pi}_0^{AQ}(a_1,q_0)\big)-\widetilde{\Pi}_0^{AQ}\big(a_1,\widetilde{\Pi}_0^{AQ}(a_0,q_0)\big)-\widetilde{\Pi}_0^{AQ}\big(\widetilde{\Pi}_0^{AA}(a_0,a_1),q_0\big)=0.
    \end{equation}
     \end{enumerate}
Checking all remaining details then gives the following.
\begin{theorem}
\label{TriolicPoissonStructure}
$\Pi_0\in D_2(\EuScript{T})_0$ is a triolic Poisson structure, (i.e. satisfies
$[\![\Pi_0,\Pi_0]\!]=0,$)
if and only if its components satisfy the following system of non-linear PDEs
\begin{enumerate}
    \item $\oint_{\{0,1,2\}} \widetilde{\Pi}_0^{AA}\big(a_0,\widetilde{\Pi}_0^{AA}(a_1,a_2)\big)=0,$
    \item $
 \widetilde{\Pi}_0^{AP}\big(a_0,\widetilde{\Pi}_0^{AP}(a_1,p_0)\big)-\widetilde{\Pi}_0^{AP}\big(a_1,\widetilde{\Pi}_0^{AP}(a_0,p_0)\big)-\widetilde{\Pi}_0^{AP}\big(\widetilde{\Pi}_0^{AA}(a_0,a_1),p_0\big)=0,$
 
 \item 
 $\widetilde{\Pi}_0^{AQ}\big(a_0,\widetilde{\Pi}_0^{AQ}(a_1,q_0)\big)-\widetilde{\Pi}_0^{AQ}\big(a_1,\widetilde{\Pi}_0^{AQ}(a_0,q_0)\big)-\widetilde{\Pi}_0^{AQ}\big(\widetilde{\Pi}_0^{AA}(a_0,a_1),q_0\big)=0,$
 
  \item $\widetilde{\Pi}_0^{AQ}\big(a_0,\widetilde{\Pi}_0^{PP}(p_0,p_1)\big)-\widetilde{\Pi}_0^{PP}\big(p_1,\widetilde{\Pi}_0^{AP}(a_0,p_0)\big)-\widetilde{\Pi}_0^{PP}\big(p_0,\widetilde{\Pi}_0^{AP}(a_0,p_1)\big)=0.$
\end{enumerate}
\end{theorem}

 The first equation in Theorem \ref{TriolicPoissonStructure} is the usual one for a Poisson structure. In other words, $\widetilde{\Pi}_0^{AA}\in D_2(A)$ is an ordinary Poisson tensor. The second and third equations, in particular, expression (\ref{eqn:Jacobi for AAQ}), are the diolic Poisson structure equations for degree $0$ biderivations in $P,Q$, respectively. The
final relation is a simple compatibility condition between the components of our graded bivector. 

The theorem given above can be rephrased in purely diolic terms using an appropriate artifical algebra.
\begin{lemma}
Consider the artifical diole algebra $\EuScript{A}_{\mathrm{Bil}(P,Q)}.$ Let $\Pi_0\in D_2\big(\EuScript{A}_{\mathrm{Bil}(P,Q)}\big).$
Then $[\![\Pi_0,\Pi_0]\!]=0,$ if and only if the components satisfy the diole PDE for $P$ and diole PDE for $Q,$ as well as the ordinary Poisson PDE. In other words, a triolic Poisson structure is equivalently a diolic Poisson structure in the artificial $\mathrm{Bil}(P,Q)$-diole algebra whose components additionally satisfy the relation (4) of Theorem \ref{TriolicPoissonStructure}.
\end{lemma}

\subsubsection{Degree $-1,-2$ Poisson structures}
\label{sssec:Degree-1-2PoissonStructures}
Consider $\{-,-\}$ a Poisson bracket on $\EuScript{T}$ of degree $-1$. That is, we have a skew-symmetric map $\{-,-\}:\EuScript{T}_i\times \EuScript{T}_j\rightarrow \EuScript{T}_{i+j-1}$ which satisfies the graded Jacobi identity and is a graded derivation. Fix the first entry as $\{P,-\}:\EuScript{T}\rightarrow \EuScript{T}.$ By restricting to $A$, we find, via the Leibniz rule for the Poisson bracket, that $\{P,-\}|_A\in D(A).$ Moreover, this is readily seen to be $A$-linear in the first argument and consequently we have a well-defined morphism of $A$-modules
\begin{equation}
    \alpha:P\rightarrow D(A), \hspace{2mm} p\mapsto\alpha(p):=\{p,-\}|_A.
\end{equation}
The Jacobi identity (which necessarily holds as $\{-,-\}$ is a Poisson bracket) arising from the Jacobiator $\mathrm{Jac}(a,p_1,p_2)=0,$ gives that 
\begin{equation}
\label{Anchormap}
\alpha\big(\{p_1,p_2\}\big)(a)=\big[\alpha(p_1),\alpha(p_2)\big](a),
\end{equation}
where the commutator on the right-hand side is the usual one on $D(A).$

Now, by restricting $\{P,-\}|_P:P\rightarrow P,$ one may deduce via the Leibniz rule that we have 
\begin{equation}
\label{PcomponentDeroperator}
\{p,-\}|_P(ap_1)=\alpha(p)(a)\cdot p_1+a\cdot \{p,p_1\}|_P.
\end{equation}
Moreover the skew-symmetry of the Poisson bracket and the Jacobi identity arising from $\mathrm{Jac}(p_1,p_2,p_3)=0,$ tell us that $\{-,-\}:P\rightarrow P\rightarrow P$ is a Lie bracket on $P,$ and so we see that equation \ref{Anchormap} above tells us that $\alpha$ is a Lie algebra morphism, so that $(P,[-,-]_P,\alpha)$ is a Lie algebroid and moreover, equation \ref{PcomponentDeroperator} tells us that $\{p,-\}_P$ is a Der-operator in $P$ whose symbol is provided by the anchor.

The matrix of endomorphism which determines this Der-operator is given by the Lie bracket in $P,$ which we denote by $[-,-]_P.$ We actually have the obvious assignment $P\mapsto \mathrm{Der}(P)$ is a morphism of $A$-modules as well.

Now, considering $\{P,-\}|_Q:Q\rightarrow Q,$ we see, via the Leibniz rule that 
\begin{equation}
    \label{QcomponentDeroperator}
    \{p,-\}_{Q}(aq)=\alpha(p)(a)\cdot q+a\{p,q\}|_Q.
\end{equation}

In this way, we see the $\{P,-\}|_Q\in \mathrm{Der}(Q)$ whose symbol is given by the anchor map and the `endomorphism part' of the Der-operator, is determined by the Poisson bracket.

By computing the Leibniz rule not on a product of the form $aq$, but rather on $p_1\cdot p_2$, which, by definition of the triole algebra is $g(p_1,p_2)$, one may establish the relation

\begin{equation}
    \label{Bracketevaluatedong}
    \{p,-\}|_Qg(p_1,p_2)=g\big([p,p_1]_P,p_2\big)+g\big(p_1,[p,p_2]_P\big).
\end{equation}

Consequently, we see from equations \ref{PcomponentDeroperator},\ref{QcomponentDeroperator} and \ref{Bracketevaluatedong}, that the pair $\big(\{P,-\}|_P,\{P,-\}|_Q\big)$ is really an element of $\mathcal{D}\mathrm{er}\big(g;P,Q\big),$ whose symbol is given by the anchor map $\alpha(P).$

Now, fixing the first entry of our Poisson bracket to take values in $Q$, we have that $\{Q,-\}|_A:A\rightarrow P$ is a derivation. This claim follows from the Leibniz rule. Moreover since $\{A,A\}\equiv 0,$ for degree reasons,
we have that $\{aq,b\}=\{b,a\}q+a\{b,q\}=a\{q,b\}$ so that the assignment
\begin{equation}
    \label{Generalanchor}
    f:Q\rightarrow D(A,P),\hspace{2mm} q\mapsto f(q):=\{q,-\}|_A,
\end{equation}
is an $A$-module homomorphism.

Considering the Jacobi-identity $\mathrm{Jac}(p,a,q)=0,$
one may readily deduce that
\begin{equation}
    \label{Jacobiforpaq}
    f(q)\big(\alpha(p)(a)\big)+\big[f(q)(a),p\big]_P+f\big(\{p,q\}\big)(a)=0.
\end{equation}

The final identity to examine is the Jacobi-identity arising from $\mathrm{Jac}\big(p,a,g(p_1,p_2)\big)=0.$

\begin{lemma}
\label{Degree-1PoissonBracketDescription}
A degree $-1$ Poisson bracket on $\EuScript{T}$ is equivalent to the datum of:
a Lie algebroid $(P,[-,-]_P,\alpha)$ and an element $Z=(Z^P,Z^Q)\in \mathcal{D}\mathrm{er}(g;P,Q)$ where $Z^P=\alpha(P) + [P,-]_P$ and where $Z^Q=\alpha(P)+ \{P,-\}|_Q$ such that the compatibility equation (\ref{Jacobiforpaq}) holds.
\end{lemma}

The objects appearing in Lemma \ref{Degree-1PoissonBracketDescription} appear to be rather unique. It remains to be seen if these appear elsewhere in mathematics or fields of mathematical physics. 
\begin{lemma}
\label{Degree-2PoissonBracketDescription}
Let $\{-,-\}_{-2}$ be a Poisson bracket in $\EuScript{T}$ of degree $-2.$ Then this equivalently determines in $Q$ the structure of a Lie algebroid, with bracket $[-,-]_Q:=\{-,-\}:Q\times Q\rightarrow Q$ and whose anchor is $\alpha:Q\rightarrow D(A),$ given by $\alpha(q):=\{q,-\}|_A.$
\end{lemma}

\section{Triolic differential operators}
\label{sec:DifferentialOperators}
We will now characterize first order differential operators explicitly. With this description, we proceed inductively to define arbitrary order operators.

Recalling the definition of first order (graded) differential operator, we provide for convenience the following generic formula. Let $t_i,t_j$ be homogeneous scalar elements of degree $i,j$ respectively in $\EuScript{T},$ and let $\Delta_h$ be some differential operator of degree $h.$ Then it is of first order if
$$\delta_{t_i,t_j}(\Delta_h)(-)=t_it_j\Delta_h(-)-(-1)^{hi+ji}t_j\Delta_h(t_i-)-(-1)^{hj}t_i\Delta_h(t_j-)+(-1)^{hj+hi}\Delta_h(t_it_j-),$$ vanishes. 
Note also that homogeneous $t_i$ are interpeted as the degree $i$ differential operator of order $0$ of scalar multiplication, so $t_i\circ\Delta_h$ is to be thought of as a differential operator of the same order as $\Delta_h$ and degree $h+i.$
\subsubsection{Degree zero differential operators.}
A first order degree zero triolic differential operator is a $\mathbb{K}$-linear map $\Delta_0:\EuScript{T}\rightarrow \EuScript{T}$ and can be easily seen to be given by a triple of operators
$\Delta_0=\big(\Delta_0^A,\Delta_0^P,\Delta_0^Q\big)$ where $\Delta_0^A\in\mathrm{Diff}_1(A,A),\Delta_0^P\in\mathrm{Diff}_1(P,P),\Delta_0^Q\in\mathrm{Diff}_1(Q,Q)$ such that the following relations hold
\begin{itemize}
    \item $\Delta_0^Q\big(g(p_0,p_1)\big)=g\big(\Delta_0^P(p_0),p_1\big)+g\big(p_0,\Delta_0^P(p_1)\big)-g(p_0,p_1)\Delta_0^A(1_A),$
    
    \item $\delta_a(\Delta_0^A)=\delta_a(\Delta_0^P)$ and $\delta_a(\Delta_0^A)=\delta_a(\Delta_0^Q),$
    
    \item $g\big(p_0,\delta_{a_0}(\Delta_0^P)(p)\big)=\big[\delta_{a_0}(\Delta_0^Q)\big]g(p_0,p).$
\end{itemize}

Let $\mathfrak{g}:P\rightarrow \text{Hom}_A(P,Q)$ be the adjoint morphism of $g$. The final equation can be written as 
$\mathfrak{g}(p_0)\circ \delta_a(\Delta_0^P)=\delta_a(\Delta_0^Q)\circ \mathfrak{g}(p_0),$
for all $p_0\in P.$ This shows that the differential operator $\Delta_0^Q$ is determined completely by $\Delta_0^P,\Delta_0^A$ and $g.$ In particular, the choice of metric form $g$ is important.

\begin{lemma}
\label{triolediffopsdegreezero}
Elements of 
$\mathrm{Diff}_k(\EuScript{T})_0$ are triples $\Delta_0=(\Delta_0^A,\Delta_0^P,\Delta_0^Q)$ where $\Delta_0^A\in \mathrm{Diff}_k(A),\Delta_0^P\in \mathrm{Diff}_k(P,P),\Delta_0^Q\in \mathrm{Diff}_k(Q,Q)$ which satisfy the following relations for $a\in A$ $p_0,p_1,p\in P$ and $q\in Q,$

\begin{enumerate}
\item $p\delta_{a}^k(\Delta_0^A)(1_A)=\delta_a^k(\Delta_0^P)p,$

\item $
q\delta_a^k(\Delta_0^A)(1_A)=\delta_a^k(\Delta_0^Q)q,$

\item $g\big(p_0,\delta_{a}^{k}(\Delta_0^P)(p_1)\big)=\delta_a^k(\Delta_0^Q)g(p_0,p_1).$

\end{enumerate}
Moreover, these operators satisfy the following compatibility relation for all $a\in A,p_0,p_1\in P$,
\begin{equation}
    \label{DiolicDiffCompatibilityRelation}
\delta_a^{k-1}\Delta_0^Qg(p_1,p_0)=g\big(p_1,\delta_a^{k-1}\Delta_0^P(p_0)\big)+g\big(p_0,\delta_a^{k-1}\Delta_0^P(p_1)\big)-g(p_0,p_1)\big[\delta_a^{k-1}\Delta_0^A\big](1_A),\end{equation}
where we set $\delta_a^{k-1}:=\delta_a\circ \ldots\circ \delta_a,$ where we have $(k-1)$ factors of $\delta_a.$
\end{lemma}

Lemma \ref{triolediffopsdegreezero} tells us that an operator of order $k$ has the properties that $\Delta_0^P$ and $\Delta_0^Q$ have shared scalar symbol, given by the symbol of $\Delta_0^A$. 
This description of degree zero triolic differential operators suggests the existence of an Atiyah-like sequence in the following sense. Since the component operators share the same scalar type symbol, there is a naturally defined projection map 
$$\xi:\mathrm{Diff}_{k}(\EuScript{T})_0\rightarrow \mathrm{Diff}_k(A), \hspace{5mm} \xi(\nabla_0^A,\nabla_0^P,\nabla_0^Q):=\nabla_0^A.$$

The following is a straightforward computation.

\begin{lemma}
Elements of the $A$-module $\ker(\xi)$ consist of pairs of differential operators $(\nabla_0^P,\nabla_0^Q)$ with $\nabla_0^P\in\mathrm{Diff}_{k-1}(P,P)$ and $\nabla_0^Q\in \mathrm{Diff}_{k-1}(Q,Q)$ which satisfy the relation \begin{equation}
    \label{Diff(gPQ)relation}
\delta_a^{k-1}\nabla_0^Qg(p_1,p_2)=g\big(\delta_a^{k-1}\nabla_0^P(p_1),p_2\big)+g\big(p_1,\delta_a^{k-1}\nabla_0^P(p_2)\big).
\end{equation}
\end{lemma}
\begin{proof}
Suppose that $\nabla_0$ is such that $\nabla_0^A\equiv 0,$ which is to say $\nabla_0\in \ker(\xi).$ Then by Lemma \ref{triolediffopsdegreezero}, we see that both $\delta_a^k\big(\nabla_0^P\big)=0$ and $\delta_a^k\big(\nabla_0^Q\big)=0.$ This tells us that both operators must be of order $(k-1).$
Finally since $\nabla_0^A$ is trivial, then the compatibility relation (\ref{DiolicDiffCompatibilityRelation}) becomes $\delta_a^{k-1}\nabla_0^Qg(p_1,p_0)=g\big(p_1,\delta_a^{k-1}\nabla_0^P(p_0)\big)+g\big(p_0,\delta_a^{k-1}\nabla_0^P(p_1)\big).$
Since $\nabla_0^Q$ and $\nabla_0^P$ are operators of order $(k-1)$, this is a compatibility equation between their main symbols.
\end{proof}

Denote the $A$-module $\ker(\xi)$ by $\mathcal{D}\mathrm{iff}_{k-1}\big(g;P,Q\big).$ One may interpret elements of this space of differential operators as a type of generalized symmetry of $g,$ for the reason that such operators have the property that their main scalar symbols are symmetries in the ordinary sense of \ref{bilinearformpreservingconnection} and Definition \ref{Pairofconnectionspreservingg}.
\begin{theorem}
\label{TriolicAtiyahSequence0k}
Let $\EuScript{T}$ be the canonical triole algebra associated to a pair of vector bundles $\pi$ and $\eta$ with modules of sections $P,Q,$ respectively. Suppose that this triole algebra is regular. Then we have a short exact sequence of $A$-modules
$$0\rightarrow \mathcal{D}\mathrm{iff}_{k-1}(g;P,Q)\rightarrow \mathrm{Diff}_k(\EuScript{T})_0\rightarrow \mathrm{Diff}_k(A)\rightarrow 0.$$
\end{theorem}
We call the sequence in Theorem \ref{TriolicAtiyahSequence0k}, the \emph{Triolic Atiyah sequence of order} $k$ and denote it by  $\mathfrak{at}_k(\EuScript{T}).$

Note that for every $k$, we have an embedding of $A$-modules $\mathcal{D}\mathrm{iff}_{k}(g;P,Q)\subset \mathcal{D}\mathrm{iff}_{k+1}(g;P,Q)$ since every differential operator of order $k$ is in particular, an operator of order $k+1.$ Moreover, if $(\Box^P,\Box^Q)$ are a pair of operators of orders $k-1$ such that equation (\ref{Diff(gPQ)relation}) holds, then they may be viewed as a pair of operators of order $k$ as well, such that this same relation trivially holds. It follows from the usual order filtrations for the modules of (graded) differential operators that we have, for all $\ell\leq k,$ an embedding $\mathfrak{at}_{\ell}(\EuScript{T})\hookrightarrow \mathfrak{at}_k(\EuScript{T}).$ The \emph{infinite order triolic Atiyah sequence} $\mathfrak{at}(\EuScript{T})$ is defined to be the direct limit of the sequence of embeddings $\ldots\mathfrak{at}_{k-1}(\EuScript{T})\subset \mathfrak{at}_k(\EuScript{T})\subset\ldots$

\subsubsection*{Coordinates.}
We may search for local splittings of the sequence (\ref{TriolicAtiyahSequence0k}) to find a local description of triolic differential operators of degree $0.$ Elements of
$\mathrm{Diff}_k(\EuScript{T})_0$ are identified with decomposed operators as $\Delta_0\cong \pmb{\Box}_0^A+ \mathbf{M},$
where $\pmb{\Box}_0$ is a $k$-th order scalar operator and $\mathbf{M}$ is a pair of matrices 
$$\mathbf{M}:=\big(||\Box_{i,j}||, ||\Delta_{k,\ell}||\big),$$
where $\pmb{\Box}:=||\Box_{i,j}||:P\rightarrow P$ and $\pmb{\Delta}:=||\Delta_{k,\ell}||:Q\rightarrow Q,$ are matrix differential operators of order $(k-1)$, with $\Box_{i,j}:A\rightarrow A$ a scalar operator, which are moreover related by equation (\ref{Diff(gPQ)relation}).

\begin{remark}[A relation to diolic differential operators]
Suppose that $\EuScript{A}$ is a diolic algebra obtained from a triolic algebra by forgetting the $Q$-component. Then a triolic differential operator of degree zero is of the form $\Delta=(\Delta_0^A,\Delta_0^P,\emptyset)$ such that $\Delta_0^A\in \mathrm{Diff}_k(A,A)$ and $\Delta_0^P\in \mathrm{Diff}_k(P,P),$ which satisfy the relation $\delta_a^k(\Delta_0^A)=\delta_a^k(\Delta_0^P),$ which means they indeed have the same scalar type symbol. In this way, we recover diolic differential operators.
\end{remark}

\subsubsection{Degree one differential operators.}
\begin{lemma}
\label{Degree1TriolicDiffops}
Triolic differential operators of order $k$ and degree $1$ are given by pairs $\Delta_1=(\Delta_1^A,\Delta_1^P)$ with $\Delta_1^A\in \mathrm{Diff}_k(A,P)$ and $\Delta_1^P\in \mathrm{Diff}_k(P,Q)$ such that $\mathfrak{g}\circ \mathrm{smbl}_k(\Delta_1^A)=\mathrm{smbl}_k(\Delta_1^P).$
\end{lemma}
Note that the above represents the symbol of the operator $\Delta_1^P$ as the symbol of $\Delta_1^A$ twisted by the $A$-linear map $\mathfrak{g}:P\rightarrow \text{Hom}_A(P,Q)$ induced by $g$. This an equality between elements of $\text{Sym}^k\big(D(A)\big)\otimes_A \text{Hom}_A(P,Q).$

Here we have employed the standard identiication with symbols with those symmetric tensors. Namely
by writing $\text{smbl}_k(\Delta_1^P)\in \text{Sym}^k\big(D(A)\big)\otimes_A \text{Hom}_A(P,Q)$ we find that $\text{smbl}_k(\Delta_1^P)\big(df^{\odot k}\big)=\frac{1}{k!}\delta_f^k(\Delta_1^P)\in\text{Hom}_A(P,Q).$
Indeed, this is what we find from computing the algebraic definition of graded differential operators, which yields the relation 
$\mathfrak{g}\circ \delta_{a_0,a_1,..,a_{k-1}}(\Delta_1^A)=\delta_{a_0,a_1,..,a_{k-1}}(\Delta_1^P).$ 

\begin{remark}
This result is consistent with the analogy with diolic differential operators. Specifically, by forgetting the $Q$-component, we arrive precisely at our diolic differential operators. 
\end{remark}
\subsubsection{Degree two differential operators.}

\begin{lemma}
Triolic differential operators of degree $2$ coincide with $Q$-valued differential operators on $A.$ 
\end{lemma}
Locally, such $k$'th order operators are given by $\Box_2\in\mathrm{Diff}_{k}(\EuScript{T})_2\cong\text{Diff}_{k}(A,Q)$ as $
\Box_2:A\rightarrow Q, \hspace{1mm}
 f\mapsto \Box_2^A(f):=\big(
\overline{\Box}_1(f),\ldots,
\overline{\Box}_{m_Q}(f)\big)^T$
where $\overline{\Box}^{\alpha}\in\mathrm{Diff}_k(A)$ which will be specified in coordinates as
$\overline{\Box}^{\alpha}(f):=\sum_{|\sigma|\leq k}B_{\sigma}^{\alpha}\partial^{\sigma}(f).$ Note that $\overline{B}_j^{\sigma}\in A$ are coefficient functions determining the scalar differential operator $\overline{\Box}_j^A.$
We can write the entire operator $\Box_2$, acting on $f\in A$ as
$\Box_2(f)=\sum_{\alpha=1}^m\overline{\Box}^{A}(f)\epsilon_{A}=\sum_{A=1}^{m_Q}\sum_{|\sigma|\leq k}B_{\sigma}^{A}\partial^{\sigma}(f)\epsilon_{A}.$
From this, it is evident that the operator $\Box_2\in\text{Diff}_k(A,Q)$ is determined by the $\text{rank}(Q)$=tuple of functions $(B_{\sigma}^1,...,B_{\sigma}^{rk(Q)}).$

\subsection{Module-valued triolic differential operators}
\label{ssec:Module-valuedtriolicdifferentialoperators}
Consider $\EuScript{R}\in\tau\mathrm{Mod}(\EuScript{T}).$ 
We will characterize the graded differential operators $\mathrm{Diff}_1(\EuScript{T},\EuScript{R})$ of degrees $0,1,2.$ Proceeding as above, we find generalized conditions for first order degree zero operators $\Delta_0:\EuScript{T}\rightarrow \EuScript{R}$. Explicitly, simple computations show that such an operator is a triple $\Delta_0=(\Delta_0^A,\Delta_0^P,\Delta_0^Q)$ where 
$\Delta_0^A\in \mathrm{Diff}_1(A,R_0),\Delta_0^P\in \mathrm{Diff}_1(P,R_1)$ and $\Delta_0^Q\in \mathrm{Diff}_1(Q,R_2),$ satisfying
\begin{enumerate}[label=$\bullet$]
\item $
\lambda_0\big(p_0,\delta_{a}(\Delta_0^A)\big)=\delta_a(\Delta_0^P)(p_0),$
\item 
$\nu\big(q_0,\delta_{a}(\Delta_0^A)\big)=\delta_a(\Delta_0^Q)(q_0),$
\item  $\lambda_1\big(p_0,\delta_{a}(\Delta_0^P)(p_1)\big)=\big(\delta_a\Delta_0^Q\big)g(p_0,p_1),$
\item  $\Delta_0^Q\big(g(p_0,p_1)\big)=\lambda_1\big(p_1,\Delta_0^P(p_0)\big)+\lambda_1\big(p_0,\Delta_0^P(p_1)\big)+\nu\big(g(p_0,p_1),\Delta_0^A(1_A)\big),$ which holds in $R_2.$
\end{enumerate}

\begin{lemma}
$\Delta_0$ in  $\mathrm{Diff}_k(\EuScript{T},\EuScript{R})_0$ for any $k>0,$ satisfy the relations 
\begin{enumerate}
    \item $\lambda_0\big(p_0,\delta_a^k(\Delta_0^A)\big)=\delta_a^k\big(\Delta_0^P\big)(p_0),$
    \item $\nu\big(q_0,\delta_a^k(\Delta_0^A)\big)=\delta_a^k(\Delta_0^Q)(q_0),$
    \item $\lambda_1\big(p_0,\delta_a^k(\Delta_0^P)(p_1)\big)=\delta_a^k(\Delta_0^Q)(g(p_0,p_1)),$
    \item $\delta_a^{k-1}\Delta_0^Q\big(g(p_0,p_1)\big)=\lambda_1\big(p_1,\delta_a^{k-1}\Delta_0^P(p_0)\big)+\lambda_1\big(p_0,\delta_a^{k-1}\Delta_0^P(p_1)\big)+\nu\big(g(p_0,p_1),\delta_a^{k-1}(\Delta_0^A)(1_A)\big).$
\end{enumerate}
\end{lemma}

The remaining components are easily described.

\begin{lemma}
$\mathrm{Diff}_1(\EuScript{T},\EuScript{R})_1$ consists of $\Box_1=(\Box_1^{A},\Box_1^P)$ where $\Box_1^A\in \mathrm{Diff}_1(A,R_1),$ while $\Box_1^P\in \mathrm{Diff}_1(P,R_2),$ and the relation 
$\lambda_1^{\sharp}(p)\circ \delta_a(\Box_1^{A})=-\delta_a\Box_1^P\circ 1_p,$
where $1_p$ is the multiplication by $p$ operator, holds. This is a relation between maps $A\rightarrow R_2$ We have also have an isomorphism of $A$-modules
 $\mathrm{Diff}_1(\EuScript{T},\EuScript{R})_2\cong \mathrm{Diff}_1(A,R_2).$
It is easily established that elements of $\mathrm{Diff}_k(\EuScript{T},\EuScript{R})_1,$ consist of pairs of $k$'th order operators such that 
$\lambda_1^{\sharp}(p)\circ \delta_a^k\Box_1^A=-\delta_a^k\Box_1^P\circ 1_p.$ Furthermore, we have an isomorphism of $A$-modules $\mathrm{Diff}_k(\EuScript{T},\EuScript{R})_2\cong \mathrm{Diff}_k(A,R_2).$
\end{lemma}
More generally we may compute the relevant relations between two truncated triolic modules. We do not use these results, but present them for for completeness.
\begin{lemma}
For truncated triolic modules $(\EuScript{R},\nu,\eta)$ and $(\EuScript{S},\tilde{\nu},\tilde{\eta}),$ we have
$\mathrm{Diff}_1(\EuScript{R},\EuScript{S})_0$ consisting of triples $\Box_0=(\Box_0^{R_0},\Box_0^{R_1},\Box_0^{R_2}\big)$ where $\Box_0^{R_i}\in \mathrm{Diff}_1(R_i,S_i)$ for each $i=0,1,2$ and the following relations hold
$$
\scalemath{.88}{
\begin{cases}
\tilde{\nu}_i^{\sharp}(p)\circ \delta_a\Box_0^{R_i}=\delta_a\Box_0^{R_{i+1}}\circ \nu_i^{\sharp}(p), i=0,1
\\
\tilde{\eta}_q^{\sharp}\circ \delta_a\Box_0^{R_0}=\delta_a\Box_0^{R_2}\circ \eta_q^{\sharp},
\\
\Box_0^{R_2}\circ \eta_{g(p_0,p_1)}^{\sharp}=\tilde{\nu}_1^{\sharp}(p_1)\circ \Box_0^{R_1}\circ \nu_0^{\sharp}(p_0)-\tilde{\nu}_1^{\sharp}(p_0)\circ \Box_0^{R_1}\circ \nu_0^{\sharp}(p_1)+\tilde{\eta}_{g(p_0,p_1)}^{\sharp}\circ \Box_0^{R_0}.
\end{cases}}
$$
\end{lemma}

Analogous relations for second order operators are given by
$\tilde{\nu}_i^{\sharp}(p)\circ \delta_a^2\Box_{0}^{R_i}=\delta_a^2\Box_0^{R_{i+1}}\circ \nu_i^{\sharp}(p), i=0,1$
for any $a,b\in A,p\in P,$ corresponding to $\delta_{a,b,p}(\Box_0)=0,$ with $a=b.$ Computing the relation $\delta_{a,p_0,p_1}(\Box)=0,$ gives
$\tilde{\eta}_{g(p_0,p_1)}^{\sharp}\circ \delta_a\Box_0^{R_0}+\delta_a\Box_0^{R_2}\circ \tilde{\eta}_{g(p_0,p_1)}^{\sharp}=\tilde{\nu}_1^{\sharp}(p_0)\circ \delta_a\Box_0^{R_1}\circ \nu_0^{\sharp}(p_1)-\tilde{\nu}_1^{\sharp}(p_1)\circ\delta_a\Box_0^{R_1}\circ\nu_0^{\sharp}(p_0).$
Proceeding in this way we find the general result for order $k$ operators:
$$
\scalemath{.88}{
\begin{cases}
\tilde{\nu}_i^{\sharp}(p)\circ \delta_a^k\Box_{0}^{R_i}=\delta_a^k\Box_0^{R_{i+1}}\circ \nu_i^{\sharp}(p), i=0,1
\\
\tilde{\eta}_{g(p_0,p_1)}^{\sharp}\circ \delta_a^{k-1}\Box_0^{R_0}+\delta_a^{k-1}\Box_0^{R_2}\circ \tilde{\eta}_{g(p_0,p_1)}^{\sharp}=\tilde{\nu}_1^{\sharp}(p_0)\circ \delta_a^{k-1}\Box_0^{R_1}\circ \nu_0^{\sharp}(p_1)-\tilde{\nu}_1^{\sharp}(p_1)\circ\delta_a^{k-1}\Box_0^{R_1}\circ\nu_0^{\sharp}(p_0).
\end{cases}}
$$

\begin{lemma}
$\mathrm{Diff}_k(\EuScript{R},\EuScript{S})_1$ consists of pairs of operators $\Box_1=(\Box_1^{R_0},\Box_1^{R_1})$ such that $\Box_1^{R_0}\in \mathrm{Diff}_k(R_0,S_1)$ and $\Box_1^{R_1}\in \mathrm{Diff}_k(R_1,S_2)$ with $\tilde{\nu}_1^{\sharp}(p)\circ \delta_a^k\Box_1^{R_0}=-\delta_a^k\Box_1^{R_1}\circ\nu_0^{\sharp}(p),$ for all $a\in A,p\in P.$
\end{lemma}
\section{Symbols of Triolic Differential Operators}
\label{sec:Aspects of Triolic Hamiltonian formalsim}
Using the canonical order filtration in the algebra of graded differential operators $\mathrm{Diff}(\EuScript{T})_{\mathcal{G}}$ the associated graded object is given for each $k\geq 0,$ by $\mathrm{Smbl}_{k}:=\mathrm{Diff}_k/\mathrm{Diff}_{k-1},$ with $\mathrm{Diff}_{-1}:=\emptyset,$ 
and the entire associated graded object
$\text{Smbl}(\EuScript{T})_{\mathcal{G}}:=\bigoplus_{k\geq 0}\text{Smbl}_{k}(\EuScript{T})_{\mathcal{G}}\equiv \bigoplus_{k\geq 0}\bigoplus_{g\in\mathcal{G}}\text{Smbl}_{k}(\EuScript{T})_g,$
a $\mathcal{G}\oplus\mathbb{Z}$-graded commutative associative algebra, is the algebra of \emph{graded symbols}.
For $\Delta_g \in \text{Diff}_{k}(\EuScript{T})_g,$ denote the image of the projection
$\text{smbl}_{k,g}:\text{Diff}_{k}(\EuScript{T})_g\rightarrow \frac{\text{Diff}_{k}(\EuScript{T})_g}{\text{Diff}_{k-1}(\EuScript{T})_g}=\text{Smbl}_{k}(\EuScript{T})_g,$
by
$\text{smbl}_{k,g}(\Delta_g):=\mathfrak{s}_{k,g}(\Delta_g).$

The graded commutative algebra structure is given by the map
\begin{equation}
\label{symbolalgebrastructure}
\star:\text{Smbl}_{k}(\EuScript{T})_i\times \text{Smbl}_{\ell}(\EuScript{T})_j\rightarrow \text{Smbl}_{k+\ell}(\EuScript{T})_{i+j}, \hspace{5mm} \text{smbl}_{k,i}(\Delta)\star\text{smbl}_{\ell,j}(\nabla):=\text{smbl}_{k+\ell,i+j}\big(\Delta\circ\nabla)
\end{equation}
for $\Delta\in \text{Diff}_{k}(\EuScript{T})_i,\nabla\in \text{Diff}_{\ell}(\EuScript{T})_j.$

\begin{lemma}
\label{symbolalg}
Let $\EuScript{T}$ be an arbitrary graded commutative algebra (not necessarily a triole algebra). Then the collection of graded symbols $\mathrm{Smbl}_{\bullet,*}(\EuScript{T})$ is a graded commutative algebra and a graded Lie algebra with respect to the bracket,
$\{\mathrm{smbl}_{k,i}(\Delta),\mathrm{smbl}_{\ell,j}(\nabla)\}:=\mathrm{smbl}_{k+\ell-1,i+j}\big([\Delta,\nabla]\big),$
where $\mathrm{smbl}_{k,i}(\Delta)\in\mathrm{Smbl}_{k,i}(\EuScript{T}),\mathrm{smbl}_{\ell,j}(\nabla)\in\mathrm{Smbl}_{\ell,j}(\EuScript{T}).$ Moreover, this bracket is a graded Poisson bracket.
\end{lemma}
Such a bracket gives rise to the notion of a Hamiltonian derivation. Namely, let $\Delta\in \text{Diff}_{k}(\EuScript{T})_i,$ be some graded differential operator of order $k$ and degree $i$ with symbol $\text{smbl}_{k,i}(\Delta):=\mathfrak{s}_{k,i}(\Delta).$ We may define a map in terms of the above graded Poisson bracket as
\begin{equation*}
H_{\mathfrak{s}_{k,i}(\Delta)}:\text{Smbl}_{\bullet,*}(\EuScript{T})\rightarrow \text{Smbl}_{\bullet,*}(\EuScript{T}),\hspace{5mm}
\mathfrak{s}_{m,j}(\nabla) \longmapsto \big\{\mathfrak{s}_{k,i}(\Delta),\mathfrak{s}_{m,j}(\nabla)\big\},
\end{equation*}
for $\nabla\in \text{Diff}_{m}(\EuScript{T})_j.$
\begin{lemma}
$H_{\mathfrak{s}_{k,i}(\Delta)}$ is a graded derivation of the algebra of graded symbols. 
\end{lemma}
 We often call $H_{\mathfrak{s}_{k,i}(\Delta)}$ the \emph{Hamiltonian vector field} corresponding to the \emph{Hamiltonian} $\mathfrak{s}_{k,i}(\Delta).$ Moreover, we have a map
$H:\text{Smbl}_{\bullet,*}(\EuScript{T})\rightarrow D\big(\text{Smbl}_{\bullet,*}(\EuScript{T})\big)_{\mathcal{G}},$ which sends a Hamiltonian to the corresponding Hamiltonian vector field $
\mathfrak{s}_{k,i}(\Delta)\mapsto H_{\mathfrak{s}_{k,i}(\Delta)}.$

There is a well known equivalence between (graded) symbols and (graded) symmetric tensor fields that we will exploit. For convenience, we now present this standard preliminary result which can be found in \cite{KraVerb}.
\begin{lemma}
\label{symsmbl}
Let $A$ be a smooth algebra and $P,Q$ be projective $A$-module. Then
 If $A$ is an algebra over the field of rational numbers, then we have for each $k\geq 0$, that  $\mathrm{Smbl}_k(P,Q)\cong S^k\big(D_1(A)\big)\otimes_A \mathrm{Hom}_A(P,Q).$
 Thus we interpret these as symmetric $k$-multiderivations of $A$ with values in $\mathrm{Hom}_A(P,Q).$ In particular, $\mathrm{Smbl}_*(A,P)\cong P\otimes_A \mathrm{Smbl}_*(A),$ while the algebra of symbols itself enjoys the following description
$\mathrm{Smbl}_*(A)\cong \mathrm{Sym}^*\big(\mathrm{Smbl}_1(A)\big)=\mathrm{Sym}^*\big(D_1(A)\big).$
 \end{lemma}
The characterization of symbols as symmeric contravariant tensor fields given in Lemma \ref{symsmbl} allows us to conclude that the Poisson bracket on principal symbols of Lemma \ref{symbolalg} is isomorphic to the Lie algebra product on symmetric contravariant tensors defined by the Schouten bracket.

This generalizes straightforward to the setting of graded geometry and we investigate now the analogous isomorphisms in the formalism of triolic differential calculus.
\subsubsection{A characterization of triolic symbols in degree zero} 
\label{sssec:DegreeZeroSymbols}

To understand degree zero triolic symbols as tensors, we should turn our attention to the quotient Atiyah sequences of the form $0\rightarrow \mathfrak{at}_{k-1}(\EuScript{T})\hookrightarrow \mathfrak{at}_{k}(\EuScript{T})\rightarrow \mathfrak{at}_{k}(\EuScript{T})/\mathfrak{at}_{k-1}(\EuScript{T})\rightarrow 0,$ which gives:

\begin{equation}
\label{QuotientAtiyahDiagram}
\adjustbox{scale=.92}{
\begin{tikzcd}
0 \arrow[r, ]&\mathcal{D}\mathrm{iff}_{k-2}\big(g;P,Q\big)\arrow[d, ""] \arrow[r, ""]
&\mathrm{Diff}_{k-1}(\EuScript{T})_0\arrow[d, "" ] \arrow[r, "\varsigma_{k-1}"]  &\mathrm{Diff}_{k-1}(A) \arrow[d, ""] \arrow[r, ""]&  0\\
0 \arrow[r, ]&\mathcal{D}\mathrm{iff}_{k-1}\big(g;P,Q\big) \arrow[d, ""] \arrow[r, ""]
&\mathrm{Diff}_{k}(\EuScript{T})_0\arrow[d, "" ] \arrow[r, "\varsigma_k"]  &\mathrm{Diff}_{k}(A)\arrow[d, ""] \arrow[r, ""]&  0 \\
0 \arrow[r, ""] & \mathcal{D}\mathrm{iff}_{k}\big(g;P,Q\big)/ \mathcal{D}\mathrm{iff}_{k-1}\big(g;P,Q\big)\arrow[r, ""] & \text{Smbl}_{k}(\EuScript{T})_0\arrow[r, ""] & \text{Smbl}_{k}(A)\arrow[r, ""] & 0, 
\end{tikzcd}}
\end{equation}
We need to describe the quotient spaces $\mathcal{D}\mathrm{iff}_{k}\big(g;P,Q\big)/ \mathcal{D}\mathrm{iff}_{k-1}\big(g;P,Q\big)$ and $\mathrm{Smbl}_k(\EuScript{T})_0$ arising in diagram (\ref{QuotientAtiyahDiagram}).

\begin{lemma}
\label{Characterizing Smbl_k(g;P,Q)}
Elements of the quotient modules $\mathcal{D}\mathrm{iff}_{k}\big(g;P,Q\big)/ \mathcal{D}\mathrm{iff}_{k-1}\big(g;P,Q\big)$ coincide with pairs of tensors $(\mathbb{X}_P,\mathbb{X}_Q)$ with $\mathbb{X}_P\in \mathrm{Sym}^{k-1}\big(D(A)\big)\otimes \mathrm{End}(P)$ and $\mathbb{X}_Q\in \mathrm{Sym}^{k-1}\big(D(A)\big)\otimes \mathrm{End}(Q)$ such that $\mathbb{X}_Q(a_1,\ldots,a_{k-1})\big(g(p_1,p_2)\big)=g\big(\mathbb{X}_P(a_1,\ldots,a_{k-1}(p_1),p_2\big)+g\big(p_1,\mathbb{X}_P(a_1,\ldots,a_{k-1})(p_2)\big).$
\end{lemma}

We will view pairs $\big(\mathbb{X}_P,\mathbb{X}_Q\big)$
as elements of $\mathrm{Sym}^{k-1}\big(D(A)\big)\otimes \mathcal{E}\mathrm{nd}\big(g;P,Q\big).$
We will denote the quotient spaces as
\begin{equation}
    \label{eqn: Degree zero symbol quotient}
\mathcal{S}\mathrm{mbl}_k\big(g;P,Q\big):=\frac{\mathcal{D}\mathrm{iff}_{k}\big(g;P,Q\big)}{ \mathcal{D}\mathrm{iff}_{k-1}\big(g;P,Q\big)}.
\end{equation}
By Lemma \ref{Characterizing Smbl_k(g;P,Q)}, elements of the generalized spaces of symbols (\ref{eqn: Degree zero symbol quotient}) may be suitably interpreted as tensors $\mathrm{Sym}^{k-1}\big(D(A)\big)\otimes \mathcal{E}\mathrm{nd}\big(g;P,Q\big).$

\begin{theorem}
\label{DegreeZeroTriolicSymbols}
There is an isomorphism of $A$-modules, $\mathrm{Smbl}_k(\EuScript{T})_0\cong \mathrm{Sym}^{k-1}\big(D(A)\big)\otimes \mathcal{D}\mathrm{er}(g;P,Q).$
\end{theorem}
\begin{proof}
We need to define a map $\lambda_0^k:\mathrm{Diff}_k(\EuScript{T})_0\rightarrow \mathrm{Sym}^{k-1}\big(D(A)\big)\otimes \mathcal{D}\mathrm{er}\big(g;P,Q\big)$ with the properties that it is well-defined when we pass to the quotient modules $\mathrm{Diff}_{k}(\EuScript{T})_0/\mathrm{Diff}_{k-1}(\EuScript{T})_0$ and that $\ker\big(\lambda_0^k\big)\cong \mathrm{Diff}_{k-1}(\EuScript{T})_0.$ 
To this end, we define $\lambda_0^k$ as follows. For each $\Box_0$ we define $\lambda_0^k(\Box_0)$ as the hieroglyph $\mathcal{P}_{\Box_0}$ with the following properties. 
It is an operator which acts on $(k-1)$-tuples of functions $f_1,\ldots,f_{k-1}\in A,$ as
\begin{equation}
\label{degreezerosymboldefinition}
    \mathcal{P}_{\Box_0}(f_1,\ldots,f_{k-1}):=\begin{cases}
    \mathcal{P}_{\Box_0}(f_1,\ldots,f_{k-1})|_P:=\delta_{f_1,\ldots,f_{k-1}}\Box_0^P \in \mathrm{Der}(P)\subset \mathrm{Diff}_1(P,P),
    \\
    \mathcal{P}_{\Box_0}(f_1,\ldots,f_{k-1})|_Q:=\delta_{f_1,\ldots,f_{k-1}}\Box_0^Q \in\mathrm{Der}(Q)\subset \mathrm{Diff}_1(Q,Q).
    \end{cases}
\end{equation}

As $\Box_0^P$ and $\Box_0^Q$ are differential operators of order $k$, the objects on the right-hand side of (\ref{degreezerosymboldefinition}) are necessarily of order $1$. Moreover, they satisfy
\begin{eqnarray*}
\mathcal{P}_{\Box_0}(f_1,\ldots,f_{k-1})|_P(fp)&=&\big(\delta_{f,f_1,\ldots,f_{k-1}}\Box_0^P\big)(p) +f\mathcal{P}_{\Box_0}(f_1,\ldots,f_{k-1})(p),
\\
\mathcal{P}_{\Box_0}(f_1,\ldots,f_{k-1})|_Q(fq)&=&\big(\delta_{f,f_1,\ldots,f_{k-1}}\Box_0^Q\big)(q) +f\mathcal{P}_{\Box_0}(f_1,\ldots,f_{k-1})|_Q(q),
\end{eqnarray*}
as they are Der-operators. We see that these Der-operators are indeed those arising from the description of elements of $\mathcal{D}\mathrm{er}(g;P,Q)$, since the operators $\mathcal{P}_{\Box_0}(f_1,\ldots,f_{k-1})|_P$ and $\mathcal{P}_{\Box_0}(f_1,\ldots,f_{k-1})|_Q$ are seen to have the same scalar type symbol, say $\sigma\big(\mathcal{P}_{\Box_0}(f_1,\ldots,f_{k-1})|_P\big)$ and $\sigma\big(\mathcal{P}_{\Box_0}(f_1,\ldots,f_{k-1})|_Q\big)$. The latter claim follows since $\Box_0$ is a triolic differential operator, and so we have, in particular $\delta_a^k\Box_0^A=\delta_a^k\Box_0^P$ and $\delta_a^k\Box_0^A=\delta_a^k\Box_0^Q.$
Consequently, we see that $(\delta_{f,f_1,\ldots,f_{k-1}}\Box_0^P)(p)$ coincides with $p\delta_{f,f_1,\ldots,f_{k-1}}\Box_0^A,$ and similarly for the other Der-operator.
Consequently, the Leibniz rules given above, indeed defined an element $\mathcal{P}_{\Box_0}=(\mathcal{P}_{\Box^P},\mathcal{P}_{\Box^Q}\big)\in\mathcal{D}\mathrm{er}(g;P,Q).$

To highlight that these satisfy the compatibility relation with the metric $g$, let us consider the $k=2$ case and restrict to an action on elements of the form $q=g(p_1,p_2),$.
In this case, we find the relation (c.f with Definition \ref{Pairofconnectionspreservingg}) 
\begin{eqnarray*}
\mathcal{P}_{\Box^Q}(f_1)\big(bg(p_1,p_2)\big)&=&\big(\delta_b\circ\delta_{f_1}\Box^Q\big)g(p_1,p_2)+b\mathcal{P}_{\Box^Q}(f_1)g(p_1,p_2)
\\
&=&\big(\delta_b\circ\delta_{f_1}\Box^Q\big)g(p_1,p_2)+bg\big(\mathcal{P}_{\Box^P}(f_1)(p_1),p_2\big)+g\big(p_1,\mathcal{P}_{\Box^P}(f_1)(p_2)\big).\end{eqnarray*}
One may readily deduce for arbitrary $k>2$ that analogous relations hold and so the element $\mathcal{P}_{\Box_0}$ indeed defines a $(k-1)$-multiderivations of $A$ with values in the $A$-module $\mathcal{D}\mathrm{er}\big(g;P,Q\big).$
\end{proof}

\subsubsection{A characterization of triolic symbols in degree one} 
\label{sssec:DegreeOneSymbols}
We continue our analysis with the next non-trivial grading.
\begin{lemma}
\label{Degree1TriolicSymbols}
The $A$-module of degree $1$ symbols of order $k$ are identified with skew-symmetric $(k-1)$-derivations of $A$, with values in the $A$-module $\mathrm{Der}^g(P,Q)\subset \mathrm{Diff}_1(P,Q).$
\end{lemma}
Lemma \ref{Degree1TriolicSymbols} can be seen to be true by similar arguments to those of the proof of Theorem \ref{DegreeZeroTriolicSymbols}. For instance, consider the situation when $k=2$ and consider  $\Delta_1=(\Delta_1^A,\Delta_1^P)\in \mathrm{Diff}_2(\mathscr{T})_1.$
To each $\Delta_1$ we associate some heiroglyph $\mathcal{P}_{\Delta_1}$ and require this to 
act on elements of $a\in A$ by
$\mathcal{P}_{\Delta_1}(a):=\delta_a\Delta_1^P.$
This enjoys the relation $\mathcal{P}_{\Delta^1}(a)(bp)-b\mathcal{P}_{\Delta^1}(a)(p)=\delta_{b,a}\Delta_1^P,$ as it is an element of $\mathrm{Der}^g(P,Q)\subset \mathrm{Diff}_1(P,Q).$ 
Note that we may express $\delta_{b,a}\Delta_1^P\in \mathrm{Hom}_A(P,Q)$ equivalently as $\mathfrak{g}\circ \delta_{b,a}\Delta_1^A$ by Lemma \ref{Degree1TriolicDiffops}.
Suppose now that $\Delta_1\in \mathrm{Diff}_1(\EuScript{T})_1\subset \mathrm{Diff}_2(\EuScript{T})_1.$ Then $\mathcal{P}_{\Delta_1}(a):=\delta_a\Delta_1^P$ is identically zero, by definition of being a first order differential operator.
Consequently, the assignment $\Delta_1\mapsto \mathcal{P}_{\Delta_1}$ is not only well-defined on $\mathrm{Diff}_2(\EuScript{T})_1$ but also on the quotients $\mathrm{Diff}_2(\EuScript{T})_1/\mathrm{Diff}_1(\EuScript{T})_1;$ which is to say, on symbols.
We formulate this more precisely by denoting this assignment by
$\lambda_1^2:\mathrm{Diff}_2(\EuScript{T})_1\rightarrow \mathcal{S}^2(\EuScript{T})_1,$ where we have yet to specify exactly what the target is.

To get a better feeling of the nature of $\mathcal{S}^2(\EuScript{T})_1$ we contemplate higher orders $k$. In doing so, we get that
$\lambda_1^k:\mathrm{Diff}_k(\EuScript{T})_1\rightarrow  \mathcal{S}^k(\EuScript{T})_1, \hspace{2mm} \Delta_1\mapsto \mathcal{P}_{\Delta_1},$
where $\mathcal{P}_{\Delta_1}(a_1,a_2,\ldots,a_{k-1}):=\delta_{a_1,a_2,\ldots,a_{k-1}}\Delta_1^P\in\mathrm{Der}_g(P,Q),$ and so it enjoys the relation
$$\mathcal{P}_{\Delta_1}(a_1,\ldots,a_{k-1})(bp)=\delta_{b,a_1,\ldots,a_{k-1}}\Delta_1^P p +f \mathcal{P}_{\Delta_1}(a_1\ldots,a_{k-1})(p).$$
As an element of $\mathrm{Der}^g(P,Q)$ we would see that the symbol is a generalized twisted symbol. That is, we have $\delta_{b,a_1,\ldots,a_{k-1}}\Delta_1^P$ to be an covariant operator acting from $A\rightarrow Q.$
Since $\Delta_1$ is a triolic differential operator of order $k$, we know that $\mathfrak{g}\circ \delta_{a_1,\ldots,a_{k}}(\Delta_1^A)=\delta_{a_1,\ldots,a_k}\Delta_1^P,$ and so we may rewrite this relation as
\begin{eqnarray*}
\mathcal{P}_{\Delta_1}\big(a_1,\ldots,a_{k-1}\big)(bp)&=&\big(\delta_{b,a_1,\ldots,a_{k-1}}\Delta_1^P\big)p+b\mathcal{P}_{\Delta_1}\big(a_1,\ldots,a_{k-1}\big)(p)
\\
&=&\big(\mathfrak{g}\circ \delta_{b,a_1,\ldots,a_{k-1}}\Delta_1^A\big)p +b\mathcal{P}_{\Delta_1}\big(a_1,\ldots,a_{k-1}\big)(p).
\end{eqnarray*}
It remains only to demonstrate that $\ker(\lambda_1^k)$ coincides with $\mathrm{Diff}_{k-1}(\EuScript{T})_1.$ This is straightforward, and so we have validated our claim \ref{Degree1TriolicSymbols}. Consequently, by Lemma \ref{Degree1TriolicDerivations} and the corresponding sequence (\ref{eqn:TrioleAtiSeq11}), we can be more explicit in our characterization.
\begin{corollary}
For a regular triole algebra $\EuScript{T}$ there is an $A$-module isomorphism $\mathrm{Smbl}_k(\EuScript{T})_1$ with $\mathrm{Sym}^{k-1}\big(D(A)\big)\otimes_A \big(D\big(A,\mathrm{Hom}(P,Q)\big) \oplus \mathrm{Hom}(P,Q)\big).$
\end{corollary}

\subsubsection{A characterization of triolic symbols in degree two} 
\label{sssec:DegreeTwoSymbols}
The description of symbols is immediate from the general discussion at the beginning of §\ref{sec:Aspects of Triolic Hamiltonian formalsim}.
\begin{lemma}
\label{Degree2TriolicDiffops}
For each $k\geq 0$ the $A$-module of symbols of the algebra of trioles in degree $2$, $\mathrm{Smbl}_k(\EuScript{T})_2=\mathrm{Diff}_k(\EuScript{T})_2/ \mathrm{Diff}_{k-1}(\EuScript{T})_2$ is isomorphic as an $A$-module to $\mathrm{Smbl}_k(A,Q)$. In particular, there is an $A$-module isomorphism $\mathrm{Smbl}_k(A,Q)\cong \mathrm{Sym}^k\big(D(A)\big)\otimes_A Q.$
\end{lemma}

For $\Delta$ a degree $2$ triolic differential operator with symbol $\text{smbl}_k(\Delta),$ we have the above isomorphism provided by 
$\gamma_k^P\big(\text{smbl}_k(\Delta)\big)(da_1\cdot..\cdot da_k):=\big(\delta_{a_1}\circ ...\circ\delta_{a_k}\big)(\Delta),$
where $d a_i\cdot da_j$ denotes the multiplication in $\text{Sym}^k(\Lambda^1),$ we have an isomorphism of $A$-modules  
$\mu_Q:Q\otimes_A \text{Sym}^k\big(D(A)\big)\rightarrow \text{Smbl}_k(A,Q),q\otimes X_1\odot X_2\odot..\odot X_k\longmapsto \frac{1}{k!}\text{smbl}_k\big(1_Q\circ X_1\circ X_2\circ...\circ X_k\big),$
where $1_Q\in \text{Hom}_A(A,Q)$ is the degree zero operator of multiplication $1_q(a):=aq,$
is such that $\gamma_Q\circ \mu_Q=id.$

\section*{Conclusion and outlook}
In this work we have presented a conceptual formalism for studying two vector bundles `in interaction' by means of a vector-valued fiber metric. This information is neatly contained in the definition of a triole algebra and we have demonstrated that the basic functors of algebraic calculus over such an algebra generate a differential calculus which respects this inner structure. We have shown that in this formalism, one readily finds natural generalizations of the notion of Der-operator and of Atiyah sequences for vector bundles.

The next natural step, that is a work in progress, is to study the representative objects, for instance triolic differential forms and the resulting de Rham cohomology theory. Such cohomologies contain important information related to invariants of vector bundles that are supplied with inner structure \cite{Triole2}. In particular, they arise from complexes that are natural generalizations of the Der-complexes studied originally in \cite{Rub01} and in another manner in \cite{Diole1}.

\bibliographystyle{alpha}
\bibliography{Bibliography.bib}

\end{document}